\documentclass[letterpaper,11pt]{amsart}
\usepackage{amsmath,amssymb,amsthm,pinlabel,tikz,hyperref,mathrsfs,color, thmtools}

\usepackage{verbatim}
\usepackage{float}
\usepackage{caption}
\usepackage{subcaption}
\usepackage{enumitem}
\newcommand{\nc}{\newcommand}
\nc{\dmo}{\DeclareMathOperator}
\dmo{\ra}{\rightarrow}
\dmo{\Prob}{\mathbb{P}}
\dmo{\E}{\mathbb{E}}
\dmo{\N}{\mathbb{N}}
\dmo{\Z}{\mathbb{Z}}
\dmo{\Q}{\mathbb{Q}}
\dmo{\R}{\mathbb{R}}
\dmo{\C}{\mathcal{C}}
\dmo{\X}{\mathcal{X}}
\dmo{\U}{\mathcal{U}}
\dmo{\T}{\mathcal{T}}
\dmo{\F}{\mathcal{F}}
\dmo{\AC}{\mathcal{AC}}
\dmo{\w}{\omega}
\dmo{\MIN}{\mathcal{MIN}}
\dmo{\Mod}{Mod}
\dmo{\PMod}{PMod}
\dmo{\PMF}{\mathcal{PMF}}
\dmo{\Mat}{Mat}
\dmo{\supp}{supp}
\dmo{\UE}{\mathcal{UE}}
\dmo{\vol}{vol}
\dmo{\B}{B}
\dmo{\PB}{PB}
\dmo{\PR}{PSL(2,\mathbb{R})}
\dmo{\GL}{GL(k, \mathbb{C})}
\dmo{\SL}{SL(2, \mathbb{Z})}
\dmo{\Isom}{Isom}
\dmo{\RP}{\mathbb{R} \mathrm{P}}
\dmo{\I}{\mathcal{I}}
\dmo{\el}{\ell_{\C}}
\dmo{\NN}{\mathcal{N}}
\dmo{\rk}{rank}
\dmo{\tr}{tr}
\dmo{\llangle}{\langle\langle}
\dmo{\rrangle}{\rangle\rangle}
\dmo{\Unif}{Unif}
\dmo{\Out}{Out}
\dmo{\sumRho}{\mathcal{N}}
\dmo{\stopping}{\vartheta}

\usetikzlibrary{decorations.markings, patterns}
\tikzset{->-/.style={decoration={
  markings,
  mark=at position #1 with {\arrow{>}}},postaction={decorate}}}

\nc{\nt}{\newtheorem}

\nt{theorem}{Theorem}

\newtheorem{thm}{{\bf Theorem}}[section]
\newtheorem{lem}[thm]{{\bf Lemma}}
\newtheorem{cor}[thm]{{\bf Corollary}}
\newtheorem{prop}[thm]{{\bf Proposition}}
\newtheorem{fact}[thm]{Fact}

\newtheorem{claim}[thm]{Claim} 
\newtheorem{remark}[thm]{Remark}

\newtheorem{definition}[thm]{Definition}
\numberwithin{equation}{section}
\newtheorem{obs}[thm]{Observation}

\title[Central limit theorem and Geodesic tracking]{Central limit theorem and geodesic tracking on hyperbolic spaces and Teichm{\"u}ller spaces}

\date{\today}
\author{Inhyeok Choi}

\address{%
		June E Huh Center for Mathematical Challenges, KIAS\\
		85 Hoegiro Dongdaemun-gu, Seoul, 02455, South Korea 
		}
\address{
		Department of Mathematical Sciences, KAIST\\
		291 Daehak-ro Yuseong-gu, Daejeon, 34141, South Korea 
}
\email{%
        inhyeokchoi48@gmail.com
        }

\begin{document}
\begin{abstract}
We study random walks on the isometry group of a Gromov hyperbolic space or Teichm{\"u}ller space. We prove that the translation lengths of random isometries satisfy a central limit theorem if and only if the random walk has finite second moment. While doing this, we recover the central limit theorem of Benoist and Quint for the displacement of a reference point and establish its converse. Also discussed are the corresponding laws of the iterated logarithm. Finally, we prove sublinear geodesic tracking by random walks with finite $(1/2)$-th moment and logarithmic tracking by random walks with finite exponential moment.


\noindent{\bf MSC classes:} 20F67, 30F60, 57M60, 60G50
\end{abstract}

\maketitle

%%%%%%%%%%%%%%%%%%%%%%%%%%%%%%%%%%%%%%%%%%%%%%%%
%
%							Introduction
%
%%%%%%%%%%%%%%%%%%%%%%%%%%%%%%%%%%%%%%%%%%%%%%%%

\section{Introduction}	\label{sec:introduction}

Throughout, $(X, d)$ either denotes a Gromov hyperbolic space, without any assumption on properness, separability or geodesicity, or Teichm{\"u}ller space of a finite-type hyperbolic surface. We fix a reference point $o \in X$. All measures considered are probability measures. $\mu$ always denotes a \emph{non-elementary} discrete measure on the isometry group $G$ of $X$ and $\w = (\w_{n})_{n=1}^{\infty}$ denotes the random walk generated by $\mu$ (see Subsection \ref{subsection:random} for details).

Each $g \in G$ is associated with two dynamical quantities, the displacement $d(o, g o)$ of $o$ and the translation length $\tau(g) := \lim_{n} \frac{1}{n} d(o, g^{n}o)$. Displacements are analogous to the sums of random variables that arise in classical random walks, while translation lengths do not have corresponding notions in the Euclidean setting. We note that the translation lengths of mapping classes have an analogy with the eigenvalues of matrices, both of which have been studied with random walks (\cite{thurston1988classification}, \cite{guivarch1990produits}, \cite{karlsson2014spectral}).

Since displacements are subadditive, Kingman's subadditive ergodic theorem implies a law of large numbers for $d(o, \w_{n}o)$ when $\mu$ has finite first moment. More precisely, there exists a strictly positive constant $\lambda$ (called the escape rate of $\mu$) such that $\frac{1}{n}d(o, \w_{n}o) \rightarrow \lambda$ in $L^{1}$ and almost surely. For Gromov hyperbolic spaces, Gou{\"e}zel recently proved in \cite{gouezel2022exponential} that finite first moment of $\mu$ is necessary: if $\mu$ has infinite first moment, then the random walk escapes faster than any finite rate. 

Given stronger moment conditions, one can also discuss \emph{central limit theorems} (CLT for short) that deal with the limit law for $\frac{1}{\sqrt{n}}(d(o, \w_{n}o) - n \lambda)$. In \cite{benoist2016central}, Benoist and Quint proved the CLT for proper, quasiconvex Gromov hyperbolic spaces under finite second moment condition. See \cite{bjorklund2010central}, \cite{gouezel2017analyticity} for different approaches in this context. Benoist and Quint's strategy was generalized by Horbez in \cite{horbez2018clt}, proving a similar CLT when $X$ is Teichm{\"u}ller space. Another approach was proposed by Mathieu and Sisto in \cite{mathieu2020deviation}, which does not rely on the boundary structure of $X$.

In the Euclidean setting, one can further discuss \emph{laws of the iterated logarithm} (LIL for short) that contrast the almost sure and in probability asymptotics of the random walk in the order of $\sqrt{n \log \log n}$. So far, this has not been studied on Gromov hyperbolic spaces and Teichm{\"u}ller spaces.

Meanwhile, investigating translation lengths is more difficult since they are not subadditive. On Gromov hyperbolic spaces, Maher and Tiozzo obtained in \cite{maher2018random} the linear growth of translation lengths in probability. They also discuss an exponential decay of error event when $\mu$ has bounded support, which promotes the growth in probability to the almost sure growth. Dahmani and Horbez generalized this idea in \cite{dahmani2018spectral}, proving a spectral theorem for random isometries of Teichm{\"u}ller space. Baik, Kim and the author proved the same result in \cite{baik2021linear} with a weaker moment condition, assuming that $\mu$ has finite first moment. 

To the best of the author's knowledge, CLTs for translation lengths have been discussed only in finitely supported settings. For example, fixing a finite generating set of the group, \cite{gekhtman2019a-central} and \cite{gekhtman2020counting} deal with the limit law of translation lengths, counted with respect to the word metric. We note a recent relevant result by Aoun in \cite{aoun2021clt}, deducing a CLT for the eigenvalues of random matrices from the CLT for the matrix norms.

Our first goal is to obtain a finer description of the translation lengths of random walks. We present one result under finite first moment condition.

\begin{theorem}[Logarithmic deviation]\label{thm:logCorr}
Suppose that $\mu$ has finite first moment. Then there exists a constant $K < \infty$ such that \begin{equation}
\limsup_{n} \frac{1}{\log n}| \tau(\w_{n}) - d(o, \w_{n} o)| < K
\end{equation}
for almost every $\w$.
\end{theorem}

Assuming the CLT for displacements, Theorem \ref{thm:logCorr} implies a CLT for translation lengths.

\begin{theorem}[Central limit theorems and Laws of the Iterated Logarithm]\label{thm:central}
Suppose that $\mu$ is non-arithmetic with finite second moment. Then there exists a Gaussian law with variance $\sigma^{2} > 0$ to which $\frac{1}{\sqrt{n}}(d(o, \w_{n}o) - n \lambda)$ and $\frac{1}{\sqrt{n}}(\tau(\w_{n}) - n \lambda)$ converge in law. Here, we have \[
\limsup_{n \rightarrow \infty} \left| \frac{d(o, \w_{n}o) - \lambda n}{\sqrt{2n \log \log n}} \right|= \limsup_{n \rightarrow \infty}  \left|\frac{\tau(\w_{n}) - \lambda n}{\sqrt{2n \log \log n}}\right| =\sigma \quad \textrm{almost surely}.
\]
\end{theorem}

 Differently from Benoist and Quint's theory, our approach to Theorem \ref{thm:central} does not rely on martingales. Instead, it is based on the addition of i.i.d. random variables \emph{with defects} in \cite{mathieu2020deviation}. This theory deduces CLT from the uniform deviation inequalities based on a purely probabilistic and does not depend on the geometric properties of $X$ or $G$. Meanwhile, the uniform deviation inequalities that control the defects come from the non-positively curved geometry of the ambient space. We obtain this control by pivoting at pivotal times, combining the ideas of \cite{gouezel2022exponential} and \cite{baik2021linear}. Since this relies only on the Gromov inequalities among points, we do not require that $X$ be proper, quasi-convex, or separable. We also do not assume that the action of $G$ is acylindrical or that $\mu$ has finite exponential moment.

Our approach also deduces the converse of CLTs.

\begin{theorem}[Converse of Central limit theorems]\label{thm:centralConv}
Suppose that $\mu$ has infinite second moment. Then for any sequence $(c_{n})_{n}$, both $\frac{1}{\sqrt{n}}(d(o, \w_{n}o) - c_{n})$ and $\frac{1}{\sqrt{n}}(\tau(\w_{n}) - c_{n})$ do not converge in law.
\end{theorem}

A subtler problem related to Theorem \ref{thm:logCorr} is the geodesic tracking (or ray approximation) considered by Kaimanovich \cite{kaimanovich2000hyp}, Duchin \cite{duchin2005thin} and Tiozzo \cite{tiozzo2015sublinear}. They proved that random walks with finite first moment stay close to geodesics in a sublinear manner.

With a stronger moment condition, tighter geodesic tracking may occur. We note that logarithmic tracking has been observed in the following cases: \begin{enumerate}
\item random walks on free groups with finite exponential moment \cite{ledrappier2001free},
\item symmetric random walks on Gromov hyperbolic spaces with finite support \cite{blachere2011harmonic},
\item simple random walks on relatively hyperbolic spaces \cite{sisto2017tracking}, and
\item random walks on Gromov hyperbolic spaces with finite support \cite{maher2018random}.
\end{enumerate}

We now present a generalization of the above results.

\begin{theorem}[Geodesic tracking]\label{thm:tracking}
Let $X$ be geodesic.
\begin{enumerate}
\item Suppose that $\mu$ has finite $p$-th moment for some $p> 0$. Then for almost every $\w_{n}$, there exists a quasi-geodesic $\gamma$ such that \[
\lim_{n} \frac{1}{n^{1/2p}} d(\w_{n} o, \gamma) = 0.
\]
\item Suppose that $\mu$ has finite exponential moment. Then there exists $K' < \infty$ satisfying the following: for almost every $\w_{n}$, there exists a quasi-geodesic $\gamma$ such that \[
\limsup_{n} \frac{1}{\log n} d(\w_{n} o, \gamma) < K'.
\]
\item The quasi-geodesic $\gamma$ in (1) and (2) can be taken as a $D$-quasi-geodesic, where $D$ only depends on the nature of $X$ and not on the measure $\mu$. If $X$ is proper in addition, then $\gamma$ can be taken as a geodesic.
\end{enumerate}
\end{theorem}

In particular, sublinear tracking occurs when $\mu$ has finite $(1/2)$-th moment. To the best of the author's knowledge, this is new even for random walks on free groups. We note a relevant observation in \cite{mathieu2020deviation}.

\begin{remark}
As in \cite{gouezel2022exponential}, we restrict the situation to discrete measures for the sake of simplicity. The author believes that the same proof will work for Borel measures by carefully choosing a Schottky random variable instead of a Schottky set with the uniform measure.

Similarly, the conclusions (1), (2) of Theorem \ref{thm:tracking} can be deduced when $X$ is only assumed to be intrinsic.
\end{remark}

Our main philosophy stems from the pivot construction in \cite{gouezel2022exponential} and pivoting in \cite{baik2021linear}. Let us explain these concepts in broad strokes. Given a sample path $ \w = (\w_{n})_{n}$, we look for steps $\mathcal{N} = \{n_{1} < \ldots < n_{k}\} \subseteq \{1, \ldots, n\}$ such that the progresses made at step $n_{i}$'s are `persistent'. That means, the paths $\gamma_{n_{i}} = [\w_{n_{i}-1} o, \w_{n_{i}} o]$ are aligned along $[o, \w_{n} o]$ in a way that $\gamma_{n_{i}}$ appear earlier than $\gamma_{n_{i+1}}$. As a result, we have $d(o, \gamma_{n_{i+1}}) \ge d(o, \gamma_{n_{i}}) + L$ and $d(o, \w_{n} o) \ge Lk$ for some $L>0$.

What we hope is the linear growth of $k = \#\mathcal{N}$ in probability. Here is a naive approach to this problem. Let $E_{i}$ be the event where $\gamma_{i}$ is well-located, i.e., $\gamma_{i}$ appears in the middle of the geodesic $[o, \w_{n} o]$. Assuming that the random walk involves two independent `loxodromic' directions, it is not hard to realize that $\Prob(E_{i}) \ge \eta$ for some positive $\eta$ that does not depend on $n$ and $i$. If $E_{i}$'s were independent, this will imply that $\sum_{i} \chi_{E_{i}}$ increases linearly in probability; unfortunately, they are not. Moreover, even if we have $\w \in E_{i} \cap E_{j}$ for some $i<j$, we are not sure whether $\gamma_{i}$ is on the left side of $\gamma_{j}$ or not. Hence, we should come up with a better data that: \begin{enumerate}
\item record the relative locations among $\gamma_{i}$'s, not only between $o$, $\w_{n} o$ and each $\gamma_{i}$;
\item realize the independence among events, which leads to the linear growth of the data in probability.
\end{enumerate}

Gou{\"e}zel's construction of pivotal times in \cite{gouezel2022exponential} accomplishes this job. Gou{\"e}zel constructed events $E_{i}$'s that depend on the steps $(g_{1}, \ldots, g_{i})$. When $\w \in E_{n_{1}} \cap \cdots \cap E_{n_{k}}$ for some $n_{1} < \ldots < n_{k}$, there exists a chain of paths \[
(o, \gamma_{1}', \ldots, \gamma_{j_{1}}' = \gamma_{n_{1}}, \gamma_{j_{1}+1}', \ldots, \gamma_{j_{2}}' = \gamma_{n_{2}}, \ldots, \gamma_{j_{k}}' = \gamma_{n_{k}}, \ldots, \w_{n} o)
\]
where each pair of consecutive paths $(\gamma_{j}', \gamma_{j+1}')$ are in good positions. Moreover, $\chi_{E_{i}}$'s behave like a martingale: when $\w \notin E_{i}$ at some $i$, one can change the directions of $\gamma_{j}$'s for the earlier $j$'s at which $\w \in E_{j}$ so that $\sum_{j=1}^{i-1} \chi_{E_{j}}$ does not change but the modified $\w'$ now belongs to $E_{i}$. This process is called \emph{pivoting}, using which one can bound $\chi_{j=1}^{i} E_{i}$ from below with a sum of i.i.d.s with exponential tail. An essential geometric ingredient for this is the so-called \emph{Schottky sets}, whose usage is motivated by the work of Boulanger, Mathieu, Sert and Sisto \cite{boulanger2022large}.

In \cite{gouezel2022exponential}, pivoting was used to guarantee the definite progress of the random walk and the deviation from a fixed direction. In \cite{baik2021linear}, Baik, Choi and Kim used the pivoting for another purpose, namely, to guarantee large translation lengths of random isometries. Roughly speaking, displacements and translation lengths almost match when there are sufficiently many pivots. While the abundance of pivotal times was deduced from the subadditive ergodic theorem in \cite{baik2021linear}, we instead unify the notions of pivotal time in \cite{gouezel2022exponential} and \cite{baik2021linear} and deduce a stronger result.

While unifying these notions, we also generalize Gou{\"e}zel's setting of Gromov hyperbolic spaces in \cite{gouezel2022exponential} and include Teichm{\"u}ller space. Gou{\"e}zel's construction of pivotal times  relies on the local alignment of pairs of consecutive geodesics, and the Gromov hyperbolicity promotes this local alignment into the global alignment. We bring the corresponding alignment lemmata for Teichm{\"u}ller space from \cite{baik2021linear}.

We first review preliminaries on Gromov hyperbolic spaces and Teichm{\"u}ller space in Section \ref{section:prelim}. We then establish the notions of witnessing (Section \ref{section:witness}), Schottky set (Subsection \ref{subsection:Schottky}) and pivotal times (Subsection \ref{subsection:pivot}). Although these notions were already introduced in \cite{gouezel2022exponential} and \cite{baik2021linear}, we re-formulate these notions to integrate the cases of (geodesic or non-geodesic) Gromov hyperbolic spaces and Teichm{\"u}ller space. When we construct pivotal times for random walks, it suffices that certain directions exhibit hyperbolicity (as opposed to the global hyperbolicity in Gromov hyperbolic spaces). In Teichm{\"u}ller space, Rafi's theory in \cite{rafi2014hyperbolicity} tells us that $\epsilon$-thick geodesics serve this role.

In Subsection \ref{subsection:1stModel}, we incorporate pivotal times into random walks and establish the prevalence of pivotal times outside an event with exponentially decaying probability. Subsection \ref{subsection:conseq} is concerned with its consequences. After proving Theorem \ref{thm:logCorr}, we establish the $2p$-moment bound of the distance at which two independent random trajectories deviate from each other (Proposition \ref{prop:dominantGrom2}). This exponent doubling was observed in \cite{mathieu2020deviation} when $p = 2$ and $G$ is a hyperbolic group acting on its Cayley graph. We express this result in both non-geodesic setting (using Gromov products) and geodesic setting (using the distance $d(o, [\check{\w}_{n} o, \w_{n} o])$); the latter one leads to Theorem \ref{thm:tracking}.

In Section \ref{section:central}, the deviation bound for $p=2$ is used to prove CLTs following the spirit of \cite{mathieu2020deviation}. We first establish a lower bound on the normalized variance of $d(o, \w_{n} o)$ from the non-arithmeticity of $\mu$. We then perform dyadic summation with independent defects as in \cite{mathieu2020deviation}. We further establish the converse of CLTs using the pivot construction (Theorem \ref{thm:centralConv}).

In Section \ref{section:lil}, we discuss the LIL and finish proving Theorem \ref{thm:central}. The basic strategy comes from \cite{deAcosta1983lil}. Here the difficulty is to deal with an infinite array of sums of i.i.d, which is circumvented by noting that the frequency of dyadic defects decreases exponentially.

\subsection*{Acknowledgments}
%We truly appreciate A, B, and C for fruitful conversations.
The author thanks Hyungryul Baik, S{\'e}bastien Gou{\"e}zel, Camille Horbez and Dongryul M. Kim for helpful discussion. Especially, the question of Gou{\"e}zel regarding the converse of CLTs has motivated Theorem \ref{thm:centralConv}. The author is also grateful to {\c C}a{\u g}r{\i} Sert for suggesting reference and explaining historical background related to the current paper. This paper presents part of the author's PhD thesis.

The author was supported by Samsung Science \& Technology Foundation grant No. SSTF-BA1702-01 and No. SSTF-BA1301-51.

%%%%%%%%%%%%%%%%%%%%%%%%%%%%%%%%%%%%%%%%%%%%%%%%
%
%					Preliminaries
%
%%%%%%%%%%%%%%%%%%%%%%%%%%%%%%%%%%%%%%%%%%%%%%%%

\section{Preliminaries}\label{section:prelim}

\subsection{Gromov hyperbolic spaces}

We recall basic definitions related to Gromov hyperbolic spaces. For more details, see \cite{ghys1990bord}, \cite{vaisala2005gromov} or \cite{bridson1999metric}.

\begin{definition}\label{dfn:Gromov}
Given a metric space $(M, d)$ and a triple $x, y, z \in M$, we define the \emph{Gromov product} of $y, z$ with respect to $x$ by \begin{equation}\label{eqn:GromovProduct}
(y, z)_{x} = \frac{1}{2} [d(x, y) + d(x, z) - d(y, z)].
\end{equation} 
$M$ is said to be \emph{$\delta$-hyperbolic} if every quadruple $x, y, z, w \in M$ satisfies the following inequality called the \emph{Gromov inequality}:
\begin{equation}\label{eqn:Gromov}
(x, y)_{w} \ge \min\{(x, z)_{w}, (y, z)_{w}\} - \delta.
\end{equation}
$X$ is said to be \emph{Gromov hyperbolic} if $X$ is $\delta$-hyperbolic for some $\delta>0$.
\end{definition}

We recall basic facts about Gromov products without proof.

\begin{fact}[{\cite[Lemma 2.8]{vaisala2005gromov}}]\label{fact:GromProdFact} 
For a metric space $M$ and $x, y, z, w \in M$,\[
\begin{aligned}
(y, y)_{x} &= 0,\\
(y, z)_{x} &= (z, y)_{x},\\
 d(x, y) &=(y, z)_{x} + (x, z)_{y},\\
0 \le (y, z)_{x} &\le d(x, y), \\
-d(x, w) \le (y, z)_{x} - (y, z)_{w} &= d(x, w) - (y, x)_{w} - (z, x)_{w} \le d(x, w).
\end{aligned}
\]
\end{fact}

Most of the arguments in this paper involve Gromov inequalities only. However, the geodesic tracking phenomenon refers to geodesics or quasi-geodesics among points. The following notions serve this purpose.

A \emph{geodesic segment} on a metric space $M$ is an isometric embedding $\gamma :[a,b] \rightarrow M$ of a closed interval $[a, b]$. The reverse $\bar{\gamma}$ of $\gamma$ refers to  the map $t \mapsto \gamma(a+b - t)$. By abusing notation, we also call the image $\gamma([a, b])$ of $\gamma$ a geodesic segment connecting $\gamma(a)$ and $\gamma(b)$. Nonetheless, geodesic segments are considered oriented and $\gamma$ and $\bar{\gamma}$ are to be distinguished. We also denote by $[a, b]$ an arbitrary geodesic that begins at $a$ and terminates at $b$.

We say that a geodesic segment $\gamma_{1} : [c, d] \rightarrow M$ is a \emph{subsegment} of $\gamma$ if $\gamma|_{[c, d]} = \gamma_{1}$. For subsegments $\gamma_{1} : [c, d] \rightarrow M$, $\gamma_{2} : [c', d'] \rightarrow M$ of $\gamma$, we say that $\gamma_{1}$ \emph{appears earlier than} $\gamma_{2}$ if $d < c'$. By abusing notation again, we also say that $[\gamma_{1}(c), \gamma_{1}(d)]$ is a subsegment of $[\gamma(a), \gamma(b)]$.

\begin{definition}\label{dfn:proper}
A metric space $M$ is said to be \emph{geodesic} if any pair of points of $M$ is connected by a geodesic. $M$ is said to be \emph{intrinsic} if for any $x, y \in M$ and $\epsilon > 0$, $x$ and $y$ are connected by an arc of length at most $d(x, y) + \epsilon$. $M$ is said to be \emph{proper} if bounded closed subsets are compact.
\end{definition}

We recall some more basic facts, again without proof.
\begin{fact}\label{fact:GromProdFact2}
Let $M$ be a geodesic space, $x, y, z, w \in M$ and $a \in [x, y]$, $b \in [x, z]$. Then we have \[\begin{aligned}
(y, z)_{x} &\ge d(w, x) - d(a, w) - d(b, w),\\
(y, z)_{x} &\le d(x, [y, z]).
\end{aligned}
\]
\end{fact}

In geodesic spaces, Gromov hyperbolicity can be interpreted in different aspects. Namely, one can require that geodesic triangles be slim, thin, or have small insizes. For us, the following facts will be needed.

\begin{fact}[{cf. \cite[Proposition III.H.1.17, III.H.1.22]{bridson1999metric}}]
Let $M$ be a $\delta$-hyperbolic geodesic space and $x, y, z, w \in M$. Then the following hold: \begin{enumerate}
\item for any $p \in [y, z]$, either $d(p, [x, y]) \le 6\delta$ or $d(p, [y, z]) \le 6\delta$;
\item if $d(x, z) \le C$ and $d(y, w) \le C$ for some $C>0$, then $[x, y]$ and $[z, w]$ are within Hausdorff distance $2C + 12\delta$.
\end{enumerate}
\end{fact}

\emph{From now on, throughout the paper, we fix $\delta > 0$.} 

In the rest of this subsection, we fix a $\delta$-hyperbolic space $X$. Recall that $G$ denotes the isometry group of $X$. 

\begin{definition}
For $g \in G$, the \emph{translation length} of $g$ is defined by \[
\tau(g) := \lim_{n \rightarrow \infty} \frac{1}{n} d(o, g^{n} o).
\]
\end{definition}

In order to discuss the dynamics of isometries on $X$, we define a canonical boundary of $X$ as follows.

\begin{definition}\label{dfn:boundaryGromov}
A sequence $(x_{n})_{n>0}$ in $X$ \emph{converges to infinity} if $(x_{n}, x_{m})_{o} \rightarrow \infty$ as $m, n \rightarrow \infty$. Two sequences $(x_{n})$, $(y_{n})$ are \emph{converging to the same infinity point} if $(x_{n}, y_{m})_{o} \rightarrow \infty$ as $m, n \rightarrow \infty$. 

The set of equivalence classes of sequences converging to the same point is called the \emph{Gromov boundary} of $X$, and is denoted by $\partial X$. We say that a sequence $(x_{n})$ in $X$ converges to $[(y_{n})] \in \partial X$ if $(x_{n})$ and $(y_{n})$ are converging to the same point.
\end{definition}

The action of $g \in G$ on $X$ induces an action $[(x_{n})] \mapsto [(gx_{n})]$ on $\partial X$.

\begin{definition}
If $g$ has bounded orbits in $X$, $g$ is said to be \emph{elliptic}. If $g$ is not elliptic and has a unique fixed point in $\partial X$, then $g$ is said to be \emph{parabolic}. If $g$ has exactly two fixed points in $\partial X$, one of which is an attractor and another one is a repeller, then $g$ is said to be \emph{loxodromic}.
\end{definition}

It is a fact that elliptic, parabolic, and loxodromic elements partition $G$. Moreover, loxodromic elements have positive translation lengths.

\begin{definition}\label{dfn:nonelementary}
Two loxodromic isometries $g$, $h$ on $X$ are said to be \emph{independent} if they have disjoint sets of fixed points.
\end{definition}

The following notation is designed to integrate the cases of geodesic and non-geodesic spaces. When $X$ is non-geodesic, the \emph{segment} $[x, y]$ on $X$ refers to an ordered pair $\gamma = (x, y)$ of points $x, y \in X$. When $X$ is geodesic, it refers to a geodesic segment $\gamma$ from $x$ to $y$. In either case, $x$ is called the \emph{initial point} and $y$ is called the \emph{terminal point}. Here the \emph{length} of $[x, y]$ is defined by $d(x, y)$. For segments $\gamma=[x, y]$, $\eta = [x, w]$ and a point $z$, we denote by $(\gamma, z)_{\ast}$ and $(\gamma, \eta)_{\ast}$ the quantities $(y, z)_{x}$ and $(y, w)_{x}$, respectively.

\subsection{Teichmuller space}\label{subsection:Teich}

In this subsection, $(X, d)$ denotes the Teichm{\"u}ller space $\T(\Sigma)$ of a closed orientable surface $\Sigma$ of genus at least 2 and $d$ denotes the Teichm{\"u}ller metric.

By Teichm{\"u}ller's theorem, $X$ is uniquely geodesic; i.e., any pair of points is connected via a unique geodesic segment. We refer the readers to \cite{imayoshi1992introduction}, \cite{hubbard2006teich} for the details on  \emph{Teichm{\"u}ller geodesics} and quadratic differentials.

It is known that $X$ is not Gromov hyperbolic \cite{masur1995teichmuller}, but the dynamics of its isometry group resembles that of hyperbolic spaces. The isometry group of $\T(\Sigma)$ is equal to the extended mapping class group $\Mod^{\pm}(\Sigma)$ of $\Sigma$, which contains the mapping class group $\Mod(\Sigma)$ as a subgroup with index 2 (\cite{royden1971automorphisms}, \cite{earle1974on-holomorphic}, \cite{earle1974on-holomorphic}). The Nielsen-Thurston classification asserts that mapping classes are either periodic, reducible or pseudo-Anosov, where pseudo-Anosov classes correspond to loxodromic isometries on hyperbolic spaces. We will also call pseudo-Anosov mapping classes loxodromic.

Thurston endowed $X$ with a natural boundary, the space $\PMF(\Sigma)$ of projective measured foliations on $\Sigma$ (cf. \cite{1979travaux}). As in the case of hyperbolic spaces, $\Mod(\Sigma)$ also acts on $\PMF(\Sigma)$ and pseudo-Anosov classes have two fixed points on $\PMF(\Sigma)$; using these fixed points, we define independent mapping classes as in Definition \ref{dfn:nonelementary}.

\begin{comment}
In detail, the fixed points of a pseudo-Anosov mapping class $[\phi$] on $\PMF$ are uniquely ergodic and correspond to the stable and unstable foliations of the pseudo-Anosov representative $\phi$ of $[\phi]$. Moreover, these endpoints are connected by a bi-infinite Teichm{\"u}ller geodesic, called the translation axis of $[\phi]$, that is invariant under $[\phi]$. The length of its quotient by $[\phi]$ equals the translation length of $[\phi]$ and twice the log of the stretch factor of $\phi$.
\end{comment}

We denote by $X_{\ge \epsilon}$ the $\epsilon$-thick part of $X$, the collection of surfaces whose shortest extremal lengths are at least $\epsilon$. By Kerckhoff's formula in \cite{kerckhoff1980asymptotic}, $x \in X_{\ge \epsilon}$ implies $y \in X_{\ge \epsilon'}$ for $\epsilon' = \epsilon e^{-2d(x, y)}$.

Let $\gamma : [0, L] \rightarrow X$ and $\gamma' : [0, L'] \rightarrow X$ be geodesics on $X$ parametrized by length. If $d(\gamma(kL), \gamma'(kL')) < \epsilon$ for all $0 \le k \le 1$, we say that $\gamma$ and $\gamma'$ \emph{$\epsilon$-fellow travel}. The following are immediate observations:

\begin{fact}
If $[x, y]$ and $[x', y']$ are $\epsilon$-fellow traveling and $[x', y']$ and $[w, z]$ are $\epsilon'$-fellow traveling, then $[x, y]$ and $[w, z]$ are $(\epsilon + \epsilon')$-fellow traveling.
\end{fact}

\begin{fact}
Suppose that $x, y, z$ are on a same geodesic. Then $[x, y]$ and $[x, z]$ are $d(y, z)$-fellow travelling.
\end{fact}

In contrast with those in hyperbolic spaces, geodesics in Teichm{\"u}ller space with pairwise near endpoints need not fellow travel. Nonetheless, the following theorems of Rafi guarantees fellow-traveling and thinness of triangles, given that some ingredients are $\epsilon$-thick.

%%%%

\begin{thm}[{\cite[Theorem 7.1]{rafi2014hyperbolicity}; see also \cite[Corollary 4.4]{baik2021linear}}]\label{thm:rafi1}
For each $C >0$, there exists a constant $\mathscr{B}(\epsilon, C)$ satisfying the following. For $x, y \in X_{\ge \epsilon}$ and $x', y' \in X$ such that \[
d(x, x') \le C \quad \textrm{and}\quad d(y, y') \le C,
\]
$[x, y]$ and $[x', y']$ $\mathscr{B}(\epsilon, C)$-fellow travel.
\end{thm}

\begin{thm}[{\cite[Theorem 8.1]{rafi2014hyperbolicity}}]\label{thm:rafi2}
There exist constants $\mathscr{C}(\epsilon)$ and $\mathscr{D}(\epsilon)$ such that the following holds. Let $x, y, z \in X$ and suppose that the geodesic $[x, y]$ contains a segment $\gamma \subseteq X_{\ge \epsilon}$ of length at least $\mathscr{C}(\epsilon)$. Then there exists a point $w\in\gamma$ such that \[
\min \left\{ d(w, [x, z]), d(w, [z,y]) \right\} < \mathscr{D}(\epsilon).
\]
\end{thm}

\begin{lem}\label{lem:GromProdWitness}
Let $[x, y]$ be an $\epsilon$-thick segment on $X$ and $z \in X$. Then for $M= \max \{ d(p, y) : p \in [x, y], d(p, [y, z]) \le \mathscr{D}(\epsilon)\}$, we have \[
(x, z)_{y} \le M + \mathscr{C}(\epsilon) + 2\mathscr{D}(\epsilon).
\]
\end{lem}

\begin{proof}
Let $p \in [x, y]$ be the point at which the maximum $M$ is achieved. Let $p' \in [y, z]$ be such that $d(p, p') \le \mathscr{D}(\epsilon)$. If $d(x, p) \le \mathscr{C}(\epsilon)$, then $(x, z)_{y} \le d(x, y) = d(x, p) + d(p, y) \le \mathscr{C}(\epsilon)+ M$ holds.

If not, we consider a subsegment $[q_{1}, q_{2}]$ of $[x, y]$ on the left of $p$ that is longer than $\mathscr{C}(\epsilon)$. Note that $[q_{1}, q_{2}]$ is $\epsilon$-thick and no point on $[q_{1}, q_{2}]$ is $\mathscr{D}(\epsilon)$-close to $[y, z]$. Hence, Theorem \ref{thm:rafi2} implies that there exist $q \in [q_{1}, q_{2}]$ and $q' \in [x, z]$ that are within distance $\mathscr{D}(\epsilon)$. Now we obtain \[\begin{aligned}
(x, z)_{y} &= \frac{1}{2}\left[ (d(x, q) + d(q, p) + M) + (d(y, p') + d(p', z)) - (d(x, q') + d(q', z))\right]\\
& \le  \frac{1}{2}\left[ \begin{aligned} (d(x, q) + d(q, p) + M) + (M + \mathscr{D}(\epsilon) + d(p', z)) \\ -(d(x, q) - \mathscr{D}(\epsilon) + d(p', z) - 2\mathscr{D}(\epsilon) - d(p, q))\end{aligned}\right]\\
& \le M + 2\mathscr{D}(\epsilon) + d(q_{1}, p).
\end{aligned}
\]
Taking infimum of $d(q_{1}, p)$, we deduce the conclusion.
\end{proof}

\subsection{Random walks}\label{subsection:random}

We first summarize basic notions for measures.
\begin{definition}
Let $\nu$ be a discrete measure on $G$. The \emph{support} of $\nu$, denoted by $\supp \nu$, refers to the set $\{g \in G : \nu(g) \neq 0\}$. 

The \emph{semigroup} $\llangle \supp \nu \rrangle$ generated by $\supp \nu$ refers to the set $\{g_{1} \cdots g_{n} : n \in \N, g_{i} \in \supp \nu\}$. 

We denote by $\nu^{n}$ the product measure of $n$ copies of $\nu$ on $G^{n}$, and denote by $\nu^{\ast n}$ the $n$-th convolution of $\nu$ on $G$. For a random variable $f=(f_{1}, \ldots, f_{n})$ on a product space $G^{n}$, we denote by $f^{\ast}$ its convolution $f_{1} \cdots f_{n}$.

We say that $\nu$ is \emph{non-elementary} if $\llangle \supp \nu \rrangle$ contains two independent loxodromic elements, and $\nu$ is \emph{non-arithmetic} if there exists $N>0$ such that $\supp \nu^{\ast N}$ contains two elements with distinct translation lengths.

We define the $p$-th moment of $\nu$ by \[
\E_{\nu} [d(o, go)^{p}] = \int d(o, go)^{p} \, d\nu(g).
\]
We also define the exponential moment (with parameter $K>0$) of $\nu$ by \[
\E_{\nu} [e^{K d(o, go)}] = \int e^{K d(o, go)} \, d\nu(g).
\]
\end{definition}

The random walk on $G$ generated by $\mu$ is constructed as follows. We consider the \emph{step space} $(G^{\Z}, \mu^{\Z})$, the product space of $G$ equipped with the product measure of $\mu$. Each step path $(g_{n})$ induces a sample path $(\w_{n})$ by \[
\w_{n} = \left\{ \begin{array}{cc} g_{1} \cdots g_{n} & n > 0 \\ id & n=0 \\ g_{0}^{-1} \cdots g_{n+1}^{-1} & n < 0, \end{array}\right.
\]
which constitutes a random walk with transition probability $\mu$. Then the \emph{Bernoulli shift} $(g_{n})_{n \in \Z} \mapsto (g_{n+1})_{n \in \Z}$ on the step space induces the shift $(\w_{n})_{n \in \Z} \mapsto (\w_{1}^{-1} \w_{n+1})_{n \in \Z}$. We also introduce the notation $\check{g}_{n} := g_{-n+1}^{-1}$ and $\check{\w}_{n} := \w_{-n}$. Note that $\check{\w}_{n} = \check{g}_{1} \cdots \check{g}_{n}$ and $(g_{n})_{n >0}$, $(\check{g}_{n})_{n > 0}$ are independent.

The following modification is suited for the Schottky set that appears later on. Let $n> 0$ and suppose that $\mu^{n} = \alpha \eta + (1-\alpha)\nu $ for some other measures $\eta, \nu$ on $G^{n}$ and $0 < \alpha < 1$. We now consider: \begin{itemize}
\item Bernoulli RVs $\rho_{i}$ (with $\Prob(\rho_{i} = 1) = \alpha$ and $\Prob(\rho_{i} = 0) = 1-\alpha$), 
\item $\eta_{i}$ with the law $\eta$, and 
\item $\nu_{i}$ with the law $\nu$,
\end{itemize} 
all independent. We then define \[
\begin{aligned}
\gamma_{i} = \left\{\begin{array}{cc}\eta_{i} & \textrm{when}\,\,\rho_{i} = 1, \\
\nu_{i} & \textrm{when}\,\,\rho_{i} = 0.\end{array}\right.
\end{aligned}
\] Then $\gamma_{i}$ are i.i.d. and $(\gamma_{1}, \ldots, \gamma_{k})$ has the law of $\mu^{nk}$. One can also construct i.i.d. $g_{i}$ with the law $\mu$ such that $(g_{n(i-1)+ 1}, \ldots, g_{ni}) = \gamma_{i}$ for each $i$.

In what follows, $\Omega$ denotes the ambient probability space on which $\rho_{i}$, $\nu_{i}$, $\eta_{i}$, $g_{i}$ are all measurable. $\w$ is reserved for the elements of $\Omega$. We fix notations $\w_{k} := g_{1} \cdots g_{k}$, $\sumRho(k) := \sum_{i=1}^{k} \rho_{i}$, and $\stopping(i) := \min \{ j \ge 0 : \sumRho(j) = i\}$.

\begin{quote}
We fix $\delta>0$ once and for all. Unless specified further, the ambient space $X$ is either a $\delta$-hyperbolic space or Teichm{\"u}ller space. In the former case, we regard all points of $X$ to be $\epsilon$-thick for any $\epsilon>0$. In the latter case, we refer to the notions in Subsection \ref{subsection:Teich}.
\end{quote}

\section{Witnessing and alignment}\label{section:witness}

In $\delta$-hyperbolic spaces, having small Gromov product among points is not transitive. More precisely, having  $(a_{i-1}, a_{i+1})_{a_{i}} < C$ for each $i$ does not immediately guarantee that $(a_{i}, a_{k})_{a_{j}} < C$ for every $i < j < k$. If consecutive points are distant, however, then we have $(a_{i}, a_{k})_{a_{j}} < C+\delta$ for $i<j < k$. Still, we should cope with a small increase of Gromov products; this is why we record `chains' of such points and calculate the Gromov products among points only when needed. (See \cite{gouezel2022exponential} for details.)

The situation is trickier in Teichm{\"u}ller space, where we should rely on the partial hyperbolicity due to Rafi. Thus, we will record chains of geodesics `witnessed by' well-aligned thick segments. This complication is mainly for Teichm{\"u}ller space; readers who are interested in $\delta$-hyperbolic spaces can also employ the chain condition and pivot construction in \cite{gouezel2022exponential}. 

The alignment lemmata for geodesics witnessed by thick geodesics were partially explained in \cite{baik2021linear}. In particular, Lemma 4.18 of \cite{baik2021linear} is a precursor of Lemma \ref{lem:concat}. Still, we record here proofs of the alignment lemmata in order to integrate the case of Gromov hyperbolic spaces and Teichm{\"u}ller space.

\begin{definition}[Witnessing in $\delta$-hyperbolic spaces] \label{dfn:witnessingDelta}
Let $X$ be a $\delta$-hyperbolic space, $D>0$, and $[x_{1}, y_{1}]$, $\ldots$, $[x_{n}, y_{n}]$, $[x, y]$ be segments on $X$. We say that $[x, y]$ is \emph{$D$-witnessed by $([x_{1}, y_{1}], \ldots, [x_{n}, y_{n}])$} if:\begin{enumerate}
\item $(x_{i-1}, x_{i+1})_{x_{i}} < D$ for $i = 1, \ldots, n$, where $x = x_{0}$ and $y = x_{n+1}$;
\item $(y_{i-1}, y_{i+1})_{y_{i}} < D$ for $i = 1, \ldots, n$, where $x = y_{0}$ and $y = y_{n+1}$, and
\item $(y_{i-1}, y_{i})_{x_{i}}, (x_{i}, x_{i+1})_{y_{i}} < D$ for $i=1, \ldots, n$.
\end{enumerate}

We say that segments $\gamma_{1}$, $\gamma_{2}$ on $X$ are \emph{$D$-glued} (at $x \in X$) if $\gamma_{1}$, $\gamma_{2}$ shares the initial point $x$ and $(\gamma_{1}, \gamma_{2})_{\ast} < D$.
\end{definition}

\begin{definition}[Witnessing in Teichm{\"u}ller space] \label{dfn:witnessingTeich}
Let $X$ be Teichm{\"u}ller space, $D>0$ and $\gamma_{1}, \ldots, \gamma_{m}$, $\eta$ be geodesic segments on $X$.

We say that $\eta$ is \emph{$D$-witnessed by ($\gamma_{1}, \ldots, \gamma_{m})$} if $\eta$ contains subsegments $\eta_{1}, \ldots, \eta_{m}$ such that $\eta_{i-1}$ appears earlier than $\eta_{i}$ and $\eta_{i}$ $D$-fellow travels with $\gamma_{i}$.

We say that segments $\gamma_{1}$, $\gamma_{2}$ on $X$ are \emph{$D$-glued} (at $x \in X$) if $\gamma_{1}$, $\gamma_{2}$ shares the initial point $x$ and $(\gamma_{1}, \gamma_{2})_{\ast} < D$.
\end{definition}

\begin{figure}[H]
\begin{tikzpicture}

\draw[line width = 1.4 mm, black!30] (0.6, 0) -- (2, 0);
\draw[line width = 1.4 mm, black!30] (2.5, 0) -- (3.9, 0);

\draw[line width = 1.4 mm, black!30] (4.4, 0) -- (5.8, 0);

\draw (0, 0) -- (6.4, 0);
\draw[very thick, decoration={markings, mark=at position 0.57 with {\draw (-0.2, 0.07) -- (0, 0) -- (-0.2, -0.07);}}, postaction={decorate}] (0.6, 0.4) -- (2, 0.4);
\draw[dashed] (0.6, 0) -- (0.6, 0.4);
\draw[dashed] (2, 0) -- (2, 0.4);

\draw[very thick, decoration={markings, mark=at position 0.57 with {\draw (-0.2, 0.07) -- (0, 0) -- (-0.2, -0.07);}}, postaction={decorate}] (2.5, -0.4) -- (3.9, -0.4);
\draw[dashed] (2.5, 0) -- (2.5, -0.4);
\draw[dashed] (3.9, 0) -- (3.9, -0.4);

\draw[very thick, decoration={markings, mark=at position 0.57 with {\draw (-0.2, 0.07) -- (0, 0) -- (-0.2, -0.07);}}, postaction={decorate}] (4.4, 0.4) -- (5.8, 0.4);
\draw[dashed] (4.4, 0) -- (4.4, 0.4);
\draw[dashed] (5.8, 0) -- (5.8, 0.4);

\fill (0, 0) circle (0.05);
\fill (6.4, 0) circle (0.05);

\draw (-0.25, 0) node {$x$};
\draw (6.65, 0) node {$y$};

\draw (0.35, 0.45) node {$x_{1}$};
\draw (2.25, 0.45) node {$y_{1}$};

\draw (2.25, -0.45) node {$x_{2}$};
\draw (4.15, -0.45) node {$y_{2}$};

\draw (4.15, 0.45) node {$x_{3}$};
\draw (6.05, 0.45) node {$y_{3}$};

\end{tikzpicture}
\caption{$D$-witnessing in Teichm{\"u}ller space. Here $[x, y]$ is $D$-witnessed by $([x_{1}, y_{1}], [x_{2}, y_{2}], [x_{3}, y_{3}])$.}
\label{fig:witnessing}
\end{figure}
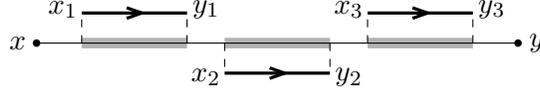

\begin{lem}\label{lem:wideWitness}
Let $X$ be a $\delta$-hyperbolic space and $[x, y]$ be a segment $D$-witnessed by $([x_{1}, y_{1}]$, $\ldots$, $[x_{n}, y_{n}])$. If each of $[x_{i}, y_{i}]$ is longer than $3D + 3\delta$, then $(x_{i}, x_{k})_{x_{j}}, (y_{i}, y_{k})_{y_{j}} < D+2\delta$ for any $0 \le i \le j \le k \le n+1$. Moreover, we have \[
d(x, y) > \sum_{i=1}^{n} d(x_{i}, y_{i}) - 3nD - 2n\delta.
\]
\end{lem}

\begin{proof}
For convenience, we let $x_{0} = x$ and $x_{n+1} = y$. We will apply the following lemma.

\begin{lem}[{\cite[Lemma 2.4]{baik2021linear}}]
Let $n \ge 0$ and $x_{0}, \ldots, x_{n+1} \in X$. Suppose that \[
(x_{i-1}, x_{i+1})_{x_{i}} + (x_{i}, x_{i+2})_{x_{i+1}} < d(x_{i}, x_{i+1}) - 3\delta
\]
for $i=1, \ldots, n-1$. Then: \begin{enumerate}
\item $|(x_{i}, x_{k})_{x_{j}} - (x_{j-1}, x_{j+1})_{x_{j}} | \le 2\delta$ for $0 \le i < j < k \le n+1$, and 
\item \[
	\left|  \left(\sum_{i=0}^{n} d(x_{i}, x_{i+1}) - 2 \sum_{i=1}^{n} (x_{i-1}, x_{i+1})_{x_{i}} \right) - d(x_{0}, x_{n+1}) \right| \le 2n \delta.
\]
\end{enumerate}
\end{lem}

To apply this, we first check that \[\begin{aligned}
d(x_{i}, x_{i+1}) &= \frac{1}{2} [d(x_{i}, x_{i+1}) + d(x_{i+1}, y_{i}) ] + \frac{1}{2} [d(x_{i}, x_{i+1}) - d(x_{i+1}, y_{i}) ] \\
&= \frac{1}{2} d(x_{i}, y_{i}) +  \frac{1}{2}\left[ d(x_{i}, x_{i+1}) - d(x_{i+1}, y_{i})\right] \\
&= d(x_{i}, y_{i}) + \frac{1}{2} \left[ d(x_{i}, x_{i+1}) - d(x_{i}, y_{i}) - d(y_{i}, x_{i+1})\right] \\
&= d(x_{i}, y_{i}) - (x_{i}, x_{i+1})_{y_{i}} > (3D+ 3\delta) - D = 2D + 3\delta
\end{aligned}
\]
for $i=1, \ldots, n$. This implies \[
(x_{i}, x_{i+1})_{x_{i}} + (x_{i}, x_{i+2})_{x_{i+1}} < 2D \le d(x_{i}, x_{i+1}) - 3\delta
\]
for $i=1, \ldots, n-1$, the desired hypothesis. Hence we have \[
(x_{i}, x_{k})_{x_{j}} \le (x_{j-1}, x_{j+1})_{x_{j}} + 2\delta <D + 2\delta
\]
for $0 \le i < j < k \le n+1$, and \[\begin{aligned}
d(x, y) = d(x_{0}, x_{n+1}) &\ge \sum_{i=0}^{n} d(x_{i}, x_{i+1}) - 2 \sum_{i=1}^{n} (x_{i-1}, x_{i+1})_{x_{i}} - 2n \delta \\
&> \sum_{i=1}^{n} [d(x_{i}, y_{i}) - (x_{i}, x_{i+1})_{y_{i}}] - 2nD - 2n\delta \\
&\ge \sum_{i=1}^{n} d(x_{i}, y_{i}) - 3nD - 2n\delta.
\end{aligned}
\]
For a similar reason we also have $(y_{i}, y_{k})_{y_{j}}< D+2\delta$ for $0\le i < j < k\le n+1$.

\end{proof}

The corresponding lemma for Teichm{\"u}ller space is as follows.

\begin{lem}\label{lem:wideWitnessTeich}
Let $X$ be Teichm{\"u}ller space and $[x, y]$ be a segment $D$-witnessed by ($[x_{1}, y_{1}], \ldots, [x_{n}, y_{n}])$. Then $(x_{i}, x_{k})_{x_{j}}$, $(y_{i}, y_{k})_{y_{j}} < 3D$ for $0 \le i \le j\le k \le n+1$. Moreover, we have \[
d(x, y) > \sum_{i=1}^{n} d(x_{i}, y_{i}) - 2nD.
\]
\end{lem}

\begin{proof}
We set $x_{0} = x_{0}' := x$ and $x_{n+1} = x_{n+1}' :=y$.
Let $[x_{i}', y_{i}']$'s be subsegments of $[x, y]$ such that $[x_{i}', y_{i}']$ appears earlier than $[x_{i+1}', y_{i+1}']$ and $[x_{i}, y_{i}]$, $[x_{i}', y_{i}']$ are $D$-fellow traveling. We then have \[
d(x_{i}', y_{i}') \ge d(x_{i}, y_{i}) - d(x_{i}, x_{i}') - d(y_{i}, y_{i}') \ge d(x_{i}, y_{i}) - 2D
\]
and \[\begin{aligned}
d(x, y) &= d(x, x_{1}') + \sum_{i=1}^{n} d(x_{i}', y_{i}') + \sum_{i=1}^{n-1} d(y_{i}', x_{i+1}') + d(y_{n}', y) \\
&\ge \sum_{i=1}^{n} d(x_{i}', y_{i}') \ge \sum_{i=1}^{n} d(x_{i}, y_{i}) - 2nD.
\end{aligned}
\]
Moreover, for $0\le i <j < k \le n+1$ we have \[\begin{aligned}
2(x_{i}, x_{k})_{x_{j}} &= d(x_{i}, x_{j}) + d(x_{j}, x_{k}) - d(x_{i}, x_{k}) \\
&\le [d(x_{i}, x_{i}') + d(x_{i}', x_{j}') + d(x_{j}', x_{j})] + [d(x_{j}, x_{j}') + d(x_{j}', x_{k}') + d(x_{k}', x_{k})] \\
&- [d(x_{i}', x_{k}') - d(x_{i}, x_{i}') - d(x_{k}, x_{k}')] \\
&\le 6D + [d(x_{i}', x_{j}') + d(x_{j}', x_{k}') - d(x_{i}', x_{k}')] = 6D.
\end{aligned}
\]
Here the final equality follows from the fact that $x_{i}'$, $x_{j}'$, $x_{k}'$ are on the same geodesic, in order from closest to farthest from $x$. Similarly we have $(y_{i}, y_{k})_{y_{j}} \le 3D$.
\end{proof}

Next lemma is due to the fact that the insize of a geodesic triangle in a $\delta$-hyperbolic space is at most $6\delta$. (cf. \cite[Proposition III.H.1.22]{bridson1999metric})

\begin{lem}\label{lem:witnessDist}
Let $X$ be a geodesic $\delta$-hyperbolic space. If a segment $[x_{0}, y_{0}]$ is  $D$-witnessed by another segment $[x, y]$, then \[
d(x, [x_{0}, y_{0}]),\,\, d(y, [x_{0}, y_{0}]) \le D + 6\delta.\]
\end{lem}

We now establish two lemmata that promote `partial witnessing' into genuine witnessing.

\begin{lem}[Small products guarantee witnessing I]\label{lem:farSegment}
For each $C, \epsilon > 0$, there exists $D >C$ that satisfies the following. If $x_{0}$, $x_{1}$, $y_{0}$, $y_{1} \in X$ satisfy \begin{enumerate}
\item $[x_{0}, x_{1}]$, $[y_{0}, y_{1}]$ are $\epsilon$-thick;
\item $(x_{0}, y_{1})_{x_{1}}$, $(y_{0}, x_{1})_{y_{1}} < C$, and
\item $d(x_{1}, y_{1}) \ge d(x_{0}, x_{1}), d(y_{0}, y_{1}), 3D$,
\end{enumerate}
then $[x_{0}, y_{0}]$ is $D$-witnessed by $([x_{0}, x_{1}], [y_{1}, y_{0}])$.
\end{lem}

\begin{proof}
When $X$ is a $\delta$-hyperbolic space, we take $D = C + \delta + 1$. Then \[
(x_{0}, x_{1})_{y_{1}} = d(x_{1}, y_{1}) - (x_{0}, y_{1})_{x_{1}} > 2C + \delta +1,
\]
\[
\min\{(x_{0}, y_{0})_{y_{1}}, (x_{0}, x_{1})_{y_{1}}\} - \delta \le (x_{1}, y_{0})_{y_{1}} < C
\]
imply $(x_{0}, y_{0})_{y_{1}} < C+\delta < D$. Similarly, we deduce $(x_{0}, y_{0})_{x_{1}} < D$.

When $X$ is Teichm{\"u}ller space, we take \[\begin{aligned}
D_{1} &=  \mathscr{B}(\epsilon, C + \mathscr{C}(\epsilon) +3\mathscr{D}(\epsilon)),\\
\epsilon_{1} &= \epsilon e^{-2D_{1}}, \\
D_{2} &= 2C + 2\mathscr{C}(\epsilon) + 6\mathscr{D}(\epsilon) + \mathscr{C}(\epsilon_{1}) +2 \mathscr{D}(\epsilon_{1}) + 1,\\
D &= 2\mathscr{D}(\epsilon, 2D_{2}) + C+2D_{2}.
\end{aligned}
\]

Note first that \[
(x_{1}, y_{1})_{x_{0}} = d(x_{0}, x_{1}) - (x_{0}, y_{1})_{x_{1}} \ge d(x_{0}, x_{1}) - C.
\] By Lemma \ref{lem:GromProdWitness}, there exist points $q \in [x_{0}, x_{1}]$ and $q' \in [x_{0}, y_{1}]$ such that $d(q, q') \le \mathscr{D}(\epsilon)$ and $d(x_{1}, q) \le C + \mathscr{C}(\epsilon) +2\mathscr{D}(\epsilon)$. Hence, $d(x_{1}, q') \le C + \mathscr{C}(\epsilon) + 3\mathscr{D}(\epsilon)$, $[x_{0}, x_{1}]$ and $[x_{0}, q']$ $D_{1}$-fellow travel, and $[x_{0}, q']$ is $\epsilon_{1}$-thick.

We now take a subsegment $[q_{1}, q_{2}]$ of $[x_{0}, q']$ such that \[
d(q_{1}, q_{2}) = \mathscr{C}(\epsilon_{1}), \quad d(q_{2} ,q') = C + \mathscr{C}(\epsilon) + 3\mathscr{D}(\epsilon) + \mathscr{D}(\epsilon_{1}) + 1.
\] (If this is not possible, then $d(x_{0}, x_{1}) \le d(x_{0}, q')+ d(q', x_{1}) \le D_{2}$ so $\{x_{0}\}\subseteq [x_{0}, y_{0}]$ and $[x_{0}, x_{1}]$ $D$-fellow travel.) Observe that \[\begin{aligned}
d([q_{1}, q_{2}], [y_{0}, y_{1}]) &\ge d([q_{1}, q_{2}], y_{1}) - d(y_{0}, y_{1}) = d(q_{2}, y_{1}) - d(y_{0}, y_{1}) \\
&= d(q_{2}, q') + d(q', y_{1}) - d(y_{0}, y_{1})\\
& \ge d(q_{2}, q') +  d(x_{1}, y_{1})- d(x_{1}, q') - d(y_{0}, y_{1})  > \mathscr{D}(\epsilon_{1}).
\end{aligned}
\]
Given this, we apply Theorem \ref{thm:rafi2} to the triangle $\triangle x_{0} y_{0} y_{1}$. Then there exist $a \in [q_{1}, q_{2}]$ and $b \in [x_{0} ,y_{0}]$ that are within distance $\mathscr{D}(\epsilon_{1})$, and $d(b, x_{1}) \le d(b, a) + d(a, q') + d(q', x_{1}) \le D_{2}$. At the moment, if we take $b' \in [x_{0}, y_{0}]$ so that $d(x_{0}, b') = d(x_{0}, x_{1})$, then $d(b', x_{1}) \le d(b', b) + d(b, x_{1}) \le 2D_{2}$. By Theorem \ref{thm:rafi1}, $[x_{0}, b']$ and $[x_{0}, y_{0}]$ $D$-fellow travel.

By symmetry, there exist $b'' \in [x_{0}, y_{0}]$ so that $d(b'', y_{0}) = d(y_{1}, y_{0})$ (or $b'' = y_{0}$) and  $[b'', y_{0}]$ and $[y_{1}, y_{0}]$ $D$-fellow travel. It remains to show that $[x_{0}, b']$ appears earlier than $[b'',y_{0}]$. By Fact \ref{fact:GromProdFact2}, $d(y_{1}, [x_{0}, x_{1}]) \ge (x_{0}, x_{1})_{y_{1}} \ge d(x_{1}, y_{1}) - C$ holds. Since the Hausdorff distance between $[x_{0}, x_{1}]$ and $[x_{0}, b']$ is less than $D$, we have $d(y_{1}, [x_{0}, b']) \ge d(x_{1}, y_{1}) - C - D> D$. Since $d(b'', y_{1}) < D$, $b'' \notin [x_{0}, b']$ and hence the conclusion.
\end{proof}

Lemma \ref{lem:farSegment} will later play a crucial role when we control the translation length of words by early pivoting. Before doing that, however, we should first define pivotal times. The lengths of the intermediate segments are not controlled at this moment, so we instead rely on the following lemma.

\begin{lem}[Small products guarantee witnessing II]\label{lem:1segment}
For each $C, \epsilon> 0$, there exists $D>C$ that satisfies the following condition. If 4 points $x_{0}, x_{1}, y_{0}, y_{1}$ in $X$ satisfy that:\begin{enumerate}
\item $[x_{0}, x_{1}]$, $[y_{0}, y_{1}]$ are $\epsilon$-thick and
\item $(x_{0}, y_{1})_{x_{1}}$, $(x_{0}, y_{0})_{y_{1}} < C$,
\end{enumerate}
then $[x_{0}, y_{0}]$ is $D$-witnessed by $([x_{0}, x_{1}], [y_{1}, y_{0}])$.
\end{lem}

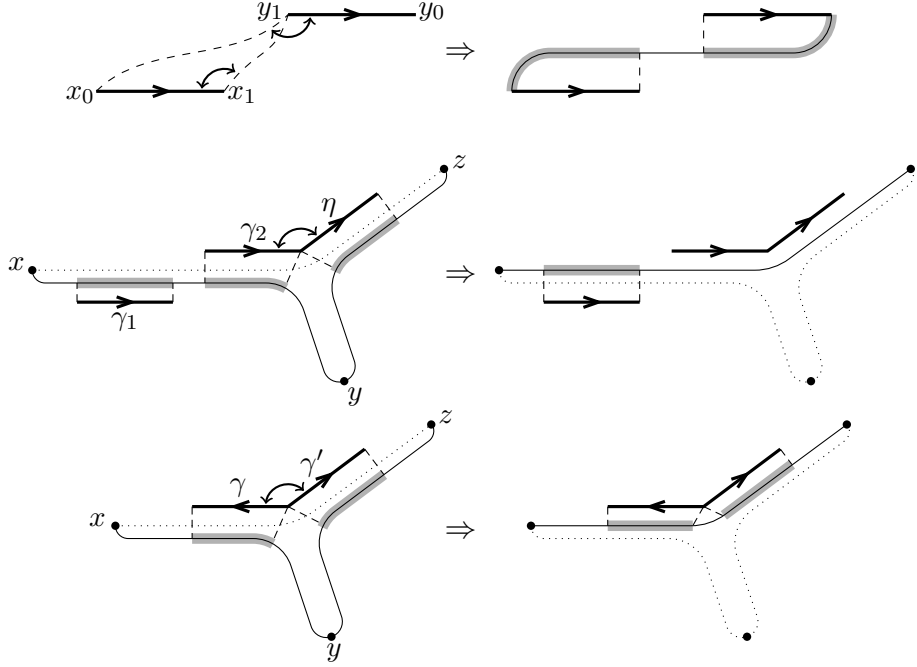
\begin{figure}[h]
\begin{tikzpicture}[scale=0.85]
\draw[very thick, decoration={markings, mark=at position 0.56 with {\draw (-0.2, 0.07) -- (0, 0) -- (-0.2, -0.07);}}, postaction={decorate}] (0, 0) -- (2, 0);
\draw[very thick, decoration={markings, mark=at position 0.556with {\draw (-0.2, 0.07) -- (0, 0) -- (-0.2, -0.07);}}, postaction={decorate}] (3, 1.2) -- (5, 1.2);
\draw[dashed] (2, 0) .. controls (2.5, 0.7) and (2.8, 0.5) .. (3, 1.2);
\draw[dashed] (0, 0) .. controls (0.8, 0.8) and (2.2, 0.4) .. (3, 1.2);
\draw[thick, <->, shift={(2, 0)}, rotate=-7] (-0.35, 0) arc (180:67:0.35);
\draw[thick, <->, shift={(3, 1.2)}, rotate=-7] (0.35, 0) arc (0:-130:0.35);
\draw(5.7, 0.6) node {\large $\Rightarrow$};

\draw (-0.25, -0.05) node {$x_{0}$};
\draw (2.28, -0.05) node {$x_{1}$};
\draw (2.72, 1.25) node {$y_{1}$};
\draw (5.25, 1.25) node {$y_{0}$};

\begin{scope}[shift={(6.5, 0)}]
\draw[line width = 1.4 mm, black!30]  (0, 0) arc (180:90:0.6) -- (2, 0.6);
\draw[line width = 1.4 mm, black!30]  (3, 0.6) --(4.4, 0.6) arc (-90:0:0.6);
\draw[very thick, decoration={markings, mark=at position 0.56 with {\draw (-0.2, 0.07) -- (0, 0) -- (-0.2, -0.07);}}, postaction={decorate}] (0, 0) -- (2, 0);
\draw[very thick, decoration={markings, mark=at position 0.56 with {\draw (-0.2, 0.07) -- (0, 0) -- (-0.2, -0.07);}}, postaction={decorate}] (3, 1.2) -- (5, 1.2);
\draw (0, 0) arc (180:90:0.6) -- (4.4, 0.6) arc (-90:0:0.6);
\draw[dashed] (2, 0) -- (2, 0.6);
\draw[dashed] (3, 1.2) -- (3, 0.6);

\end{scope}

\begin{scope}[shift={(-0.3, -3.3)}]

\draw[line width = 1.4 mm, black!30]  (0, 0.3) -- (1.5, 0.3);
\draw[line width = 1.4 mm, black!30]  (2, 0.3) -- (3.5 - 0.096856471680698 - 3*0.15, 0.3) arc (90:90-30:0.15*4.162277660168379);
\draw[line width = 1.4 mm, black!30, shift={(5, 1.3)}, rotate=36.869897645844021]  (0, 0) -- (-1.5+ 0.096856471680698+3*0.15, 0) arc (90:90+30:0.15*4.162277660168379);

\draw[thick, <->, shift={(3.5, 0.8)}, rotate=-7] (-0.35, 0) arc (180:52:0.35);
\draw[very thick, decoration={markings, mark=at position 0.57 with {\draw (-0.2, 0.07) -- (0, 0) -- (-0.2, -0.07);}}, postaction={decorate}] (0, 0) -- (1.5, 0);
\draw[very thick, decoration={markings, mark=at position 0.28 with {\draw (-0.2, 0.07) -- (0, 0) -- (-0.2, -0.07);}, mark=at position 0.78 with {\draw (-0.2, 0.07) -- (0, 0) -- (-0.2, -0.07);}}, postaction={decorate}] (2, 0.8) -- (3.5,0.8) -- (4.7, 1.7);
\draw (-0.7, 0.5) arc (-180:-90:0.2) -- (3.5 - 0.096856471680698 - 3*0.15, 0.3) arc (90:90-71.565051177077989:0.15*4.162277660168379)-- (3.5 + 0.6 - 0.25*0.948683298050514, 0.8 - 0.6*3 - 0.25*0.316227766016838) arc (18.434948822922011-180:18.434948822922011:0.25) -- (3.5 - 0.096856471680698+ 0.948683298050514*0.15 + 2*0.237170824512628, 0.3 - 0.15*2.846049894151541 + 2*0.079056941504209) arc (270 - 71.565051177077989:270 - 2*71.565051177077989:0.15*4.162277660168379) -- (5.7, 1.825) arc (36.869897645844021-90:36.869897645844021:0.2);

\draw[dotted] (-0.7, 0.5) -- (3.6 - 0.4, 0.5) arc (-90:-90+36.869897645844021:0.4*3) -- (5.74, 2.085);
\draw[densely dashed] (0, 0) -- (0, 0.3);
\draw[densely dashed] (1.5, 0) -- (1.5, 0.3);
\draw[densely dashed] (2, 0.8) -- (2, 0.3);
\draw[densely dashed] (3.5, 0.8) -- (3.25, 0.23);
\draw[densely dashed] (3.5, 0.8) -- (4.05, 0.52);
\draw[densely dashed] (4.7, 1.7) -- (5, 1.3);

\fill (-0.7, 0.5) circle (0.065);
\fill  (3.5 + 0.6 +0.25*0.316227766016838, 0.8 - 0.6*3 - 0.25*0.948683298050514) circle (0.065);
\fill  (5.74, 2.085) circle (0.065);

\draw (6, 0.4) node {\large $\Rightarrow$};

\draw (-0.95, 0.6) node {$x$};
\draw (4.35, -1.48) node {$y$};
\draw (5.98, 2.2) node {$z$};

\draw (0.75, -0.3) node {$\gamma_{1}$};
\draw (2.77, 1.09) node {$\gamma_{2}$};
\draw (3.94,  1.48) node {$\eta$};

\end{scope}

\begin{scope}[shift={(7, -3.3)}]

\draw[line width = 1.4 mm, black!30]  (0, 0.5) -- (1.5, 0.5);

\draw[very thick, decoration={markings, mark=at position 0.57 with {\draw (-0.2, 0.07) -- (0, 0) -- (-0.2, -0.07);}}, postaction={decorate}] (0, 0) -- (1.5, 0);
\draw[very thick, decoration={markings, mark=at position 0.28 with {\draw (-0.2, 0.07) -- (0, 0) -- (-0.2, -0.07);}, mark=at position 0.78 with {\draw (-0.2, 0.07) -- (0, 0) -- (-0.2, -0.07);}}, postaction={decorate}] (2, 0.8) -- (3.5,0.8) -- (4.7, 1.7);
\draw[dotted] (-0.7, 0.5) arc (-180:-90:0.2) -- (3.5 - 0.096856471680698 - 3*0.15, 0.3) arc (90:90-71.565051177077989:0.15*4.162277660168379)-- (3.5 + 0.6 - 0.25*0.948683298050514, 0.8 - 0.6*3 - 0.25*0.316227766016838) arc (18.434948822922011-180:18.434948822922011:0.25) -- (3.5 - 0.096856471680698+ 0.948683298050514*0.15 + 2*0.237170824512628, 0.3 - 0.15*2.846049894151541 + 2*0.079056941504209) arc (270 - 71.565051177077989:270 - 2*71.565051177077989:0.15*4.162277660168379) -- (5.7, 1.825) arc (36.869897645844021-90:36.869897645844021:0.2);
\draw (-0.7, 0.5) -- (3.6 - 0.3, 0.5) arc (-90:-90+36.869897645844021:0.3*3) -- (5.74, 2.085);
\draw[densely dashed] (0, 0) -- (0, 0.5);
\draw[densely dashed] (1.5, 0) -- (1.5, 0.5);

\fill (-0.7, 0.5) circle (0.065);
\fill  (3.5 + 0.6 +0.25*0.316227766016838, 0.8 - 0.6*3 - 0.25*0.948683298050514) circle (0.065);
\fill  (5.74, 2.085) circle (0.065);

\end{scope}

\begin{scope}[shift={(-0.5, -7.3)}]

\draw[line width = 1.4 mm, black!30]  (2, 0.3) -- (3.5 - 0.096856471680698 - 3*0.15, 0.3) arc (90:90-30:0.15*4.162277660168379);
\draw[line width = 1.4 mm, black!30, shift={(5, 1.3)}, rotate=36.869897645844021]  (0, 0) -- (-1.5+ 0.096856471680698+3*0.15, 0) arc (90:90+30:0.15*4.162277660168379);

\draw[thick, <->, shift={(3.5, 0.8)}, rotate=-7] (-0.35, 0) arc (180:52:0.35);
\draw[very thick, decoration={markings, mark=at position 0.22 with {\draw (0.2, 0.07) -- (0, 0) -- (0.2, -0.07);}, mark=at position 0.78 with {\draw (-0.2, 0.07) -- (0, 0) -- (-0.2, -0.07);}}, postaction={decorate}] (2, 0.8) -- (3.5,0.8) -- (4.7, 1.7);
\draw (0.8, 0.5) arc (-180:-90:0.2) -- (3.5 - 0.096856471680698 - 3*0.15, 0.3) arc (90:90-71.565051177077989:0.15*4.162277660168379)-- (3.5 + 0.6 - 0.25*0.948683298050514, 0.8 - 0.6*3 - 0.25*0.316227766016838) arc (18.434948822922011-180:18.434948822922011:0.25) -- (3.5 - 0.096856471680698+ 0.948683298050514*0.15 + 2*0.237170824512628, 0.3 - 0.15*2.846049894151541 + 2*0.079056941504209) arc (270 - 71.565051177077989:270 - 2*71.565051177077989:0.15*4.162277660168379) -- (5.7, 1.825) arc (36.869897645844021-90:36.869897645844021:0.2);

\draw[dotted] (0.8, 0.5) -- (3.6 - 0.4, 0.5) arc (-90:-90+36.869897645844021:0.4*3) -- (5.74, 2.085);

\draw[densely dashed] (2, 0.8) -- (2, 0.3);
\draw[densely dashed] (3.5, 0.8) -- (3.25, 0.23);
\draw[densely dashed] (3.5, 0.8) -- (4.05, 0.52);
\draw[densely dashed] (4.7, 1.7) -- (5, 1.3);

\fill (0.8, 0.5) circle (0.065);
\fill  (3.5 + 0.6 +0.25*0.316227766016838, 0.8 - 0.6*3 - 0.25*0.948683298050514) circle (0.065);
\fill  (5.74, 2.085) circle (0.065);

\draw (6.2, 0.4) node {\large $\Rightarrow$};

\draw (0.52, 0.6) node {$x$};
\draw (4.22, -1.48) node {$y$};
\draw (5.97, 2.2) node {$z$};

\draw (2.72, 1.09) node {$\gamma$};
\draw (3.88,  1.48) node {$\gamma'$};

\end{scope}

\begin{scope}[shift={(6, -7.3)}]

\draw[line width = 1.4 mm, black!30]  (2, 0.5) -- (3.33, 0.5);
\draw[line width = 1.4 mm, black!30, shift={(4.88, 1.46)}, rotate=36.869897645844021]  (0, 0) -- (-1.33, 0);

\draw[very thick, decoration={markings, mark=at position 0.22 with {\draw (0.2, 0.07) -- (0, 0) -- (0.2, -0.07);}, mark=at position 0.78 with {\draw (-0.2, 0.07) -- (0, 0) -- (-0.2, -0.07);}}, postaction={decorate}] (2, 0.8) -- (3.5,0.8) -- (4.7, 1.7);
\draw[dotted] (0.8, 0.5) arc (-180:-90:0.2) -- (3.5 - 0.096856471680698 - 3*0.15, 0.3) arc (90:90-71.565051177077989:0.15*4.162277660168379)-- (3.5 + 0.6 - 0.25*0.948683298050514, 0.8 - 0.6*3 - 0.25*0.316227766016838) arc (18.434948822922011-180:18.434948822922011:0.25) -- (3.5 - 0.096856471680698+ 0.948683298050514*0.15 + 2*0.237170824512628, 0.3 - 0.15*2.846049894151541 + 2*0.079056941504209) arc (270 - 71.565051177077989:270 - 2*71.565051177077989:0.15*4.162277660168379) -- (5.7, 1.825) arc (36.869897645844021-90:36.869897645844021:0.2);
\draw (0.8, 0.5) -- (3.6 - 0.3, 0.5) arc (-90:-90+36.869897645844021:0.3*3) -- (5.74, 2.085);
\draw[densely dashed] (2, 0.8) -- (2, 0.5);
\draw[densely dashed] (3.5, 0.8) -- (3.33, 0.5);
\draw[densely dashed] (3.5, 0.8) -- (3.82, 0.65);
\draw[densely dashed] (4.7, 1.7) -- (4.88, 1.46);

\fill (0.8, 0.5) circle (0.065);
\fill  (3.5 + 0.6 +0.25*0.316227766016838, 0.8 - 0.6*3 - 0.25*0.948683298050514) circle (0.065);
\fill  (5.74, 2.085) circle (0.065);
\end{scope}

\end{tikzpicture}
\caption{Schematics for Lemma \ref{lem:1segment}, \ref{lem:concat} and \ref{lem:concatUlt}.}
\label{fig:scheme1}
\end{figure}

\begin{proof}
When $X$ is a $\delta$-hyperbolic space, we take $D = 2C$. First observe that \[\begin{aligned}
d(x_{0}, y_{0}) &= d(x_{0}, y_{1}) + d(y_{1}, y_{0}) - 2(x_{0}, y_{0})_{y_{1}} \\
&= d(x_{0}, x_{1}) + d(x_{1}, y_{1}) - 2(x_{0}, y_{1})_{x_{1}} + d(y_{1}, y_{0}) - 2(x_{0}, y_{0})_{y_{1}} \\
& > d(x_{0}, x_{1}) + d(x_{1}, y_{1}) + d(y_{1}, y_{0})  - 4C.
\end{aligned}
\]
This implies the following: \[\begin{aligned}
2(x_{1}, y_{0})_{y_{1}} &= d(x_{1} ,y_{1}) + d(y_{1}, y_{0}) - d(x_{1}, y_{0}) \\
&\le d(x_{1}, y_{1}) + d(y_{1}, y_{0}) - [d(x_{0}, y_{0}) - d(x_{0}, x_{1})] < 4C,
\end{aligned}
\]
\[\begin{aligned}
2(x_{0}, y_{0})_{x_{1}} &=  d(x_{0}, x_{1}) + d(x_{1}, y_{0}) - d(x_{0}, y_{0}) \\
&\le d(x_{0}, x_{1}) + d(x_{1}, y_{1}) + d(y_{1}, y_{0}) - d(x_{0}, y_{0}) < 4C.
\end{aligned}
\]

When $X$ is Teichm{\"u}ller space, we take \[
\begin{aligned}
D_{1} &=  \mathscr{B}(\epsilon,C + \mathscr{C}(\epsilon) + 3\mathscr{D}(\epsilon)),
\\
\epsilon_{1} &= \epsilon e^{-2D_{1}} \\
D_{2} &= 4C + 3\mathscr{C}(\epsilon) + 9\mathscr{D}(\epsilon) + \mathscr{C}(\epsilon_{1}) + 2\mathscr{D}(\epsilon_{1}) + 1,\\
D &= \mathscr{B}(\epsilon, D_{2}) + C + D_{1}.
\end{aligned}
\]

As in the previous lemma, from $(x_{0}, y_{0})_{y_{1}} < C$ and $(x_{0}, y_{1})_{x_{1}} < C$, we obtain points $p' \in [x_{0}, y_{0}]$, $q' \in [x_{0}, y_{1}]$ such that $d(p', y_{1}), d(q', x_{1}) \le C + \mathscr{C}(\epsilon) + 3\mathscr{D}(\epsilon)$. This implies that $[p', y_{0}]$ and $[y_{1}, y_{0}]$ $D_{1}$-fellow travel, and $[x_{0}, q']$ and $[x_{0}, x_{1}]$ $D_{1}$-fellow travel. Note also that $[x_{0}, q']$ is $\epsilon_{1}$-thick. 

We now take a subsegment $[q_{1}, q_{2}]$ of $[x_{0}, q']$ such that $d(q_{1}, q_{2}) = \mathscr{C}(\epsilon_{1})$ and $d(q_{2}, q') = 3C + 2\mathscr{C}(\epsilon) + 6\mathscr{D}(\epsilon) + \mathscr{D}(\epsilon_{1}) + 1$. (If this is not possible, then $[x_{0}, x_{1}]$ is shorter than $D_{2}$ so $[x_{0}, x_{1}]$ and $\{x_{0}\}$ $D_{2}$-fellow travel.) If a point $a \in [q_{1}, q_{2}]$ and $b \in [y_{1}, y_{0}]$ are within distance $\mathscr{D}(\epsilon_{1})$, then Fact \ref{fact:GromProdFact2} implies\[
(x_{0}, y_{0})_{y_{1}} \ge d(a, y_{1}) - d(a,b) \ge d(q_{2}, q') - \mathscr{D}(\epsilon_{1}) > C,
\]
which is a contradiction. Thus, we instead obtain points $a \in [q_{1}, q_{2}]$ and $b \in [x_{0}, y_{0}]$ that are within distance $\mathscr{D}(\epsilon_{1})$. Note that \[\begin{aligned}
d(x_{0}, b)& \le d(x_{0}, a) + d(a, b) \le d(x_{0}, q_{2}) + d(a, b) \\
& \le d(x_{0}, q') -[ 3C + 2\mathscr{C}(\epsilon) + 6\mathscr{D}(\epsilon) + \mathscr{D}(\epsilon_{1}) + 1] + \mathscr{D}(\epsilon_{1}) \\
& \le d(x_{0}, x_{1}) - [2C + \mathscr{C}(\epsilon) + 3\mathscr{D}(\epsilon) + 1] \\
& \le d(x_{0}, y_{1}) + C-  [2C + \mathscr{C}(\epsilon) + 3\mathscr{D}(\epsilon) + 1] \le  d(x_{0}, p') - 1.
\end{aligned}
\]
Hence, $[x_{0}, b]$ appears earlier than $[p', y_{0}]$. Moreover, since $x_{0}, x_{1}$ are $\epsilon$-thick and $d(x_{1}, b) \le d(x_{1}, q') + d(q', a) + d(a, b) \le D_{2}$, $[x_{0}, b]$ and $[x_{0}, x_{1}]$ $D$-fellow travel.
\end{proof}

Note that in Lemma \ref{lem:1segment}, if $[x, y]$ is $\epsilon$-thick and $(z, y)_{x} < C$ then $[z, y]$ is $D$-witnessed by $[x, y]$. We now introduce the notion of alignment.

\begin{definition}[Alignment and marking]
Let $C, D>0$. Sequences of segments $(\gamma_{i})_{i=1}^{N}$, $(\eta_{i})_{i=1}^{N}$ are said to be \emph{$D$-aligned} if the following hold: 
 \begin{enumerate}
\item for $i=1, \cdots, N$, $\gamma_{i}$ and $\bar{\eta}_{i}$ are $D$-glued at a point $p_{i}$;
\item for $i=2, \cdots, N$, $[p_{i-1}, p_{i}]$ is $D$-witnessed by $(\gamma_{i-1}, \eta_{i})$;
\end{enumerate}

Given $D$-aligned sequences $\left(\gamma_{i}\right)_{i=1}^{N}$, $\left(\eta_{i}\right)_{i=1}^{N}$ of segments, we say that a segment $[x, y]$ is \emph{$(C, D)$-marked with} $\left(\gamma_{i}\right)$, $\left(\eta_{i}\right)$ if $(\eta_{1}, x)_{\ast} < C$ and $(\bar{\gamma}_{N}, y)_{\ast} <C$. We also say that $[p_{1}, y]$ is \emph{$(C, D)$-head-marked} by $\left(\gamma_{i}\right)_{i=1}^{N}$, $\left(\eta_{i}\right)_{i=2}^{N}$. Similarly, we say that $[x, p_{N}]$ is \emph{$(C, D)$-tail-marked} by $\left(\gamma_{i}\right)_{i=1}^{N-1}$, $\left(\eta_{i}\right)_{i=1}^{N}$, and that $[p_{1}, p_{N}]$ is \emph{fully $D$-marked} by $\left(\gamma_{i}\right)_{i=1}^{N-1}$, $\left(\eta_{i}\right)_{i=2}^{N}$.
\end{definition}

\begin{figure}[H]
\begin{tikzpicture}[scale=0.9]
\def\a{0.8}
\def\b{3}
\def\c{1.2}

\draw[line width = 1.4 mm, black!30, rotate=36.869897645844021]  (0, 0) arc (-180:-90:0.5*\a) -- (\c, -0.5*\a);
\draw[line width = 1.4 mm, black!30, rotate=36.869897645844021]  (\b - \c, -0.5*\a) -- (\b - 0.5*\a, -0.5*\a) arc (90:0:0.5*\a);
\draw[line width = 1.4 mm, black!30, shift={(\a*0.6 + \b*0.8, -\a*0.8 + \b*0.6)}, rotate=-36.869897645844021] (0, 0) arc (180:90:0.5*\a) -- (\c, 0.5*\a);

\begin{scope}[shift={(2*\a*0.6 + 2*\b*0.8, 0)}]
\draw[line width = 1.4 mm, black!30, rotate=36.869897645844021]  (0, 0) arc (-180:-90:0.5*\a) -- (\c, -0.5*\a);
\draw[line width = 1.4 mm, black!30, rotate=36.869897645844021]  (\b - \c, -0.5*\a) -- (\b - 0.5*\a, -0.5*\a) arc (90:0:0.5*\a);
\draw[line width = 1.4 mm, black!30,rotate=-36.869897645844021] (0, 0) arc (0:-90:0.5*\a) -- (-\c, -0.5*\a);

\end{scope}

\draw[densely dashed] (-\c*0.8, \c*0.6) .. controls (\a*0.15 - \b*0.2 - \c*0.8, \c*0.6-\a*0.2) and (-\b*0.8 + \a*0.1 + \c*0.4, \b*0.6 - \a*0.7 - \c*0.3) .. (-\b*0.8 + \a*0.1, \b*0.6 - \a*0.7);
\draw[thick, <->, shift={(-\c*0.8, \c*0.6)}, rotate=20] (-0.33, 0) arc (180:295:0.33);

\draw[densely dashed] (\c*0.8, \c*0.6) -- (\c*0.8 + 0.5*\a*0.6, \c*0.6 - 0.5*\a*0.8);
\draw[densely dashed] (\a*0.6 + \b*0.8 - \c*0.8, -\a*0.8 + \b*0.6 - \c*0.6) -- (0.5*\a*0.6 + \b*0.8 - \c*0.8, -0.5*\a*0.8 + \b*0.6 - \c*0.6);
\draw[densely dashed] (\a*0.6 + \b*0.8 + \c*0.8, -\a*0.8 + \b*0.6 - \c*0.6) -- (1.5*\a*0.6 + \b*0.8 + \c*0.8, -0.5*\a*0.8 + \b*0.6 - \c*0.6);

\draw[very thick, decoration={markings, mark=at position 0.28 with {\draw (-0.2, 0.07) -- (0, 0) -- (-0.2, -0.07);}, mark=at position 0.78 with {\draw (-0.2, 0.07) -- (0, 0) -- (-0.2, -0.07);}}, postaction={decorate}] (-\c*0.8, \c*0.6) -- (0, 0) -- (\c*0.8, \c*0.6);
\draw[very thick, decoration={markings, mark=at position 0.28 with {\draw (-0.2, 0.07) -- (0, 0) -- (-0.2, -0.07);}, mark=at position 0.78 with {\draw (-0.2, 0.07) -- (0, 0) -- (-0.2, -0.07);}}, postaction={decorate}] (\a*0.6 + \b*0.8 - \c*0.8, -\a*0.8 + \b*0.6 - \c*0.6) -- (\a*0.6 + \b*0.8, -\a*0.8 + \b*0.6) -- (\a*0.6 + \b*0.8 + \c*0.8, -\a*0.8 + \b*0.6 - \c*0.6);
\draw[rotate=36.869897645844021] (0, 0) arc (-180:-90:0.5*\a) -- (\b - 0.5*\a, -0.5*\a) arc (90:0:0.5*\a);
\draw[shift={(\a*0.6 + \b*0.8, -\a*0.8 + \b*0.6)}, rotate=-36.869897645844021] (0, 0) arc (180:90:0.5*\a) -- (\b - 0.5*\a, 0.5*\a) arc (-90:0:0.5*\a);

\draw[xscale=-1, rotate=36.869897645844021] (0, 0) arc (-180:-90:0.5*\a) -- (\b - 0.5*\a, -0.5*\a);

\draw (-\c*0.8*0.6 + \a*0.6*0.35, \c*0.6*0.6 + \a*0.8*0.35) node {$\eta_{1}$}; 
\draw (\c*0.8*0.6 - \a*0.6*0.35, \c*0.6*0.6 + \a*0.8*0.35) node {$\gamma_{1}$}; 
\draw (\a*0.6 + \b*0.8-\c*0.8*0.6 + \a*0.6*0.35 , -\a*0.8 + \b*0.6 - \c*0.6*0.6 - \a*0.8*0.35) node {$\eta_{2}$};
\draw (\a*0.6 + \b*0.8+\c*0.8*0.6 - \a*0.6*0.35 , -\a*0.8 + \b*0.6 - \c*0.6*0.6 - \a*0.8*0.35) node {$\gamma_{2}$};
\draw (0, -0.35) node {$p_{1}$};
\draw (\a*0.6 + \b*0.8, -\a*0.8 + \b*0.6+0.4) node {$p_{2}$}; 

\draw (-\b*0.8 + \a*0.1 -0.2, \b*0.6 - \a*0.7+0.1) node {$x$};

\begin{scope}[shift={(2*\a*0.6 + 2*\b*0.8, 0)}]

\draw[densely dashed] (-\c*0.8, \c*0.6) -- (-\c*0.8 -0.5*\a*0.6, \c*0.6 - 0.5*\a*0.8);
\draw[densely dashed] (\c*0.8, \c*0.6) -- (\c*0.8 + 0.5*\a*0.6, \c*0.6 - 0.5*\a*0.8);
\draw[densely dashed] (\a*0.6 + \b*0.8 - \c*0.8, -\a*0.8 + \b*0.6 - \c*0.6) -- (0.5*\a*0.6 + \b*0.8 - \c*0.8, -0.5*\a*0.8 + \b*0.6 - \c*0.6);

\draw[densely dashed, shift={(\b*0.8+\a*0.6,  \b*0.6 - \a*0.8 )}, rotate=180](-\c*0.8, \c*0.6) .. controls (\a*0.15 - \b*0.2 - \c*0.8, \c*0.6-\a*0.2) and (-\b*0.8 + \a*0.1 + \c*0.4, \b*0.6 - \a*0.7 - \c*0.3) .. (-\b*0.8 + \a*0.1, \b*0.6 - \a*0.7);
\draw[thick, <->, shift={(\b*0.8+\a*0.6,  \b*0.6 - \a*0.8 )}, rotate=180, shift={(-\c*0.8, \c*0.6)}, rotate=20] (-0.33, 0) arc (180:295:0.33);

\draw[very thick, decoration={markings, mark=at position 0.28 with {\draw (-0.2, 0.07) -- (0, 0) -- (-0.2, -0.07);}, mark=at position 0.78 with {\draw (-0.2, 0.07) -- (0, 0) -- (-0.2, -0.07);}}, postaction={decorate}] (-\c*0.8, \c*0.6) -- (0, 0) -- (\c*0.8, \c*0.6);
\draw[very thick, decoration={markings, mark=at position 0.28 with {\draw (-0.2, 0.07) -- (0, 0) -- (-0.2, -0.07);}, mark=at position 0.78 with {\draw (-0.2, 0.07) -- (0, 0) -- (-0.2, -0.07);}}, postaction={decorate}] (\a*0.6 + \b*0.8 - \c*0.8, -\a*0.8 + \b*0.6 - \c*0.6) -- (\a*0.6 + \b*0.8, -\a*0.8 + \b*0.6) -- (\a*0.6 + \b*0.8 + \c*0.8, -\a*0.8 + \b*0.6 - \c*0.6);
\draw[rotate=36.869897645844021] (0, 0) arc (-180:-90:0.5*\a) -- (\b - 0.5*\a, -0.5*\a) arc (90:0:0.5*\a);
\draw[shift={(\a*0.6 + \b*0.8, -\a*0.8 + \b*0.6)}, rotate=-36.869897645844021] (0, 0) arc (180:90:0.5*\a) -- (\b - 0.5*\a, 0.5*\a);

\draw (-\c*0.8*0.6 + \a*0.6*0.35, \c*0.6*0.6 + \a*0.8*0.35) node {$\eta_{3}$}; 
\draw (\c*0.8*0.6 - \a*0.6*0.35, \c*0.6*0.6 + \a*0.8*0.35) node {$\gamma_{3}$}; 
\draw (\a*0.6 + \b*0.8-\c*0.8*0.6 + \a*0.6*0.35 , -\a*0.8 + \b*0.6 - \c*0.6*0.6 - \a*0.8*0.35) node {$\eta_{4}$};
\draw (\a*0.6 + \b*0.8+\c*0.8*0.6 - \a*0.6*0.35 , -\a*0.8 + \b*0.6 - \c*0.6*0.6 - \a*0.8*0.35) node {$\gamma_{4}$};

\draw (0, -0.37) node {$p_{3}$};
\draw (\a*0.6 + \b*0.8, -\a*0.8 + \b*0.6+0.4) node {$p_{4}$};

\draw[shift={(\b*0.8+\a*0.6,  \b*0.6 - \a*0.8 )}, rotate=180] (-\b*0.8 + \a*0.1 -0.2, \b*0.6 - \a*0.7+0.1) node {$y$};

\end{scope}
%\draw[thick] (-\c*0.8 + 2*\a*0.6 + 2*\b*0.8, \c*0.6) -- (2*\a*0.6 + 2*\b*0.8, 0) -- (\c*0.8+2*\a*0.6 + 2*\b*0.8, \c*0.6);
%\draw[thick] (\a*0.6 + \b*0.8 - \c*0.8 + 2*\a*0.6 + 2*\b*0.8, -\a*0.8 + \b*0.6 - \c*0.6) -- (\a*0.6 + \b*0.8+ 2*\a*0.6 + 2*\b*0.8, -\a*0.8 + \b*0.6) -- (\a*0.6 + \b*0.8 + \c*0.8+ 2*\a*0.6 + 2*\b*0.8, -\a*0.8 + \b*0.6 - \c*0.6);

\end{tikzpicture}
\caption{Alignment and marking. Here, $\left(\gamma_{i}\right)_{i=1}^{4}$ and $\left(\eta_{i}\right)_{i=1}^{4}$ are $D$-aligned and $[x, y]$ is $(C, D)$-marked with $\left(\gamma_{i}\right)_{i=1}^{4}$, $\left(\eta_{i}\right)_{i=1}^{4}$. We also say that $[p_{1}, y]$ is $(C, D)$-head-marked with $\left(\gamma_{i}\right)_{i=1}^{4}$ and $\left(\eta_{i}\right)_{i=2}^{4}$. Similarly, we say that $[x, p_{4}]$ is $(C, D)$-tail-marked with $\left(\gamma_{i}\right)_{i=1}^{3}$, $\left(\eta_{i}\right)_{i=1}^{4}$.}
\label{fig:scheme2}
\end{figure}
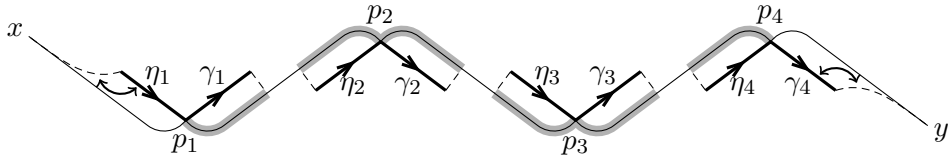

This definition is designed for recording the alignment of points in a cumulative way; once recorded, the Gromov products among points can be controlled via the following lemmata. Lemma \ref{lem:concat} appeared in \cite{baik2021linear} with $\gamma_{1}, \gamma_{2}, \eta$ being the progresses made by certain pseudo-Anosov mapping classes. Also assumed there was that $\gamma_{1}$ begins at $x$ and $\bar{\gamma}_{2}$, $\eta$ are glued at $y$. We present the proof of Lemma \ref{lem:concat} to remove such restrictions.

\begin{lem}[Propagation of small products, {\cite[Lemma 4.18]{baik2021linear}}]\label{lem:concat}
For each $D, \epsilon > 0$, there exist $E, L > D$ that satisfy the following property. Let $x \in X_{\ge \epsilon}$, $y, z \in X$ and $\gamma_{1}$, $\gamma_{2}$, $\eta$ be $\epsilon$-thick segments that are longer than $L$. Suppose that $\bar{\gamma}_{2}$ and $\eta$ are $D$-glued.
If $[x, y]$ is $D$-witnessed by $(\gamma_{1}, \gamma_{2})$ and $[y, z]$ is $E$-witnessed by $\eta$, then $[x, z]$ is $E$-witnessed by $\gamma_{1}$.
\end{lem}

\begin{proof}
Let $\bar{\gamma}_{2}$, $\eta$ be glued at $y' \in X$ and $\gamma_{1} = [x_{0}, y_{0}]$.

When $X$ is a $\delta$-hyperbolic space, we set $E = D + 4\delta$ and $L = 4D + 6\delta + 1$. First observe that \[\begin{aligned}
(\eta, z)_{\ast} \ge L - (\bar{\eta}, z)_{\ast} \ge L - E \ge D + 2\delta + 1,\\
(\bar{\gamma}_{2}, y_{0})_{\ast} \ge L - (\gamma_{2}, y_{0})_{\ast} \ge L - D \ge D + 2\delta + 1.
\end{aligned}
\]
Given these, the Gromov inequality\[
\min\{(\bar{\gamma}_{2}, y_{0})_{\ast}, (y_{0}, \eta)_{\ast}\} - \delta \le (\bar{\gamma}_{2}, \eta)_{\ast} \le D 
\]
implies $(y_{0}, \eta)_{\ast} \le D + \delta$ and \[
\min\{(\eta, z)_{\ast}, (z, y_{0})_{y'}\} - \delta \le (\eta, y_{0})_{\ast} \le D + \delta
\]
implies $(z, y_{0})_{y'} \le D + 2\delta$. Then Fact \ref{fact:GromProdFact} implies \[\begin{aligned}
(y, z)_{y_{0}} & = d(y_{0}, y') + (y, z)_{y'} - (y, y_{0})_{y'} - (z, y_{0})_{y'}\\
& \ge d(y_{0}, y') - 2D - 2\delta \ge L - 3D - 2\delta \ge D + 3\delta + 1.
\end{aligned}
\]
Hence, together with the result of Lemma \ref{lem:wideWitness} that \[
\min\{(x, z)_{y_{0}}, (z, y)_{y_{0}}\} - \delta \le (x, y)_{y_{0}} \le D + 2\delta,
\]
we deduce $(x, z)_{y_{0}} \le D + 3\delta$. Since \[
(\bar{\gamma}_{1}, x)_{\ast} \ge L - (\gamma_{1}, x) \ge L - D \ge D + 4 \delta + 1,
\] this implies that $(\bar{\gamma}_{1}, z)_{\ast} \le D + 4\delta$. This in turn implies that \[
(\gamma_{1}, z)_{\ast} \ge L - (\bar{\gamma}_{1}, z)_{\ast} \ge L - D - 4\delta \ge D + \delta + 1.\] Since we assumed $(x, \gamma_{1})_{\ast} \le D$, we deduce that $(x, z)_{x_{0}} \le D + \delta$.

Now let $X$ be Teichm{\"u}ller space. We take \[
\begin{aligned}
\epsilon_{1} &= \epsilon e^{-2D},\\
E&= \mathscr{B}(\epsilon_{1}, \mathscr{D}(\epsilon_{1})) + D,\\
L &= 5D+ 2E +2\mathscr{C}(\epsilon_{1}) + 2\mathscr{D}(\epsilon_{1}) + 2.
\end{aligned}
\]

Let $p, p', q, q' \in [x, y]$ be such that $[p, p']$ and $\gamma_{1}$ $D$-fellow travel, $[q, q']$ and $\gamma_{2}$ $D$-fellow travel, and $[p, p']$ appears earlier than $[q, q']$. Then $d(y', q') < D$ and $[p, p']$, $[q, q']$ are $\epsilon_{1}$-thick. Since $d(q', q) \ge L - 2D$, we have a subsegment $[q_{1}, q_{2}]$ of $[q', q]$ with $d(q', q_{1}) = 3D + E + \mathscr{C}(\epsilon_{1}) + \mathscr{D}(\epsilon_{1})+1$ and $d(q_{1}, q_{2}) = \mathscr{C}(\epsilon_{1})$. 

Suppose that there exist points $a \in [q_{1}, q_{2}]$ and $b \in [y, z]$ that are within distance $\mathscr{D}(\epsilon_{1})$. We then observe that \[\begin{aligned}
d(y, y') + E &\le d(y, y') + [3D + E + \mathscr{C}(\epsilon_{1}) + \mathscr{D}(\epsilon_{1}) +  1] - D - \mathscr{D}(\epsilon_{1})\\
&\le [d(y, y')- d(y', q')] + d(q', q_{1}) - \mathscr{D}(\epsilon_{1})\\& \le d(y, q') + d(q', a) - \mathscr{D}(\epsilon_{1}) \\& \le  d(y, a) - d(a, b) \le d(y, b) \le d(y, a) + d(a, b) \\
&\le d(y, q') + d(q', a) + \mathscr{D}(\epsilon_{1})\\& \le d(y, y') + d(y', q') + d(q', q_{1} )+ d(q_{1}, q_{2}) + \mathscr{D}(\epsilon_{1}) \\
&\le d(y, y') + D + [3D + E + \mathscr{C}(\epsilon_{1}) + \mathscr{D}(\epsilon_{1}) + 1] +\mathscr{C}(\epsilon_{1}) + \mathscr{D}(\epsilon_{1})\\&\le d(y, y') + L - E.
 \end{aligned}
\]
This implies that $b$ belongs to a subsegment of $[y, z]$ that $E$-fellow travels with $\eta$, so $d(b, b') \le E$ holds for some $b' \in \eta$. Further, since $[q, q']$ and $\gamma_{2}$ $D$-fellow travel and $a \in [q, q']$, there exists $a' \in \gamma_{2}$ such that $d(a, a') \le D$. Then Fact \ref{fact:GromProdFact2} implies that \[\begin{aligned}
(\bar{\gamma}_{2}, \eta)_{\ast} &\ge  d(y', a') - d(a', b') \\
& \ge [d(q', q_{1}) - d(y', q') - d(q_{1}, a) - d(a, a')] - [d(a', a) + d(a, b) + d(b, b')] \\
&\ge d(q', q_{1}) - 3D - E - \mathscr{C}(\epsilon_{1}) -  \mathscr{D}(\epsilon_{1}) > D,
\end{aligned}
\]
a contradiction. Thus, there instead exist points $a \in [q_{1}, q_{2}]$ and $b \in [x, z]$ that are within distance $\mathscr{D}(\epsilon_{1})$. Moreover, $x, a$ are $\epsilon_{1}$-thick. By Theorem \ref{thm:rafi1},  $[x, b]$ and $[x, a]$ $\mathscr{B}(\epsilon_{1}, \mathscr{D}(\epsilon_{1}))$-fellow travel. Since $[x, a]$ contains $[p, p']$ that $D$-fellow travels with $\gamma_{1}$, we deduce that a subsegment of $[x, b]$ and $\gamma_{1}$ $E$-fellow travel.
\end{proof}

\begin{lem}[Witness in the middle]\label{lem:concatUlt}
For each $E$, $\epsilon > 0$, there exist $F, L> E$ that satisfies the following property. Let $x, y, y', z \in X $ and let $\gamma$, $\gamma'$ be $\epsilon$-thick geodesic segments that are longer than $L$ and $E$-glued at $y'$. If $[y, x]$ is $E$-witnessed by $\gamma$ and $[y, z]$ is $E$-witnessed by $\gamma'$, then $[x, z]$ is $F$-witnessed by $\bar{\gamma}$ and by $\gamma'$. In particular, $|(x, z)_{y}- d(y, y')| < F$.
\end{lem}

\begin{proof}
When $X$ is a $\delta$-hyperbolic space, we set $F = 2E + 2\delta$ and $L = 2E + 6\delta + 1$. 
Let $x', z' \in X$ be such that $\gamma = [y', x']$ and $\gamma' = [y', z']$.

Since $[y, z]$ is $E$-witnessed by $\gamma'$, whose length is at least $L$, we have \[
(z', z)_{y'} = d(y', z') - (y', z)_{z'} \ge L - E \ge E + 3\delta+1.
\]
Similarly, we deduce $(x, x')_{y'} \ge E + 3\delta + 1$. Now the Gromov inequality tells us that 
\[
\min\{ (x', x)_{y'}, (x, z')_{y'}\} -\delta \le (x', z')_{y'} \le E,
\]
which forces $(x, z')_{y'} \le E + \delta$. We then have
\[
\min\{(x, z)_{y'}, (z,z')_{\ast}\} - \delta \le (x, z')_{y'} \le E + \delta,
\]
which implies that $(x, z)_{y'} \le E + 2\delta$. To show that $[x, z]$ is $F$-witnessed by $\gamma' = [y', z']$ it suffices to observe \[
(x, z)_{z'} = (x, z)_{y'} + (z, y')_{z'} - (x, z')_{y'} \le (E+2\delta) + E = 2E + 2\delta.
\]Similarly we deduce that $[x, z]$ is $F$-witnessed by $\bar{\gamma}$.

When $X$ is Teichm{\"u}ller space, we take \[
\begin{aligned}
\epsilon_{1} &= \epsilon e^{-2E},\\
L &=8E + 2\mathscr{C}(\epsilon_{1}) + 2\mathscr{D}(\epsilon_{1}) + 2,\\
F_{1} &= \mathscr{B}(\epsilon_{1}, \mathscr{D}(\epsilon_{1})) + 2(5E +\mathscr{C}(\epsilon_{1}) + \mathscr{D}(\epsilon_{1}) + 1), \\
F &= F_{1} + 2E.
\end{aligned}
\]
Let $q, q'\in [x, y]$ be such that $[q, q']$ and $\bar{\gamma}$ $E$-fellow travel. Note that $[q, q']$ is $\epsilon_{1}$-thick. Since $d(q', q) \ge L - 2E$, there exists a subsegment $[q_{1}, q_{2}]$ of $[q', q]$ such that $d(q', q_{1}) \ge 5E + \mathscr{C}(\epsilon_{1}) + \mathscr{D}(\epsilon_{1})+1$ and $d(q_{1}, q_{2}) = \mathscr{C}(\epsilon_{1})$. 

Suppose that there exist points $a \in [q_{1}, q_{2}]$ and $b \in [y, z]$ that are within distance $\mathscr{D}(\epsilon_{1})$. We then observe that \[\begin{aligned}
d(y, y') + E &\le d(y, y') + [5E + \mathscr{C}(\epsilon_{1}) + \mathscr{D}(\epsilon_{1}) + 1] - E - \mathscr{D}(\epsilon_{1})\\
&\le [d(y, y')- d(y', q')] + d(q', q_{1}) - \mathscr{D}(\epsilon_{1})\\& \le d(y, q') + d(q', a) - \mathscr{D}(\epsilon_{1}) \\& \le  d(y, a) - d(a, b) \le d(y, b) \le d(y, a) + d(a, b) \\
&\le d(y, q') + d(q', q_{2}) + \mathscr{D}(\epsilon_{1})\\& \le d(y, y') + d(y', q') + d(q', q_{1}) + d(q_{1}, q_{2}) + \mathscr{D}(\epsilon_{1}) \\
&\le d(y, y') + E + [5E +\mathscr{C}(\epsilon_{1}) +  \mathscr{D}(\epsilon_{1}) + 1] +\mathscr{C}(\epsilon_{1}) + \mathscr{D}(\epsilon_{1})\le d(y, y') + L - E.
 \end{aligned}
\]
This implies that $b$ belongs to a subsegment of $[y, z]$ that $E$-fellow travels with $\gamma'$, and $d(b, b') \le E$ for some $b' \in \gamma'$. Similarly, we take $a' \in \gamma$ such that $d(a, a') \le E$. Then Fact \ref{fact:GromProdFact2} implies that \[\begin{aligned}
(\gamma, \gamma')_{\ast} &\ge  d(y', a') - d(a', b') \\
&\ge [d(q', q_{1}) - d(q', y') - d(q_{1}, a) - d(a, a')] - [d(a', a) + d(a, b) + d(b, b')] \\
&\ge d(q', q_{1}) - 4E - \mathscr{C}(\epsilon_{1}) - \mathscr{D}(\epsilon_{1}) > E,
\end{aligned}
\]
a contradiction. Thus, there instead exist points $a \in [q_{1}, q_{2}]$ and $b \in [x, z]$ that are within distance $\mathscr{D}(\epsilon_{1})$. 

The above argument gives points $a_{1}, a_{2} \in [q, q']$ and $b_{1}, b_{2} \in [x, z]$ such that $d(a_{i}, b_{i}) \le \mathscr{D}(\epsilon_{1})$ and $d(q, a_{1}), d(q', a_{2}) \le 5E + 2\mathscr{C}(\epsilon_{1}) + \mathscr{D}(\epsilon_{1}) + 1$. Since $a_{i}$'s are $\epsilon_{1}$-thick, $[a_{1}, a_{2}]$ and $[b_{1}, b_{2}]$ $\mathscr{B}(\epsilon_{1}, \mathscr{D}(\epsilon_{1}))$-fellow travel. This then implies that $[q, q']$ and $[b_{1}, b_{2}]$ $F_{1}$-fellow travel. Therefore, $\bar{\gamma}$ and $[b_{1}, b_{2}]$ $F$-fellow travel. For a  similar reason, $[x, z]$ is $F$-witnessed by $\gamma'$.

In both cases, $(x, y)_{y'}, (y, z)_{y'}< E$ and $(z, x)_{y'} < F$. Thus Fact \ref{fact:GromProdFact2} implies \[\begin{aligned}
d(y, y') - F &\le d(y, y') - 2E\\
& \le (x, z)_{y} = (x, z)_{y'} + d(y, y') - (x, y)_{y'} - (z, y)_{y'}\\
& \le d(y, y') + F. \qedhere
\end{aligned}
\]

\end{proof}

Lemma \ref{lem:concat} and Lemma \ref{lem:concatUlt} together imply the following.

\begin{cor}\label{cor:induction}
Let $D, M, \epsilon > 0$ and \begin{itemize}
\item $E = E(\epsilon, D)$, $L_{1} = L(\epsilon, D)$ as in Lemma \ref{lem:concat}, and 
\item $F = F(\epsilon, E)$, $L_{2} = L(\epsilon, E)$ as in Lemma \ref{lem:concatUlt}.
\end{itemize}
Let also $\left(p_{i}\right)_{i=0}^{N+1}$ be points on $X_{\ge \epsilon}$ and $\left(\gamma_{i}\right)_{i=1}^{N}$, $\left(\eta_{i}\right)_{i=1}^{N}$ be segments on $X_{\ge \epsilon}$. Suppose that: \begin{enumerate}
\item $\gamma_{i}$, $\eta_{i}$ are longer than $\max(L_{1}, L_{2}, M+2F+3D + 2\delta)$;
\item $\left(\gamma_{i}\right)_{i=1}^{N}$, $\left(\eta_{i}\right)_{i=1}^{N}$ are $D$-aligned and glued at $\left(p_{i}\right)_{i=1}^{N}$, and 
\item $[p_{0}, p_{N+1}]$ is $(D, D)$-marked with $\left(\gamma_{i}\right)$, $\left(\eta_{i}\right)$.
\end{enumerate}
Then we have: \begin{enumerate}
\item $d(p_{i}, p_{i+1}) > M + 2F$ for $0 \le  i \le N$,
\item $[p_{i}, p_{k}]$ is $F$-witnessed by $\gamma_{j}$ and $\eta_{j}$ for $0 \le i <j < k\le N+1$,
\item $(p_{i}, p_{k})_{p_{j}} < F$ for $0 \le i<j<k \le N+1$, and
\item $d(p_{i}, p_{l}) \ge d(p_{j}, p_{k}) + M(j-i) + M(l-k)$ for $0 \le i \le j \le k \le l \le N+1$.
\end{enumerate}
\end{cor}

\begin{proof}
For $i = 0, \ldots, N-1$, $[p_{i}, p_{i+1}]$ is $D$-witnessed by $\eta_{i+1}$. Note that the length of $\eta_{i+1}$ is at least $M + 2F + 3D + 3\delta> 3D + 3\delta$. Hence, when $X$ is a $\delta$-hyperbolic space, we have \[
d(p_{i}, p_{i+1}) > \operatorname{diam}(\eta_{i+1}) - 3D - 2\delta > M + 2F
\] by Lemma \ref{lem:wideWitness}. When $X$ is Teichm{\"u}ller space, $[p_{i}, p_{i+1}]$ has a subsegment that $D$-fellow travels with $\eta_{i}$. This implies that $d(p_{i}, p_{i+1} > \operatorname{diam}(\eta_{i}) - 2D \ge M + 2F$. Moreover, $[p_{N}, p_{N+1}]$ has a subsegment that $D$-fellow travels with $\gamma_{N}$ and we similarly deduce that $d(p_{N}, p_{N+1})> M+2F$.

Our next goal is to show that $[p_{j}, p_{k}]$ is $E$-witnessed by $\gamma_{j}$ for $1 \le j < k \le N+1$. We prove this by inducting on $k-j$. When $k-j = 1$, $[p_{j}, p_{k}]$ is assumed to be $D$-witnessed by $\gamma_{j}$. Now given $1 < j < k \le N+1$ such that $[p_{j}, p_{k}]$ is $E$-witnessed by $\gamma_{j}$, we claim that $[p_{j-1}, p_{k}]$ is $E$-witnessed by $\gamma_{j-1}$. Note that: \begin{itemize}
\item $[p_{j-1}, p_{j}]$ is $D$-witnessed by $(\gamma_{j-1}, \eta_{j})$; 
\item $\bar{\eta}_{j}$ and $\gamma_{j}$ are $D$-glued, and 
\item $\bar{\eta}_{j}, \gamma_{j}$ are $\epsilon$-thick segments that are longer than $L_{1} = L(\epsilon, D)$ as in Lemma \ref{lem:concat}.
\end{itemize}
Lemma \ref{lem:concat} then guarantees the claim and completes the induction. 

Similarly, we observe that $[p_{i}, p_{j}]$ is $E$-witnessed by $\eta_{j}$ for $0 \le i < j \le N$. Now for $0 \le i < j < k \le N+1$, we have that:
\begin{itemize}
\item $[p_{j}, p_{i}]$ is $E$-witnessed by $\eta_{j}$;
\item$[p_{j}, p_{k}]$ is $E$-witnessed by $\gamma_{j}$;
\item $\gamma_{j}$ and $\eta_{j}$ are $E$-glued, and 
\item $\bar{\eta}_{j}, \gamma_{j}$ are $\epsilon$-thick segments that are longer than $L_{2} = L(\epsilon, E)$ as in Lemma \ref{lem:concatUlt}.
\end{itemize}
Then Lemma \ref{lem:concatUlt} then guarantees that $[p_{i}, p_{k}]$ is $F$-witnessed by $\eta_{j}$ and $\gamma_{j}$, as desired. In particular, $[p_{i}, p_{k}]$ passes through the $F$-neighborhood of $p_{j}$ and Fact \ref{fact:GromProdFact2} tells us that $(p_{i}, p_{k})_{p_{j}} < F$. This implies that \[\begin{aligned}
d(p_{i}, p_{j+1}) &= d(p_{i}, p_{j}) + d(p_{j}, p_{j+1}) - 2(p_{i}, p_{j+1})_{p_{j}}\\
& >d(p_{i}, p_{j}) + (M+2F) - 2F = d(p_{i}, p_{j}) + M
\end{aligned}
\]
for $0 \le i \le j \le N$. Similarly, we have $d(p_{i-1}, p_{j}) > d(p_{i}, p_{j}) + M$ for $1 \le i \le j \le N+1$. Applying these two inequalities inductively, we achieve the fourth item of the conclusion.
\end{proof}

\begin{lem}[Witness copied]\label{lem:passerBy}
For each $F, \epsilon > 0$, there exist $G, L > F$ that satisfy the following condition. If $x, y, z, p_{1}, p_{2}$ in $X$ satisfy that: \begin{enumerate}
\item $[p_{1}, p_{2}]$ is $\epsilon$-thick and longer than $L$,
\item $[x, y]$ is $F$-witnessed by $[p_{1}, p_{2}]$, and
\item $(x, z)_{y} \ge d(p_{1}, y)-F$,
\end{enumerate}
then $[z, y]$ is $G$-witnessed by $[p_{1}, p_{2}]$.
\end{lem}

Note that we do not require vertices $x$, $y$, $z$ to be $\epsilon$-thick.

\begin{proof}
When $X$ is a $\delta$-hyperbolic space, we set $G = 3F + 2\delta$ and $L = 4F + 3\delta + 1$. We first observe \[\begin{aligned}
(x, z)_{p_{2}}& \ge (x, z)_{y} - d(y, p_{2}) \ge d(y, p_{1}) - d(y, p_{2}) - F \\
&= d(p_{1}, p_{2}) -2(y, p_{1})_{p_{2}}- F> d(p_{1} p_{2}) - 3F,\\
(x, p_{1})_{p_{2}} &= d(p_{1}, p_{2}) - (x, p_{2})_{p_{1}} \ge d(p_{1}, p_{2}) - F.
\end{aligned}
\]
We then have \[\begin{aligned}
(z, p_{1})_{p_{2}} &\ge \min \{(z, x)_{p_{2}}, (x, p_{1})_{p_{2}} \} - \delta \ge  d(p_{1}, p_{2}) - 3F - \delta \ge F + \delta + 1,\\
(z, p_{2})_{p_{1}} &= d(p_{1}, p_{2}) -(z, p_{1})_{p_{2}}   \le 3F + \delta.
\end{aligned}\] 
Now note the Gromov inequality \[
(z, p_{2})_{p_{1}} \ge \min\{(z, y)_{p_{1}}, (y, p_{2})_{p_{1}} \} - \delta.
\]
Since $(p_{2}, y)_{p_{1}} \ge d(p_{1}, p_{2}) - F \ge 3F + 2\delta + 1$, we deduce that $(z, y)_{p_{1}} \le 3F + 2\delta$. Also, $(p_{1}, y)_{p_{2}} \le F$ and $(z, p_{1})_{p_{2}} \ge F+ \delta + 1$ implies $(z, y)_{p_{2}} \le F + \delta$.

When $X$ is Teichm{\"u}ller space, we take \[
\begin{aligned} 
\epsilon_{1} &=\epsilon e^{-2F},\\
L &= 4F + \mathscr{D}(\epsilon_{1}) + \mathscr{C}(\epsilon_{1}) + 1,\\
G &= \mathscr{B}(\epsilon_{1}, \mathscr{D}(\epsilon_{1})) + L.
\end{aligned}
\]

Let $[x', y']$ be a subsegment of $[x, y]$ that $F$-fellow travels with $[p_{1}, p_{2}]$. Note that $[x', y']$ is $\epsilon_{1}$-thick and \[\begin{aligned}
(x, z)_{y} &\ge d(p_{1}, y) - F \ge d(x', y) - 2F, \\
d(x', y') &\ge L - 2F \ge 2F + \mathscr{D}(\epsilon_{1}) + \mathscr{C}(\epsilon_{1}) + 1.
\end{aligned}
\] Let us now take a subsegment $[x'', y'']$ of $[x', y']$ such that $d(x', x'') \ge 2F + \mathscr{D}(\epsilon_{1}) + 1$ and $d(x'', y'') = \mathscr{C}(\epsilon_{1})$. Suppose that there exist points $a \in [x'', y'']$ and $b \in [x, z]$ that are within distance $\mathscr{D}(\epsilon_{1})$. Then by Fact \ref{fact:GromProdFact2}, \[
(y, z)_{x} \ge d(a, x) - d(a, b) \ge d(x, x') +  d(x', x'') - \mathscr{D}(\epsilon_{1}) \ge d(x, x') + 2F+1
\]
holds. This implies that \[
d(x, y) = (y, z)_{x} + (x, z)_{y} > d(x, x') + d(x', y) = d(x, y),
\] a contradiction. Having this, Theorem \ref{thm:rafi2} implies that there exist $a \in [x'', y'']$ and $b \in [y, z]$ that are within distance $\mathscr{D}(\epsilon_{1})$. 

The previous argument provides us with points $a_{1}, a_{2} \in [x', y']$ and $b_{1}, b_{2} \in [y, z]$ such that \begin{itemize}
\item $d(a_{i}, b_{i}) \le \mathscr{D}(\epsilon_{1})$ for $i=1, 2$;
\item $d(x', a_{1}) \le 2F + \mathscr{D}(\epsilon_{1}) + \mathscr{C}(\epsilon_{1}) + 1$, and 
\item $d(y', a_{2}) \le \mathscr{C}(\epsilon_{1})$.
\end{itemize}
Now, $[a_{1}, a_{2}] \subseteq [x', y']$ and $[b_{1}, b_{2}] \subseteq [y, z]$ have pairwise $\mathscr{D}(\epsilon_{1})$-near endpoints. Moreover, $a_{i}$'s are $\epsilon_{1}$-thick. Theorem \ref{thm:rafi1} then tells us that $[a_{1}, a_{2}]$ and $[b_{1}, b_{2}]$ $\mathscr{B}(\epsilon_{1}, \mathscr{D}(\epsilon_{1}))$-fellow travel. Also, the bounds on $d(x', a_{1})$ and $d(y', a_{2})$ tell us that $[x', y']$ and $[a_{1}, a_{2}]$ $(2F + \mathscr{C}(\epsilon_{1}) + \mathscr{D}(\epsilon_{1}) + 1)$-fellow travel. Finally, $[x', y']$ and $[p_{1}, p_{2}]$ $F$-fellow travel. Combining all these, we conclude that $[b_{1}, b_{2}]$ and $[p_{1}, p_{2}]$ are $G$-fellow traveling.
\end{proof}

\section{Pivotal times}

\subsection{Schottky sets and pivots} \label{subsection:Schottky}

From now on, we fix a basepoint $o \in X$ once and for all. We recall the following definition of Schottky set in \cite{gouezel2022exponential}, which originates from \cite{boulanger2022large}.

\begin{definition}[cf. {\cite[Definition 3.11]{gouezel2022exponential}}]
Let $K, K', \epsilon> 0$. A finite set $S$ of isometries of $X$ is said to be \emph{$(K, K')$-Schottky} if the following hold: \begin{enumerate}
\item for all $x, y \in X$, $|\{s \in S : (x, s^{i}y)_{o} \ge K\,\,\textrm{for some}\,\,i>0\}| \le2$;
\item for all $x, y \in X$, $|\{s \in S : (x, s^{i}y)_{o} \ge K\,\,\textrm{for some}\,\,i<0\}| \le2$;
\item for all $s_{1}, s_{2}\in S$ and $i, j > 0$, we have $(s_{1}^{-i} o, s_{2}^{j})_{o} \le K$;
\item for all $s \in S$ and $i \neq 0$, $d(o, s^{i}o) \ge K'$.
\end{enumerate}
When $X$ is Teichm{\"u}ller space, $S$ is said to be \emph{$(K, K', \epsilon)$-Schottky} if the following condition holds in addition to the above three: \begin{enumerate}\setcounter{enumi}{3}
\item for all $s \in S$ and $i \in \Z$, the geodesic $[o, s^{i}o]$ is $\epsilon$-thick.
\end{enumerate}
\end{definition}

Note that any subset of a Schottky set is still a Schottky set. We now present the main result of this subsection.

\begin{prop}[cf. {\cite[Proposition 3.12]{gouezel2022exponential}}]\label{prop:Schottky}
Let $a, b$ be independent loxodromic isometries of $X$. Then there exist $K, \epsilon > 0$ such that for each $K' > 0$, there exist $n \in \N$ and a $(K, K')$-Schottky set of cardinality at least 310 in $\{w_{1} \cdots w_{n} : w_{i} \in \{a, b\}\}$. When $X$ is Teichm{\"u}ller space, it can be chosen as a $(K, K', \epsilon)$-Schottky set .
\end{prop}

The proof of Proposition \ref{prop:Schottky} for $\delta$-hyperbolic spaces is given in \cite{gouezel2022exponential}. We now suppose that $X$ is Teichm{\"u}ller space. Let $S_{0} = \{a, a^{-1}, b, b^{-1}\}$. We recall two lemmata from our earlier work:

\begin{lem}[cf. {\cite[Lemma 4.8]{baik2021linear}}]\label{lem:uniformThickAxis}
There exists a constant $M_{1} > 0$ such that for $\phi \in S_{0}$, the Hausdorff distance between $\{\phi^{i} o\}_{i=0}^{n}$ and $[o, \phi^{n} o]$ is bounded by $M_{1}$. Consequently, there exists a constant $\epsilon_{0} >0$ so that $[o, \phi^{n}o]$ is $\epsilon_{0}$-thick for all $n \in \Z$.
\end{lem}

\begin{lem}[{\cite[Lemma 4.11]{baik2021linear}}]\label{lem:pAGromov}
There exists a constant $M_{2}$ such that: \begin{enumerate}
\item $(a^{m} o, a^{n}o)_{o}, (b^{m} o, b^{n}o)_{o}\le M_{2}$ for $m \ge  0$ and $n \le 0$, and
\item $(a^{n} o, b^{m} o)_{o} \le M_{2}$ for all $n, m \in \Z$.
\end{enumerate}
\end{lem}

\begin{proof}[Proof of Proposition \ref{prop:Schottky}]
Let $M_{1}, \epsilon_{0}$ be as in Lemma \ref{lem:uniformThickAxis} and $M_{2}$ be as in Lemma \ref{lem:pAGromov}. We fix $M' = \max(M_{1}, M_{2})$ and set the following: \begin{itemize}
\item $D_{0} = D(C = M', \epsilon_{0})$ as in Lemma \ref{lem:1segment};
\item $E_{0} = E(D = D_{0}, \epsilon_{0})$, $L_{0} =L(D = D_{0}, \epsilon_{0})$ as in Lemma \ref{lem:concat};
\item $F_{0} = F(E = E_{0}, \epsilon_{0})$, $L_{1} = L(E = E_{0}, \epsilon_{0})$ as in Lemma \ref{lem:concatUlt};
\item $G_{0} = G(F = F_{0}, \epsilon_{0})$, $L_{2} = L(F = F_{0}, \epsilon_{0})$ as in Lemma \ref{lem:passerBy}; 
\item $G_{1} = G(F = G_{0}, \epsilon_{0})$, $L_{3} = L(F = G_{0}, \epsilon_{0})$ as in Lemma \ref{lem:passerBy}; 
\item $F_{1} = 2F_{0} + G_{1} +2M' + 1$;
\item $F_{2} = F(E = G_{0}, \epsilon_{0})$ and $L_{4} = L(E = G_{0}, \epsilon_{0})$ in Lemma \ref{lem:concatUlt}; 
\item $L_{5} = \max( L_{0}, L_{1}, L_{2}, L_{3}, L_{4}, 6D_{0} + 2F_{0}+ F_{1} + 2F_{2} + 8\delta+2)$. 
\end{itemize}

Since $\tau(a), \tau(b)>0$, there exists $N$ such that $d(o, \phi^{n} o) > L_{5}$ for all $\phi \in S_{0}$ and $n \ge N$. Note that $(o, \phi^{N} o)_{\phi^{N} o} = 0$ and $(o, \phi^{2N} o)_{\phi^{N} o} \le M'$. Thus, $[o, \phi^{2N} o]$ is $D_{0}$-witnessed by $([o, \phi^{N} o], [\phi^{N}o, \phi^{2N}o])$ by Lemma \ref{lem:1segment}.

For sequences $(\phi_{i})_{i} \in S_{0}^{\Z}$ such that $\phi_{i}^{-1} \neq  \phi_{i+1}$, let us define \[
w_{i} = \phi_{1}^{2N} \cdots \phi_{i}^{2N}, \quad v_{i} = w_{i-1} \phi_{i}^{N}.
\] 

Note that $[o, w_{m} o]$ is fully $D_{0}$-marked with $([w_{i} o, v_{i+1} o])_{i=0}^{m-1}$ and $([v_{i} o, w_{i} o])_{i=1}^{m}$. Then by Corollary \ref{cor:induction}, $[o, w_{m} o]$ is $F_{0}$-witnessed by each $[o, v_{1}o]$, $[v_{1}o, w_{1}o]$, $\ldots$, $[v_{m} o, w_{m}o]$. This implies that $[o, w_{m} o]$ is $\epsilon_{1}$-thick for $\epsilon_{1} = \epsilon_{0}e^{-8F_{0}}$,\begin{equation}\label{eqn:GromSchottky}
(w_{i} o, w_{k} o)_{w_{j} o} \le F_{0}
\end{equation}
for all $i \le j \le k$, and
\begin{equation}\label{eqn:bitFarther}\begin{aligned}
d(o, w_{i+1} o) & = d(o, w_{i} o) + d(w_{i} o, w_{i+1}o)- 2(o, w_{i+1} o)_{w_{i} o} \\
& \ge d(o, w_{i} o) + \left[\begin{array}{c} d(w_{i}o, v_{i+1} o) + d(v_{i+1} o, w_{i+1} o) \\- 2(w_{i} o, w_{i+1} o)_{v_{i+1} o}\end{array}\right] - 2F_{0}\\
& \ge d(o, w_{i} o) + 2d(o, \phi_{i+1}^{N} o) - 2M' - 2F_{0} \\
& \ge d(o, w_{i} o) + d(o, \phi_{i+1}^{N} o) +F_{1} \ge d(o, v_{i+1} o)+F_{1}
\end{aligned}
\end{equation}
for each $i\ge 0$. In particular, the second last inequality yields $d(o, w_{i} o) \ge F_{1} i$.

We now define \[\begin{aligned}
S' &:= \{g_{1}, \ldots, g_{2^{10}}\} = \{\phi_{1}^{2N} \cdots \phi_{10}^{2N} : \phi_{i} \in \{a, b\}\},\\
V(g_{i}^{\pm}) &:= \{x \in X : (x, g_{i}^{\pm 2} o)_{o} \ge d(o, g_{i}^{\pm 1} o)-F_{1}\}, \\
V'(g_{i}^{\pm}) &:=  \{x \in X : (x, g_{i}^{\pm 2} o)_{o} \ge d(o, g_{i}^{\pm 1} o)\}.
\end{aligned}\]
Inequality \ref{eqn:GromSchottky} tells us that property (3) holds for $s_{1}, s_{2}\in S'$ and $K > F_{0}$.

Our first claim is that $V(g_{1}^{+}), \ldots, V(g_{2^{10}}^{+}), V(g_{1}^{-}), \ldots, V(g_{2^{10}}^{-})$ are all disjoint. To show this, let $h_{1} = \phi_{1}^{2N} \cdots \phi_{10}^{2N}$ and $h_{2} = \psi_{1}^{2N} \cdots \psi_{10}^{2N}$ be distinct elements among $\{g_{1}, \ldots, g_{2^{10}}, g_{1}^{-1}, \ldots, g_{2^{10}}^{-1}\}$, i.e., there exists $t \in \{0, \ldots, 9\}$ such that $\phi_{i} = \psi_{i}$ for $i \le t$ but $\phi_{t+1} \neq \psi_{t+1}$. If some $x$ belongs to $V(h_{1}) \cap V(h_{2})$, then \[
(x, h_{1}^{2} o)_{o} \ge d(o, w_{10} o) - F_{1} \ge d(o, w_{t+1} o) - F_{1} \ge d(o, v_{t+1} o)
\]
by Inequality \ref{eqn:bitFarther}. Similarly, we have $(x, h_{2}^{2} o)_{o} \ge d(o, w_{t} \psi_{t+1}^{N} o)$. Since $[o, h_{1}^{2} o]$ ($[o, h_{2}^{2} o]$, resp.) is $F_{0}$-witnessed by $[w_{t} o, v_{t+1} o]$ ($[w_{t} o, w_{t} \psi_{t+1}^{N} o]$, resp.), Lemma \ref{lem:passerBy} tells us that $[o, x]$ is $G_{0}$-witnessed by $[w_{t} o, v_{t+1}o]$ and $[w_{t} o, w_{t} \psi_{t+1}^{N} o]$. By Lemma \ref{lem:concatUlt}, $[x, x]$ is $F_{2}$-witnessed by $[w_{t}o, v_{t+1} o]$; this is impossible because the length of $[w_{t} o, v_{t+1} o]$ is at least $L_{5}$, greater than $2F_{2}$.

The next claim is that if $x \notin V(g_{i}^{-})$, then $g_{i}^{2} x \in V'(g_{i})$. Indeed, Corollary \ref{cor:induction} asserts that  $(o, g_{i}^{2} o)_{g_{i} o} \le F_{0} \le F_{1}/2$, which implies that  \[\begin{aligned}
(g_{i}^{2}x, g_{i}^{2} o)_{o} &= (x, o)_{g_{i}^{-2} o} = d(o, g_{i}^{-2} o) - (x, g_{i}^{-2} o)_{o}\\
& \ge d(o, g_{i}^{2} o) - d(o, g_{i} o) +F_{1} \ge d(o, g_{i}^{} o).
\end{aligned}
\]

Iterating this, we deduce that $g_{i}^{2k} x \in V'(g_{i})$ for $k > 0$. Similarly, if $x \notin V(g_{i})$, then $g_{i}^{-2k} x \in V'(g_{i}^{-})$ for $k > 0$. 

Now let $x, y\in X$ and $k>0$. Since $\{V(g_{i}^{+}), V(g_{i}^{-})\}$ are disjoint, $y \in V(g_{i}^{-})$ for at most one $g_{i} \in S'$ and $x \in V(g_{j}^{+})$ for at most one $g_{j} \in S'$. If $s=\phi_{1}^{2N} \cdots \phi_{10}^{2N} \in S'$ is neither of them, then \[
(x, s^{2} o)_{o} < d(o, so) - F_{1}, \quad (s^{2k}y, s^{2} o)_{o} \ge d(o, s o).
\] As $[o, s^{2} o]$ is $F_{0}$-witnessed by $[s \phi_{10}^{-N} o, so]$, this implies that $[o, s^{2k}y]$ is $G_{0}$-witnessed by $[s \phi_{10}^{-N} o, so]$. Now if we suppose that $(x, s^{2k} y)_{o} > d(o, so)$, then $[o, x]$ is also $G_{1}$-witnessed by $[s \phi_{10}^{-N} o, so]$. Then Fact \ref{fact:GromProdFact2} implies \[\begin{aligned}
(x, s^{2} o)_{o}& \ge d(o, so) - (G_{1} + F_{0}) \ge d(o, so) - F_{1},
\end{aligned}
\] a contradiction. Hence we deduce that $(x, s^{2k} y)_{o} \le d(o, so)$ for all $k > 0$.

Similarly, if $s \neq g_{i}$ such that $y \in V(g_{i}^{+})$ and $s \neq g_{j}$ such that $x \in V(g_{j}^{-})$, then $(x, s^{-2k} y)_{o} \le d(o, s^{-1} o)$ for all $k > 0$. Thus, we can take $\epsilon = \epsilon_{1}$, $K = \max_{g_{k} \in S'} d(o, g_{k} o)$ and $S = \{g_{k}^{2i} : g_{k} \in S'\}$ for any $i > K'/F_{1}$. 
\end{proof}

In the proof, we have actually proven a slightly stronger result: for each $x \in X$, $|\{s \in S : (x, s^{i} o)_{o} \ge K$ for some $i > 0\} |\le 1$ and $|\{s \in S : (x, s^{i} o)_{o} \ge K$ for some $i <0\}| \le 1$. This is because $o$ belongs to none of $V(g_{i}^{\pm})$.

We now make a choice for the remaining of the paper. Since $\mu$ is non-elementary, there exist two independent loxodromics $a, b$ in $\llangle \supp \mu \rrangle$. Taking suitable powers, we may assume that $a, b \in \supp \mu^{\ast n}$ for some common $n$. Given these $a$ and $b$, let $C_{0} = K$ and $\epsilon$ be the constants from Proposition \ref{prop:Schottky}. Let also \begin{itemize}
\item $D_{1} = D(C = C_{0}, \epsilon)$ as in Lemma \ref{lem:farSegment};
\item $D_{2}= D(C = C_{0}, \epsilon)$ as in Lemma \ref{lem:1segment};
\item $D_{0} = \max(D_{1}, D_{2})$;
\item $E_{0} = E(D=D_{0}, \epsilon)$, $L_{1} = L(D = D_{0}, \epsilon)$ as in Lemma \ref{lem:concat};
\item $F_{0} = F(E = E_{0}, \epsilon)$, $L_{2} = L(E = E_{0}, \epsilon)$ as in Lemma \ref{lem:concatUlt};
\item $D_{3} = D(C = F_{0}, \epsilon)$ as in Lemma \ref{lem:1segment};
\item $G_{0} = G (F = 2F_{0}, \epsilon)$, $L_{3} = L(F = 2F_{0}, \epsilon)$ as in Lemma \ref{lem:passerBy};
\item $F_{1} = F(E = G_{0}, \epsilon)$, $L_{2} = L(E = G_{0}, \epsilon)$ as in Lemma \ref{lem:concatUlt};
\item $F_{2} = \mathscr{B}(\epsilon, 2F_{0}) + 2F_{0} + 12\delta$;
\item $L_{0} = \max(L_{1}, L_{2}, L_{3}, 16D_{0}+ 8F_{0}+2G_{0}+16\delta+2, 16D_{3})$.
\end{itemize}
By Proposition \ref{prop:Schottky}, there exists a $(C_{0}, L_{0}, \epsilon)$-Schottky set $S_{0}$ of cardinality at least 310 in $\supp \mu^{\ast N}$ for some $N$. \emph{We fix these $a, b, N, S_{0}$ from now on}. For each $g \in G$, $s \in S_{0}$ and $i \in \{\pm 1, \pm 2\}$, segments of the form $[go, gs^{i} o]$ are called \emph{Schottky segments}; the maximum of their lengths is denoted by $\mathscr{M}$.

\subsection{Pivotal times}\label{subsection:pivot}

We begin by fixing a subset $S$ of the Schottky set $S_{0}$ with $|S| \ge 305$. The following definition is a variation of the one in \cite{gouezel2022exponential}. The main difference arises in backtracking, but most of the proofs are identical with the ones in \cite{gouezel2022exponential}.

Throughout this subsection \textit{except for Lemma \ref{lem:pivotEquivInterm}}, we fix isometries $(w_{i})_{i=0}^{\infty}$, $(v_{i})_{i=1}^{\infty}$ in $G$. Let $w_{0, 0}^{+}=w_{0, 2}^{+} = id$, and for $i \ge 1$, we consider \[\begin{aligned}
w_{i, 2}^{-} = w_{i-1, 2}^{+} w_{i-1}, \quad\,\, w_{i, 1}^{-}  = w_{i, 2}^{-}  a_{i}, \quad\,\,  w_{i, 0}^{-} &= w_{i, 2}^{-}  a_{i}^{2},\\
w_{i, 0}^{+} = w_{i, 2}^{-}  a_{i}^{2} v_{i}, \quad w_{i, 2}^{+} =w_{i, 2}^{-}  a_{i}^{2}v_{i}b_{i}^{2}
\end{aligned}
\]
and the translates $y_{i, t}^{\pm} = w_{i, t}^{\pm} o$ of $o$. Here $a_{i}$, $b_{i}$ are to be drawn from $S$ with the uniform measure and recorded as $s = (a_{1}, b_{1}, \cdots, a_{n}, b_{n})$.
We inductively define the \emph{set of pivotal times} $P_{n}$ and the moving point $z_{n}$. First take $P_{0} = \emptyset$ and $z_{0} = o$. Now given $P_{n-1}$ and $z_{n-1}$, $P_{n}$ and $z_{n}$ are determined as follows. \begin{enumerate}
\item  When $(z_{n-1}, y_{n, i}^{-})_{y_{n, 2}^{-}} < C_{0}$ for $i=0, 1$ and $(y_{n, 0}^{+}, y_{n+1, 2}^{-})_{y_{n, 2}^{+}}< C_{0}$, then we set $P_{n} = P_{n-1} \cup \{n\}$ and $z_{n} = y_{n, 0}^{+}$.
\item If not, we seek for sequences $\{i(1) < \cdots < i(N)\} \subseteq P_{n-1}$ such that $N > 1$ and $[y_{i(1), 0}^{+}, y_{n+1, 2}^{-}]$ is $(C_{0}, D_{0})$-head-marked with Schottky segments \begin{align}\label{eqn:witnessCond}
\left(\gamma_{i}\right)_{i=1}^{k} &= \left([y_{i(1), 0}^{+}, y_{i(1), 2}^{+}], [y_{i(2), 1}^{-}, y_{i(2), 0}^{-}], \ldots, [y_{i(N), 1}^{-}, y_{i(N), 0}^{-}]\right), \\ \label{eqn:witnessCond2}
\left(\eta_{i}\right)_{i=2}^{k} &=  \left( [y_{i(2), 2}^{-}, y_{i(2), 1}^{-}], \ldots, [y_{i(N), 2}^{-}, y_{i(N), 1}^{-}] \right).
\end{align}
If exists, let $\{i(1) < \cdots < i(N)\}$ be such sequence with maximal $i(1)$; we set $P_{n} = P_{n-1} \cap \{1, \ldots, i(1)\}$ and $z_{n} = y_{i(N), 1}^{-}$. If such sequence does not exist, then we set $P_{n} = \emptyset$ and $z_{n} = o$.\footnote{When there are several sequences that realize maximal $i(1)$, we choose the maximum in the lexicographic order on the length of sequences and $i(2)$, $i(3)$, $\ldots$.}
\end{enumerate}

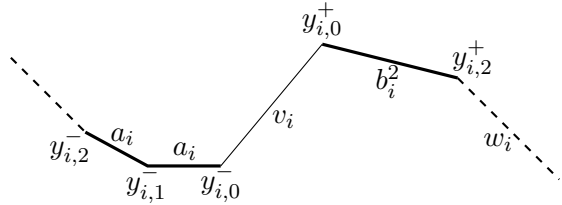
\begin{figure}[H]
\begin{tikzpicture}[scale=0.9]
\draw[thick, dashed] (-1.1, 1.1) -- (0, 0);
\draw[very thick] (0, 0) -- (0.92, -0.5)-- (2, -0.5);
\draw (2, -0.5) -- (3.5, 1.3);
\draw[very thick] (3.5, 1.3) -- (5.5, 0.8);
\draw[thick, dashed] (5.5, 0.8) -- (7, -0.7);

\draw (-0.23, -0.23) node {$y_{i, 2}^{-}$};
\draw (0.9, -0.77) node {$y_{i, 1}^{-}$};
\draw (2, -0.75) node {$y_{i, 0}^{-}$};
\draw (3.5, 1.65) node {$y_{i, 0}^{+}$};
\draw (5.72, 1.08) node {$y_{i, 2}^{+}$};

\draw (0.55, -0.07) node {$a_{i}$};
\draw (1.47, -0.28) node {$a_{i}$};
\draw (2.93, 0.19) node {$v_{i}$};
\draw (4.47, 0.75) node {$b_{i}^{2}$};
\draw (6.1, -0.15) node {$w_{i}$};
\end{tikzpicture}
\caption{Loci $y_{i, k}^{\pm}$ inside a trajectory.}
\label{fig:trajectory}
\end{figure}

For pivotal times $i$, we call $y_{i, 0}^{-}$ and $y_{i, 0}^{+}$ \emph{pivotal loci}.  We record some basic facts on pivotal times below.
\begin{enumerate}
\item The choice of $P_{n}$ is measurable with respect to the choice of $a_{i}, b_{i}$.
\item $i \in P_{m}$ only if $i$ becomes a pivotal time at step $i$ and survives during steps $i+1, \ldots, m$.
\item Let $m< n$ and $i < j$. If $i, j\in P_{m}$ and $j \in P_{n}$, then $i \in P_{n}$.
\end{enumerate}

\begin{lem}\label{lem:intermediate}
Let $l< m$ be consecutive pivotal times in $P_{n}$ and $t \in \{0, 1\}$. Then $[y_{l, 0}^{+}, y_{m, t}^{-}]$ is fully $D_{0}$-marked with Schottky segments $(\gamma_{i})_{i=1}^{N-1}$, $(\eta_{i})_{i=2}^{N}$, where $\gamma_{1} = [y_{l, 0}^{+}, y_{l, 2}^{+}]$ and $\eta_{N} = [y_{m, 2}^{-}, y_{m, t}^{-}]$.
\end{lem}

\begin{proof}
Note that $l$ is chosen as a pivotal time at step $l$; this implies $(y_{l, 0}^{+}, y_{l+1, 2}^{-})_{y_{l, 2}^{+}} <C_{0}<D_{0}$ and $z_{l} =y_{l, 0}^{+}$. If $l = m-1$ and $m$ was newly chosen at step $m = l+1$, we have $(z_{l}, y_{m, t}^{-})_{y_{m, 2}^{-}} < C_{0}<D_{0}$ also. Then Lemma \ref{lem:1segment} implies that $[y_{m-1, 0}^{+}, y_{m, t}^{-}]$ is fully $D_{0}$-marked with $[y_{m-1, 0}^{+}, y_{m-1, 2}^{+}]$ and $[y_{m, 2}^{-}, y_{m, t}^{-}]$.

If $l < m-1$, then $l$ survived in $P_{m-1}$ by the second criterion; there exist $l= i(1) < \ldots < i(N)$ such that $[y_{l, 0}^{+}, y_{m, 2}^{-}]$ is $(C_{0}, D_{0})$-head-marked with \[
\left([y_{l, 0}^{+}, y_{l, 2}^{+}], [y_{i(2), 1}^{-}, y_{i(2), 0}^{-}], \ldots, [y_{i(N), 1}^{-}, y_{i(N), 0}^{-}]\right), \,\, \left( [y_{i(2), 2}^{-}, y_{i(2), 1}^{-}], \ldots, [y_{i(N), 2}^{-}, y_{i(N), 1}^{-}] \right).\]
 Moreover, $z_{m-1} = y_{i(N), 1}^{-}$. Since $m$ was also newly chosen at step $m$, we have $(z_{m-1}, y_{m, t}^{-})_{y_{m, 2}^{-}} < C_{0}$. Then Lemma \ref{lem:1segment} implies that $[z_{m-1}, y_{m, t}^{-}]$ is also $D_{0}$-witnessed by $( [y_{i(N), 1}^{-}, y_{i(N), 0}^{-}], [y_{m, 2}^{-}, y_{m, t}^{-}])$ as desired.
\end{proof}

\begin{lem}\label{lem:extremal}
Suppose that $P_{n}$ is nonempty and $k = \min P_{n}$, $m = \max P_{n}$. Then for $t = 0, 1$, $[o, y_{k, t}^{-}]$ is $(C_{0}, D_{0})$-tail-marked with $[y_{k, 2}^{-}, y_{k, t}^{-}]$. Moreover, $[y_{m, 0}^{+}, y_{n+1, 2}^{-}]$ is $(C_{0}, D_{0})$-head-marked with segments $(\gamma_{i})_{i=0}^{N}$, $(\eta_{i})_{i=1}^{N}$ of the form \ref{eqn:witnessCond} and \ref{eqn:witnessCond2}, with $\gamma_{1} = [y_{m, 0}^{+}, y_{m, 2}^{+}]$.
\end{lem}

\begin{proof}
$k = \min P_{n}$ implies that $P_{k-1}$ has been empty and $z_{k-1} = o$. Moreover, $k$ was newly chosen at step $k$ so $(o, y_{k, t}^{-})_{y_{k, 2}^{-}} = (z_{k-1}, y_{k, t}^{-})_{y_{k, 2}^{-}} < C_{0}$ holds, hence the conclusion. For the latter statement, we observe how $m$ survived in $P_{n}$. If $m = n$ and was chosen due to the first criterion, then $[y_{m, 0}^{+}, y_{m+1, 2}^{-}]$ is $(C_{0}, D_{0})$-head-marked with $\gamma_{1} = [y_{m, 0}^{+}, y_{m, 2}^{+}]$. If not, $m$ survived due to the second criterion, which is clearly the desired condition.
\end{proof}

\begin{lem}\label{lem:0thCasePivot}
Suppose that $a_{i}, b_{i}$ are drawn from $S$ with respect to the uniform measure. Then $\Prob( |P_{n+1}| = |P_{n}| + 1)\ge 9/10$.
\end{lem}

\begin{proof}
$n+1$ becomes a new pivotal time if and only if the following two independent conditions are satisfied. By the remark after Proposition \ref{prop:Schottky}, \begin{equation}\label{eqn:pivotCondition}
(z_{n-1}, y_{n, t}^{-})_{y_{n, 2}^{-}} = ((w_{n, 2}^{-})^{-1}z_{n-1}, a_{n}^{2-t}o)_{o}<C_{0} \quad (t=0, 1)
\end{equation} holds for at least 304 choices of $a_{n}$ out of 305 possibilities. Hence, its chance is at least 0.99. Similarly, $(y_{n, 0}^{+}, y_{n+1, 2}^{-})_{y_{n, 2}^{+}} = (b_{n}^{-2} o, w_{n} o)_{o}<C_{0}$ with probability at least 0.99. Multiplying them yields the desired estimate.
\end{proof}

As in \cite{gouezel2022exponential}, given a choice $s = (a_{1}, b_{1}, \cdots, a_{n}, b_{n})$ with pivotal times $i_{1}, \ldots, i_{m}$, $\tilde{s} = (\tilde{a}_{1}, \tilde{b}_{1}, \ldots, \tilde{a}_{n}, \tilde{b}_{n})$ is \emph{pivoted from} $s$ if it has the same pivotal times with $s$, $\tilde{b}_{i} = b_{i}$ for all $i$ and $\tilde{a}_{i} = a_{i}$ for all $i$ that are not pivotal times.

\begin{lem}[{\cite[Lemma 4.7]{gouezel2022exponential}}] \label{lem:pivotEquiv}
Let $i$ be a pivotal time for the choice $s = (a_{1}, b_{1}, \cdots, a_{n}, b_{n})$, and $\bar{s}$ be obtained from $s$ by replacing $a_{i}$ with $\bar{a}_{i}$. Then $\bar{s}$ is pivoted from $s$ if $(z_{i-1}, \bar{y}_{i, t}^{-})_{y_{i, 2}^{-}} < C_{0}$ for $t=0, 1$ holds, and there are at least 304 such choices of $\bar{a}_{i}$.
\end{lem}

\begin{proof}
By the remark after Proposition \ref{prop:Schottky}, at least 304 choices satisfy the condition. Suppose that the condition holds. The other condition $(b_{i}^{-2} o, w_{i} o)_{o} < C_{0}$ for $i$ to be a pivotal time depends on the choice of $b_{i}$ (not $a_{i}$ or $\bar{a}_{i}$), so it is still satisfied. Hence, $i$ is selected in $P_{i}(\bar{s})$ and $\bar{z}_{i} = \bar{y}_{i, 1}^{+}$. The conditions that determine the later pivotal time of $s$ actually depend on the choice of $a_{i+1}$, $b_{i+1}$, $\ldots$. Indeed, since $i \in P_{n}$ and no backtracking occurs beyond $i$, the criteria for later pivotal times do not refer to the choices before $i$. Since $a_{i+1}$, $b_{i+1}$, $\ldots$ remain the same, we have $P_{n}(s) = P_{n}(\bar{s})$.
\end{proof}

The logic of the previous proof leads to the following lemma.

\begin{lem}[cf. {\cite[Lemma 5.7]{gouezel2022exponential}}] \label{lem:pivotEquivInterm}
Given isometries $(w_{0}, \ldots, w_{n})$, $(v_{1}, \ldots, v_{n})$, let $i$ be a pivotal time for a choice $s$. Then the set of pivotal times for $s$ remain the same if $v_{i}$ is replaced with some other isometry.
\end{lem}

For each $s \in S^{2n}$,  we denote by $\mathcal{E}_{n}(s)$ the set of choices that are pivoted from $s$. Note that $\mathcal{E}_{n}(s)$ for various $s$ are equivalence classes in $S^{2n}$.

\begin{lem}[{\cite[Lemma 4.8]{gouezel2022exponential}}] \label{lem:pivotCondition}
For each $j \ge0$ and $s \in S^{2n}$, we have \[
\Prob\Big(|P_{n+1}(\tilde{s}, a_{n+1}, b_{n+1})| < |P_{n}(s)| - j \, \Big| \, \tilde{s} \in \mathcal{E}_{n}(s)\Big) \le 1/10^{j+1}.
\]
\end{lem}

\begin{proof}
For $\tilde{s} \in \mathcal{E}_{n}(s)$ and a choice of elements $a_{n+1}, b_{n+1} \in S$ (which we call a \emph{good choice}) such that $P_{n+1}(s, a_{n+1}, b_{n+1}) = P_{n}(s) \cup \{n\}$, we have $P_{n+1}(\tilde{s}, a_{n+1}, b_{n+1}) = P_{n+1}(s, a_{n+1}, b_{n+1}) = P_{n}(s) \cup\{n\}$.This is because $(\tilde{s}, a_{n+1}, b_{n+1})$ is pivoted from $(s, a_{n+1}, b_{n+1})$. Moreover, Lemma \ref{lem:0thCasePivot} tells us that \[
\Prob(|P_{n+1}(s, a_{n+1}, b_{n+1})| < |P_{n}(s)|) \le 1 - \Prob(|P_{n+1}(s, a_{n+1}, b_{n+1})| = |P_{n}(s)| + 1) \le 1/10.
\]
This settles the case $j=0$.

Now let $l < m$ be the last 2 pivotal times for $s$. We fix a bad choice $(a_{n+1}, b_{n+1})$ such that $|P_{n+1}(s)| \neq |P_{n}(s)| + 1$, and a choice $\tilde{s} \in \mathcal{E}_{n}(s)$ until $n$.
Let us now fix $\tilde{a}_{i}, \tilde{b}_{i}$ of $\tilde{s}$ except at $\tilde{a}_{m}$ to define the collection $E(\tilde{s})$ of choices $\bar{s} = (\tilde{a}_{1}, \tilde{b}_{1}, \ldots, \bar{a}_{m}, \tilde{b}_{m}, \ldots)$ in $\mathcal{E}_{n}(s)$. Here, the condition of $\bar{a}_{m}$ that gurantees $\bar{s} \in \mathcal{E}_{n}(s)$ is
\begin{equation}\label{eqn:addCondPivot1}
(\tilde{z}_{m-1}, \bar{y}_{m, t}^{-})_{\bar{y}_{m, 2}^{-}}  = ((\tilde{w}_{m-1, 2}^{-})^{-1} \tilde{z}_{m-1}, \bar{a}_{m}^{2-t} o)_{o} < C_{0} \,\,(t=0, 1).
\end{equation}There are at least 304 such choices.

We now count the number of $\bar{a}_{m}$ that additionally satisfy
\begin{equation}\label{eqn:addCondPivot2}
(\bar{y}_{m, 1}^{-}, \bar{y}_{n+2, 2}^{-})_{\bar{y}_{m, 0}^{-}} = (\bar{a}_{m}^{-1} o, v_{m} \tilde{b}_{m}^{2} w_{m} \cdots v_{n+1} b_{n+1}^{2} w_{n+1} o)_{o} < C_{0}.
\end{equation}
These conditions are of the form $(\bar{a}_{m}^{2-t} o, x)_{o} < C_{0}$ or $(\bar{a}_{m}^{-1} o, x)_{o} < C_{0}$ where the involved $x$ is \emph{constant across $E(\tilde{s})$}. Thus, we miss at most 2 choices.

We now claim that when $\bar{a}_{m}$ satisfies Inequalities \ref{eqn:addCondPivot1} and \ref{eqn:addCondPivot2}, $|P_{n+1}(\bar{s})| \ge |P_{n}(s)|-1$. To see this, note that $P_{n}(\bar{s}) = P_{n}(s)$ since $\bar{s} \in\mathcal{E}_{n}(s)$. In particular,  $\max P_{m-1}(\bar{s}) = l$ and $m$ was chosen as a new pivotal time. By Lemma \ref{lem:extremal}, there exist $l = i(1) < \ldots < i(N)$ (for some $N\ge 1$) such that $[\tilde{y}_{l, 0}^{+}, \bar{y}_{m, 2}^{-}]$ is $(C_{0}, D_{0})$-head-marked with \[
\left([\tilde{y}_{i(1), 0}^{+}, \tilde{y}_{i(1), 2}^{+}], [\tilde{y}_{i(2), 1}^{-}, \tilde{y}_{i(2), 0}^{-}], \ldots, [\tilde{y}_{i(N), 1}^{-}, \tilde{y}_{i(N), 0}^{-}]\right),\quad \left([\tilde{y}_{i(2), 2}^{-}, \tilde{y}_{i(2), 1}^{-}], \ldots, [\tilde{y}_{i(N), 2}^{-}, \tilde{y}_{i(N), 1}^{-}]\right)
\] and $\tilde{z}_{m-1} = \tilde{y}_{i(N), 1}^{-}$. By Inequality \ref{eqn:addCondPivot1} and Lemma \ref{lem:1segment}, $[\tilde{y}_{i(N), 1}^{-}, \bar{y}_{m, 1}^{-}]$ is $D_{0}$-witnessed by $([\tilde{y}_{i(N), 1}^{-}, \tilde{y}_{i(N), 0}^{-}], [\bar{y}_{m, 2}^{-}, \bar{y}_{m, 1}^{-}])$. By Inequality \ref{eqn:addCondPivot2}, $[\bar{y}_{m, 1}^{-}, \bar{y}_{n+2, 2}^{-}]$ is $(C_{0}, D_{0})$-head-marked with $[\bar{y}_{m, 1}^{-} , \bar{y}_{m, 0}^{-}]$. Finally, $[\bar{y}_{m, 1}^{-}, \bar{y}_{m, 2}^{-}]$ and  $[\bar{y}_{m, 1}^{-}, \bar{y}_{m, 0}^{-}]$ are $D_{0}$-glued since $(\bar{a}_{m}^{-1} o, \bar{a}_{m} o)_{o} < C_{0} < D_{0}$ for each $\bar{a}_{m} \in S$. Thus, $[\tilde{y}_{l, 0}^{+}, \bar{y}_{n+2, 2}^{-}]$ is $(C_{0}, D_{0})$-head-marked with 
 \[
 \begin{array}{c} \left([\tilde{y}_{i(1), 0}^{+}, \tilde{y}_{i(1), 2}^{+}], [\tilde{y}_{i(2), 1}^{-}, \tilde{y}_{i(2), 0}^{-}], \ldots, [\tilde{y}_{i(N), 1}^{-}, \tilde{y}_{i(N), 0}^{-}], [\bar{y}_{m, 1}^{-} , \bar{y}_{m, 0}^{-}]\right),\\ \left([\tilde{y}_{i(2), 2}^{-}, \tilde{y}_{i(2), 1}^{-}], \ldots, [\tilde{y}_{i(N), 2}^{-}, \tilde{y}_{i(N), 1}^{-}], [\bar{y}_{m, 2}^{-}, \bar{y}_{m, 1}^{-}]\right) .\end{array}
\] In other words, $P_{n}(s) \cap \{1, \ldots, l\} \subseteq P_{n+1}(\bar{s}, a_{n+1}, b_{n+1})$ and $|P_{n+1}| \ge |P_{n}| - 1$.

In summary, out of at most 305 choices of $\bar{a}_{m}$ that make $\bar{s} \in \mathcal{E}_{n}(s)$, at least 303 choices of $\bar{a}_{m}$ make $|P_{n+1}(\bar{s})| \ge |P_{n}(\bar{s})| -1$. Thus, conditioned on $E(\tilde{s})$, $|P_{n+1}(\bar{s})| < |P_{n}(\bar{s})| -1$ has probability less than $1/10$. Now note that $\mathcal{E}_{n}(s)$ is partitioned into $E(\tilde{s})$'s for various $\tilde{s}$, induced from the equivalence relation that two sequences differ only at the $(2m-1)$-th coordinate. Summing up the conditional probability, we deduce the following: given a bad choice $(a_{n+1}, b_{n+1})$, $|P_{n+1}(\bar{s}, a_{n+1}, b_{n+1})| < |P_{n}(\bar{s})| -1$ has probability less than $1/10$. Recall that for a good choice $(a_{n+1}, b_{n+1})$, this event has probability zero. Finally, the bad choices $(a_{n+1}, b_{n+1})$ constitute a probability less than $1/10$. Thus, $\Prob(|P_{n+1}(\bar{s})| < |P_{n}(\bar{s}) |-1) \le (1/10) \times (1/10)$.

For $j = 2$, we consider a partition $\{E_{\alpha}\}_{\alpha}$ of $\mathcal{E}_{n}(s)$ made by pivoting the $l$-th choice only, i.e., \[
\tilde{s} \sim \tilde{s}' \quad \Leftrightarrow \quad \forall i \neq l,[\tilde{a}_{i} = \tilde{a}_{i}'], \forall i[\tilde{b}_{i} = \tilde{b}_{i}'].
\]In a similar fashion as before, we deduce the following: fixing a choice of elements $a_{n+1}, b_{n+1} \in S$, on each equivalence class $E_{j}$, we have \begin{equation}\label{eqn:caseJ2}
\Prob\left(|P_{n+1}(\bar{s}, a_{n+1}, b_{n+1})| < |P_{n}(\bar{s})| - 2 \right) < 1/10.
\end{equation}

Now consider $\tilde{s} \in \mathcal{E}_{n}(s)$ and $a_{n+1}, b_{n+1} \in S$ such that $P_{n+1}(\tilde{s}, a_{n+1}, b_{n+1}) \ge |P_{n}(s)| - 1$. This means that $\{1, \ldots l\} \cap P_{n}(s) \subseteq P_{n+1}(\tilde{s}, a_{n+1}, b_{n+1})$. Then for any $\tilde{s}$ that is in the same equivalence class as $\bar{s}$, $(\bar{s}, a_{n+1}, b_{n+1})$ is pivoted from $(\tilde{s}, a_{n+1}, b_{n+1})$ at a pivotal time. Hence, $P_{n+1} (\bar{s}, a_{n+1}, b_{n+1})$ equals $P_{n+1}(\tilde{s}, a_{n+1}, b_{n+1})$ and its cardinality is at least $|P_{n}(s)| - 1$. 

Based on this observation, we can come up with two collections: \[\begin{aligned}
\mathcal{C}_{1} &= \{ (E_{\alpha}, a_{n+1}, b_{n+1}) : \textrm{for all}\,\tilde{s} \in E_{\alpha}, |P_{n+1}(\tilde{s}, a_{n+1}, b_{n+1})| \ge |P_{n}(\bar{s})| - 1 \}\\
\mathcal{C}_{2} &=\{ (E_{\alpha}, a_{n+1}, b_{n+1}) : \textrm{for all}\,\tilde{s} \in E_{\alpha}, |P_{n+1}(\tilde{s}, a_{n+1}, b_{n+1})| < |P_{n}(\bar{s})| - 1 \}.
\end{aligned}
\]
We know that $\mathcal{C}_{1}$ takes up probability at least $1 - 1/10^{2}$, which is the case $j=1$. Moreover, for each combination $(E_{\alpha}, a_{n+1}, b_{n+1}) \in \mathcal{C}_{2}$, we have the bound on the conditional probability (Inequality \ref{eqn:caseJ2}). This leads to the desired bound for $j=2$.
 
We keep doing this until $j< |P_{n}(s)|$. The case $j \ge |P_{n}(s)|$ is void.
\end{proof}

This estimate on each equivalence class $\mathcal{E}_{n}(s)$ implies the following:
\begin{prop}[{\cite[Proposition 4.10]{gouezel2022exponential}}] \label{prop:pivotEstimate}
There exist $K_{0} > 0$ such that  for any choice of $w_{i}$, $v_{i}$ and any $n$, we have $\Prob(|P_{n}| \le  n/K_{0}) \le K_{0}e^{- n/K_{0}}.$
\end{prop}

\section{Pivots in random walks}

\subsection{The first model and pivoting}
\label{subsection:1stModel}

Our first model is almost verbatim from one of Gou{\"e}zel's models in \cite[Section 4A]{gouezel2022exponential}. We pick any $c \in S_{0}$ and $S \subseteq S_{0} \setminus \{c\}$ with $|S|= 305$. Recall that each $s \in S_{0}$ belongs to $\supp \mu^{\ast N}$; we fix isometries $a_{1}(s), \ldots, a_{N}(s) \in \supp \mu$ such that $a_{1}(s) \cdots a_{N}(s) = s$. Let $\mu_{S^{(2)}}$ be the uniform measure on the set $\{(a_{1}(s), \ldots, a_{N}(s), a_{1}(s), \ldots, a_{N}(s)) : s \in S\}$, and $1_{\{c\}^{2}}$ be the measure concentrated on $(a_{1}(c), \ldots, a_{N}(c), a_{1}(c), \ldots, a_{N}(c))$. Then there exist a measure $\nu$ and $0<\alpha<1$ such that \[
\mu^{6N}=\alpha (\mu_{S^{(2)}} \times 1_{\{c\}^{2}} \times \mu_{S^{(2)}}) + (1-\alpha) \nu.
\] Here we set $\eta = \mu_{S^{(2)}} \times 1_{\{c\}^{2}} \times \mu_{S^{(2)}}$ and employ the setting in Subsection \ref{subsection:random}: we are given independent RVs $\{\rho_{i}, \eta_{i}, \nu_{i}\}$ and auxiliary RVs $\sumRho(k)$, $\stopping(i)$. Together with these, we define $\alpha_{i}$ ($\beta_{i}$, resp.) as the product of the first (last, resp.) $N$ coordinates of $\eta_{i}$. Then $\{\rho_{i}, \alpha_{i}, \beta_{i}, \nu_{i}\}$ also become independent.

For a fixed $n$, let $\gamma' = g_{6N\lfloor n  / 6N\rfloor + 1} \cdots g_{n}$ and observe \begin{equation}\begin{aligned}\label{eqn:pivotSetting}
\w_{n} &=w_{0}\, \cdot \, a_{1}^{2} c^{2} b_{1}^{2} \, \cdot \, w_{1} \, \cdot \, a_{2}^{2} c^{2} b_{2}^{2} \, \cdots\, a_{\sumRho(\lfloor n  / 6N\rfloor)}^{2} c^{2} b_{\sumRho(\lfloor n  / 6N\rfloor)}^{2} \, \cdot \, w_{\sumRho(\lfloor n  / 6N\rfloor)}'.\\
&\quad \quad\left( \begin{array}{c}w_{i} = \nu_{\stopping(i) + 1}^{\ast} \cdots \nu_{\stopping(i+1) - 1}^{\ast}$, $a_{i} = \alpha_{\stopping(i)}$, $b_{i} = \beta_{\stopping(i)}, \\
w'_{\sumRho(\lfloor n  / 6N\rfloor)} = \nu_{\stopping(\sumRho(\lfloor n  / 6N\rfloor)) + 1}^{\ast} \ldots \nu_{\lfloor n  / 6N\rfloor}^{\ast} \gamma' \end{array}\right)
\end{aligned}
\end{equation}
In this setting, we keep using the notation $w_{i, j}^{\pm}$ and $y_{i, j}^{\pm}$. Note that \[
w_{i, 2}^{-} = \w_{6N(\stopping(i)-1)}, \quad w_{i, 2}^{+} = \w_{6N\stopping(i)}.
\]
We now define pivotal times as in Subsection \ref{subsection:pivot}. Note that the set of pivotal times $P_{n}(\w)$ until $n$ depends on the choice $\{\rho_{i}, \nu_{i}\}_{i=1}^{\lfloor n/6N\rfloor}$, $\gamma'$ and $s = (a_{1}, b_{1}, \ldots, a_{\sumRho(\lfloor n/6N\rfloor)}, b_{\sumRho(\lfloor n/6N\rfloor)})$. As before, a trajectory $\tilde{\w}$ is said to be \emph{pivoted from $\w$ until $n$} if $P_{n}(\w) = P_{n}(\tilde{\w})$ and their values of $\rho_{i}$, $\alpha_{i}$, $\beta_{i}$, $\nu_{i}$ coincide except for $\alpha_{j}$'s at pivotal times $j$ of $\w$. Finally, we define $Q_{n}(\w) := \cap_{k \ge n} P_{k}(\w)$ and the set of \emph{eventual pivotal times} $Q = \cup_{n}Q_{n}$. A small observation is: 

\begin{obs}\label{obs:eventualPivotK}
For each $k$, $Q_{k}(\w)$ consists of $|Q_{k}(\w)|$ smallest elements of $Q(\w)$. In other words, if we label elements of $Q(\w)$ by $i(1) < i(2) < \ldots$, then $Q_{k}(\w) = \{i(1), i(2), \ldots, i(|Q_{k}(\w)|)\}$.
\end{obs}

Since $\sumRho(k)$ is the sum of i.i.d.s with strictly positive expectation and finite exponential moment, there exists $K'> 0$ such that \[
\Prob(\sumRho(\lfloor n/6N \rfloor) \le n/K') \le K'e^{-n/K'}.
\] Once $\rho_{i}$ (and thus $\sumRho(k)$, $\stopping(i)$) are determined, $\{a_{1}, b_{1}, \ldots, a_{\sumRho(\lfloor n/6N \rfloor)}, b_{\sumRho(\lfloor n/6N \rfloor)}\}$ are independently drawn from $S$ with the uniform measure. Combining these facts with Proposition \ref{prop:pivotEstimate}, we deduce \[
\Prob(\w : |P_{n}(\w)| \le  n/K_{0}K') \le K'e^{-n/K'} + K_{0}e^{- n/K_{0}}.
\]

Note also that the trajectories for $P_{n}$, $P_{n+1}, \ldots$ all share the subwords until $b_{\sumRho(\lfloor n/6N \rfloor)}$. Consequently, the first $\min \{|P_{n}|, |P_{n+1}|, \ldots\}$ pivotal times of $P_{n}$, $P_{n+1}$, $\ldots$ coincide, which constitute $Q_{n}$. In particular, the first $\min \{|P_{n}|, |P_{n+1}|, \ldots\}$ elements of $Q$ constitute $Q_{n}$ and $|Q_{n}| = \min \{|P_{n}|, |P_{n+1}|, \ldots\}$. The discussion so far implies:

\begin{prop}\label{prop:expDecay}
There exist $K_{1} > 0$ such that $\Prob(|P_{n}| \le n/K_{1}) \le K_{1}e^{- n/K_{1}}$ and $\Prob(|Q_{n}| \le  n/K_{1}) \le K_{1} e^{- n/K_{1}}$.
\end{prop}

Here the second inequality follows from\[
\Prob(|Q_{n}| \le n/K_{0} K' ) \le \sum_{i=n}^{\infty} \Prob(|P_{i}| \le  i/K_{0} K' ) \le \frac{K'}{1-e^{-1/K'}} e^{- n/K'} +  \frac{K_{0}}{1-e^{-1/K_{0}} }e^{-n/K_{0}}.
\]

We now relate pivotal times with alignment. For a sample path $\w$ and $n > 0$, let $\{j(1) < \ldots < j(|P_{n}(\w)|)\} \subseteq \{1, \ldots, \mathcal{N}(\lfloor n/6 N \rfloor)\}$ be the set of pivotal times $P_{n}(\w)$ until step $n$. Then we have that: \begin{enumerate}
\item for each $l = 1, \ldots, |P_{n}(\w)| - 1$, $[y_{j(l), 0}^{+}, y_{j(l+1), 0}^{-}]$ is fully $D_{0}$-marked with Schottky segments $\left(\gamma_{i}^{(l)}\right)_{i=1}^{N(l)-1}, \left(\eta_{i}^{(l)}\right)_{i=2}^{N(l)}$ such that $\gamma_{1}^{(l)} = [y_{j(l), 0}^{+}, y_{j(l), 2}^{-}]$ and $\eta_{N(l)}^{(l)} = [y_{j(l+1), 2}^{-}, y_{j(l+1), 0}^{-}]$;
\item $[o, y_{j(1), 0}^{-}]$ is $(C_{0}, D_{0})$-tail-marked with $[y_{j(1), 2}^{-}, y_{j(1), 0}^{-}]$;
\item $[y_{j(|P_{n}(\w)|, 0}^{+}, \w_{n} o]$ is $(C_{0}, D_{0})$-head-marked with some Schottky segments $\left(\gamma_{i}^{(|P_{n}(\w)|)}\right)_{i=1}^{N(|P_{n}(\w)|)}$, $\left(\eta_{i}^{(|P_{n}(\w)|)}\right)_{i=2}^{N(|P_{n}(\w)|)}$ where $\gamma_{1}^{(|P_{n}(\w)|)} = [y_{j(|P_{n}(\w)|, 0}^{+}, y_{j(|P_{n}(\w)|, 2}^{+}]$, and
\item for each $l=1, \ldots, |P_{n}(\w)|$, sequences of Schottky segments \[
\left( [y_{j(l), 0}^{-}, w_{j(l), 0}^{-} c o], [y_{j(l), 0}^{+}, y_{j(l), 2}^{+}]\right), \quad \left( [y_{j(l), 2}^{-}, y_{j(l), 0}^{-} ], [w_{j(l), 0}^{-}c o, y_{j(l), 0}^{+}]\right)
\]
are $D_{0}$-aligned.
\end{enumerate}
The first item is due to Lemma \ref{lem:intermediate}, and the second and the third items are due to Lemma \ref{lem:extremal}. For the final item, the relevant inequalities for the gluing at $y_{j(l), 0}^{\pm}$ are: 
\[
(co, a^{-2}o)_{o}< C_{0}, \quad(c^{-1}o, b^{2}o)_{o} < C_{0},
\]
which hold because $c \neq a, b$ are chosen from a $(C_{0}, L_{0}, \epsilon)$-Schottky set. Moreover, 
$[o, c^{2} o]$ is $D_{0}$-witnessed by $([o, co], [co, c^{2}o])$: this follows from Lemma \ref{lem:1segment}, since $(o, co)_{co}, (o, c^{2}o)_{co} < C_{0}$. 

Combining all these items, we deduce that:

\begin{prop}\label{prop:concatUltAlign}
Given the values of $\{\rho_{i}, \nu_{i}, s\}$, let $j(1)< \ldots< j(|P_{n}(\w)|)$ be the elements of $P_{n}(\w)$. Then there exist sequences of Schottky segments, $(\gamma_{i})_{i=1}^{N}$ and $(\eta_{i})_{i=1}^{N}$ such that $[o, w_{n} o]$ is $(C_{0}, D_{0})$-marked with $(\gamma_{i})_{i}, (\eta_{i})_{i}$. Moreover, there exist indices $j'(1), \ldots, j'(|P_{n}(\w)|) \in \{1, \ldots, N-1\}$ such that $j'(l) \le j'(l+1) - 2$ for $l = 1, \ldots, N-1$ and \[\begin{aligned}
\gamma_{j'(l)} &= [y_{j(l), 0}^{-}, w_{j(l), 0}^{-} c o], \quad &\gamma_{j'(l) + 1} &= [y_{j(l), 0}^{+}, y_{j(l), 2}^{+}], \\
 \eta_{j'(l)} &= [y_{j(l), 2}^{-}, y_{j(l), 0}^{-}], \quad &\eta_{j'(l) + 1} &= [\w_{j(l), 0}^{-} c o, y_{j(l), 0}^{+}].
 \end{aligned}
\]
\end{prop}

Recall now that Schottky segments are longer than $L_{0} \ge L_{1}, L_{2}, 2[6D_{0} + 2F_{0} + 8\delta + 1]$. Then Corollary \ref{cor:induction} implies the following:

\begin{prop}\label{prop:concatUlt}
Given the values of $\{\rho_{i}, \nu_{i}, s\}$, let $j(1), \ldots, j(|P_{n}(\w)|)$ be pivotal times in $P_{n}(\w)$. Let also $
x_{0} = o$, $x_{2|P_{n}(s)| + 1} = \w_{n} o$, and
\[
(x'_{2l-1}, x_{2l - 1}, x_{2l}, x'_{2l}) = \left(y_{j(l), 2}^{-}, y_{j(l), 0}^{-}, y_{j(l), 0}^{+}, y_{j(l), 2}^{+}\right) \quad (l=1, \ldots, |P_{n}(\w)|).
\]
Then for any $0 \le i \le j \le  k \le 2|P_{n}(s)|+ 1$, we have \[
(x_{i}, x_{k})_{x_{j}} < F_{0},\quad d(x_{i}, x_{j+1}) \ge d(x_{i}, x_{j}) + L_{0}/2.
\]
Moreover, $[x_{i}, x_{k}]$ is $F_{0}$-witnessed by $[x_{2j-1}', x_{2j-1}]$ ($[x_{2j}, x_{2j}']$, resp.) if $i <2j-1 \le k$ ($i \le 2j < k$, resp.).
\end{prop}

The following lemma will be used only in the proof of Theorem \ref{thm:tracking}. We can also define backward pivotal times in the backward path $\check{\w}$, with choices from $S^{-1} = \{s^{-1} : s \in S\}$. We denote the set of backward pivotal times until $n$ by $\check{P}_{n}(\check{\w})$. We also analogously define the set of backward eventual pivotal times $\check{Q}(\check{\w}) := \cup_{n} \check{Q}_{n}(\check{\w})$, where $\check{Q}_{n} (\w) := \cap_{k \ge n} \check{P}_{k}(\check{\w})$.

\begin{lem}\label{lem:eventualBiInf}
For a.e. bi-infinite path $(\check{\w}, \w)$, there exist infinitely many forward eventual times $\{i(1) < i(2) < \ldots \}$. Moreover, there exists $m \in \Z_{>0}$ such that if $|Q_{k}| \ge m$ and $|\check{Q}_{k'}| \ge m$, then $(\check{\w}_{k'} o, \w_{k} o)_{x} \le F_{0}$ with $x= y_{i(l), 0}^{\pm}$ for $l = m, \ldots, |Q_{k}(\w)|$.
\end{lem}

\begin{proof}
By Proposition \ref{prop:expDecay}, almost every $(\check{\w}, \w)$ has infinitely many forward eventual pivotal times $Q(\check{\w}, \w) := \{i(1), i(2), \ldots\}$ and backward eventual pivotal times $\check{Q}( \check{\w}, \w) := \{\check{i}(1), \check{i}(2), \ldots\}$. Hence, we focus on an equivalence class $\mathcal{E}$ of bi-infinite paths $(\check{\w}, \w)$ that are pivoted from each other at forward/backward eventual pivotal times. Each path is then determined by its choices $a_{i(1)}, \check{a}_{i(1)}, a_{i(2)}, \check{a}_{i(2)}, \ldots$ at forward/backward eventual pivotal times, and these choices follow independent uniform distributions. Now given the choices $a_{i(1)}, \check{a}_{i(1)}, \ldots, a_{i(l-1)}, \check{a}_{i(l-1)}$, we collect the choices $a_{i(l)}, \check{a}_{i(l)}$ that satisfy \begin{enumerate}
\item $(\check{y}_{\check{i}(l), 0}^{-}, y_{i(l), 2}^{-})_{\check{y}_{\check{i}(l), 2}^{-} } < C_{0}$,
\item $(\check{y}_{\check{i}(l), 0}^{-}, y_{i(l), 0}^{-})_{y_{i(l), 2}^{-}} < C_{0}$.
\end{enumerate}

The above conditions are satisfied by at least $303 \times 303$ choices out of at most $305 \times 305$ choices. This implies that on $\mathcal{E}$, the above condition is satisfied for some $1 \le l \le m$ for conditional probability at least $1 - (0.01)^{m}$. Hence, for a.e. bi-infinite path $\w \in \mathcal{E}$, there exists some $m$ that satisfies the above conditions for $l=m$. Let us fix $k, k'$ such that $|Q_{k}(\w)|, |\check{Q}_{k'}(\check{\w})| \ge m$. We observe: \begin{enumerate}
\item $[\check{y}_{\check{i}(m), 0}^{-}, y_{i(m), 0}^{-}]$ is fully $D_{0}$-marked with Schottky segments $[y_{i(l), 2}^{-}, y_{i(l), 0}^{-}]$, $[\check{y}_{\check{i}(k), 0}^{-}, \check{y}_{\check{i}(k), 2}^{-}]$ (by Lemma \ref{lem:1segment}), 
\item $[y_{i(m), 0}^{-}, \w_{k} o]$ is $(C_{0}, D_{0})$-head-marked with sequences of Schottky segments, $(\gamma_{i})_{i=1}^{N}$ and $(\eta_{i})_{i=2}^{N}$, where $\gamma_{1} = [y_{i(m), 0}^{-}, w_{j(m), 0}^{-} c o]$ and some of $\gamma_{i}$'s have endpoints $y_{i(m), 0}^{\pm}$, $y_{i(m+1), 0}^{\pm}$, $\ldots$. This is due to Proposition \ref{prop:concatUltAlign}, with an observation \[
P_{k}(\w) \supseteq Q_{k}(\w) \supseteq \{i(1), \ldots, i(|Q_{k}(\w)|)\}.
\]
\item $[\check{\w}_{k'}o, \check{y}_{\check{i}(m), 0}^{-}]$ is $(C_{0}, D_{0})$-tail-marked with sequences of Schottky segments, $(\check{\gamma}_{i})_{i=1}^{N'-1}$ and $(\check{\eta}_{i})_{i=1}^{N'}$, where $\eta_{N'} = [\check{w}_{\check{i}(m), 0}^{-} c o, \check{y}_{\check{i}(m), 0}^{-}]$. This is again due to Proposition \ref{prop:concatUltAlign}.
\end{enumerate}

To concatenate these alignments, we finally need to check that $[y_{i(m), 0}^{-}, y_{i(m), 2}^{-}]$ and $[y_{i(m), 0}^{-}, w_{i(m), 0}^{-} c o]$ are $D_{0}$-glued. Since $(c, o, a^{-2} o) < C_{0}<D_{0}$ for any $a \neq c$ in $S_{0}$, this is guaranteed. Similarly, $[\check{y}_{\check{i}(m), 0}^{-}, \check{y}_{\check{i}(m), 2}^{-}]$ and $[\check{y}_{\check{i}(m), 0}^{-}, \check{w}_{\check{i}(m), 0}^{-} c o]$ are $D_{0}$-glued.

Combining these, we observe that $[\check{\w}_{k'} o, \w_{k} o]$ is $(C_{0}, D_{0})$-marked with sequences of Schottky segments, whose endpoints include $y_{i(l), 0}^{\pm}$'s for $l = m, \ldots, |Q_{k}(\w)|$. Now Corollary \ref{cor:induction} yields the conclusion.
\end{proof}

\subsection{Pivoting and its consequences}\label{subsection:conseq}

Using the prevalence of pivotal loci, Gou{\"e}zel recovered in \cite{gouezel2022exponential} the result of Maher and Tiozzo that non-elementary random walks on a weakly hyperbolic group escape to infinity. We recover an analogous result due to Kaimanovich and Masur.

\begin{cor}[{\cite[Theorem 2.2.4]{kaimanovich1996poisson}}]\label{cor:masur}
Almost every sample path on Teichm{\"u}ller space escapes to infinity and tends to a uniquely ergodic foliation.
\end{cor}

\begin{proof}
As before, let $Q(\w) = \{i(1) < i(2) < \ldots\}$. By Proposition \ref{prop:expDecay} and Borel-Cantelli, $|Q_{n}(\w)|$ tends to infinity for a.e. path $\w$. The escape to infinity then follows from Proposition \ref{prop:concatUlt}. 

By a standard Arzel{\`a}-Ascoli argument, we observe that $\{\gamma_{n} = [o, \w_{n} o]\}$ has a subsequence $\{\gamma_{n_{i}}\}$ that converges to a half-infinite geodesic $\gamma$. Note that $\gamma_{n}$ are eventually $F_{0}$-close to each eventual pivotal locus by Lemma \ref{lem:concatUlt}. In particular, $\gamma$ cannot fall into $\epsilon$-thin part and the vertical foliation $V_{\gamma}$ of $\gamma$ is uniquely ergodic (\cite[Theorem 1.1]{masur1992hausdorff}). Here, $\gamma$ tends to $\zeta = [V_{\gamma}] \in \PMF$.

Now $\{y_{i(k), 0}^{-}\}_{k=1}^{\infty}$ tend to $\zeta$ since they are $F_{0}$-close to $\gamma$ and escape to infinity. Moreover, thanks to Proposition \ref{prop:concatUlt}, we have \[
d(o, y_{i(|Q_{n}|), 0}^{-}) +  d(y_{i(|Q_{n}|), 0}^{-}, \w_{n} o)- d(o, \w_{n} o) = 2(o, \w_{n} o)_{y_{i(|Q_{n}|), 0}^{-}} \le 2F_{0}.
\]
This implies that $d(o, \w_{n} o) \ge d(o, y_{i(|Q_{n}|), 0}^{-}) - 2F_{0}$. Since $y_{i(|Q_{n}|), 0}^{-}$ tend to $\zeta$, so do $\w_{n} o$ by Lemma 1.4.2 of \cite{kaimanovich1996poisson}.
\end{proof}

We now turn to the proof of Theorem \ref{thm:logCorr}. 

\begin{proof}[Proof of Theorem \ref{thm:logCorr}]
Let $\lambda$ be the escape rate of $\w$. Without loss of generality, we may assume that $K_{1} > (8\mathscr{M} + 8F_{0})/\lambda$. We now take $M =2K_{1}/\log 50$.  For each $n \in \Z_{> 0}$, we let $m = \lfloor M \log n \rfloor$ and define \[\begin{aligned}
E_{n, 1} &:= \left\{ \w\in \Omega : |Q_{m}(\w)| \ge m/K_{1}\right\},\\
E_{n, 2} &:=  \left\{ \w \in \Omega : d(o, \w_{i} o) \le 2\lambda m\,\, \textrm{for all } \,\, i \le m\right\}\\
E_{n, 3} &:= \left\{\w \in \Omega :  d(o, \w_{n} o) > 0.5\lambda n \right\},\\
F_{n} &:= \{\w\in E_{n, 1} \cap E_{n, 2} \cap E_{n, 3} : |d(o, \w_{n} o) - \tau(\w_{n})| \ge 5\lambda m \}.
\end{aligned}
\]

We denote by $\mathcal{E}^{m/K_{1}}(\w)$ the collection of trajectories pivoted from $\w$ only at the first $\lfloor m/K_{1} \rfloor$ eventual pivotal times. Fixing $\w \in F_{n}$, we will estimate $\Prob(F_{n} | \mathcal{E}^{m/K_{1}}(\w))$.

Let $i(1)< \ldots< i(\lfloor m/K_{1} \rfloor )$ be the first $\lfloor m/K_{1} \rfloor$ eventual pivotal times of $\w$.  Note that $d(o, w_{i(l), 2}^{-} o)= d(o, \w_{6N\stopping(i(l))} o) \le 2\lambda m$ for $l= 1, \ldots, \lfloor m/K_{1} \rfloor$ since $\w \in E_{n, 1} \cap E_{n, 2}$. By Lemma \ref{lem:pivotEquiv}, at least 304 choices of $\bar{a}_{i(1)}$ make $\bar{\w}$ pivoted from $\w$. For each such choice, at least 304 choices of $\bar{a}_{i(2)}$ make $\bar{\w}$ pivoted from $\w$. Inductively, there are at least $304^{\lfloor m/K_{1}\rfloor}$ choices for $\bar{\w} \in \mathcal{E}^{m/K_{1}}(\w)$.

Our next goal is to show that only few choices of $\bar{a}_{i(1)}$ and $\bar{a}_{i(\lfloor m/K_{1} \rfloor)}$ are allowed for $\bar{\w} \in \mathcal{E}^{m/K_{1}}(\w) \cap F_{n}$. Let us consider \[
(\bar{x}'_{2l-1}, \bar{x}_{2l-1}, \bar{x}_{2l}) := (\bar{y}_{j(l), 0}^{-}, \bar{y}_{j(l), 0}^{-}, \bar{y}_{j(l), 0}^{+})
\]
for $l=1, \ldots, \lfloor m/K_{1} \rfloor$ and $\bar{x}_{0} := o$. Since $Q_{m}(\w) \subseteq P_{m}(\w)$ and $\bar{\w}$ is pivoted from $\w$ at $i(l)$'s in $Q_{m}(\w)$, we also have $i(1), \ldots, i(\lfloor m/K_{1} \rfloor) \in P_{m}(\bar{\w}) = P_{m}(\w)$. Then Proposition \ref{prop:concatUlt} tells us the following. First, since $[\bar{x}_{2l-2}, \bar{x}_{2l-1}]$ is $F_{0}$-witnessed by $[\bar{x}_{2l-1}', \bar{x}_{2l-1}]$, we have \[
\left| [d(\bar{x}_{2l-2}, \bar{x}'_{2l-1}) +(\bar{x}_{2l-1}', \bar{x}_{2l-1})] - d(\bar{x}_{2l-2}, \bar{x}_{2l-1}) \right| \le 2F_{0}.
\]
for each $l$. Moreover, we have $d(\bar{x}_{0}, \bar{x}_{l}) \le d(\bar{x}_{0}, \bar{w}_{m} o)$ for each $l \le 2\lfloor m/K_{1} \rfloor$. Finally, for each $i \le j \le k$ we have  $(\bar{x}_{i}, \bar{x}_{k})_{\bar{x}_{j}} < F_{0}$ for $i \le j \le k$. This implies that for each $1 \le t \le \lfloor m /K_{1} \rfloor$, we have
 \[\begin{aligned}
& \left| d(\bar{x}_{0}, \bar{x}_{2t} ) - \sum_{l=1}^{t} [d(\bar{x}_{2l-2}, \bar{x}'_{2l-1}) + d(\bar{x}_{2l-1}, \bar{x}_{2l})]\right|\\
 &\le \left| d(\bar{x}_{0}, \bar{x}_{2t + 1}) - \sum_{l=1}^{t} [d(\bar{x}_{2l-2}, \bar{x}'_{2l-1}) +d(\bar{x}_{2l-1}', \bar{x}_{2l-1})+d(\bar{x}_{2l-1}, \bar{x}_{2l})] \right| + \left|  \sum_{l=1}^{t } d(\bar{x}_{2l-1}', \bar{x}_{2l-1}) \right| \\
&\le \left| d(\bar{x}_{0}, \bar{x}_{2t }) - \sum_{l=1}^{t} [d(\bar{x}_{2l-2}, \bar{x}_{2l-1}) +d(\bar{x}_{2l-1}, \bar{x}_{2l})] \right| +  2t F_{0} + \left|  \sum_{l=1}^{t } d(\bar{x}_{2l-1}', \bar{x}_{2l-1}) \right| \\
&\le 4tF_{0} + 2tF_{0} + t\mathscr{M}.
\end{aligned}
\]

This inequality is useful because the terms $d(\bar{x}_{2l-2}, \bar{x}_{2l-1}')$ and $d(\bar{x}_{2l-1}, \bar{x}_{2l})$ are not affected by the pivoting. In particular, we have \[\begin{aligned}
d(\bar{x}_{0}, \bar{x}_{2t}) &\le d(x_{0}, x_{2t}) + (6F_{0} + \mathscr{M})t \\
& \le d(o, \w_{m} o) + (6F_{0} + \mathscr{M})t \\
&\le 2 \lambda m + 2 \lfloor m/K_{1} \rfloor (\mathscr{M} + 6F_{0})
\end{aligned}
\]
for each $1 \le t \le \lfloor m/K_{1} \rfloor$. Here, in the second inequality we also used the fact that $i(t) \in Q_{m}(\w) \subseteq P_{m}(\w)$.

Now let $v =( \bar{w}_{i(\lfloor m/K_{1} \rfloor), 0}^{-} )^{-1} \bar{\w}_{n} \bar{w}_{i(1), 2}^{-}$. Note that $v$ is not modified by pivoting at $i(1), \ldots, i(\lfloor m/K_{1} \rfloor)$. Moreover, we have \begin{equation}\label{eqn:largeIntermediate}
\begin{aligned}
d(o, vo) &\ge d(o, \w_{n} o) - d(o,  y_{i(\lfloor m/K_{1} \rfloor), 0}^{-}) - d(o, y_{i(1), 2}^{-})\\
&\ge d(o, \w_{n} o) - d(o,  y_{i(\lfloor m/K_{1} \rfloor), 0}^{-}) - d(o, y_{i(1), 0}^{-}) - 2\mathscr{M} \\
 & \ge 0.5 \lambda n - 4 \lambda m - 2\mathscr{M} \ge 2\mathscr{M} + 3D_{0}
 \end{aligned}
\end{equation}
for sufficiently large $n$.

Suppose now that $\bar{a}_{i(1)}$, $\bar{a}_{i(\lfloor m/K_{1} \rfloor)}$ satisfy \begin{equation}\label{eqn:alignThmC}
(\bar{a}_{i(\lfloor m/K_{1} \rfloor)}^{-1} o, vo)_{o} < C_{0}, \quad (\bar{a}_{i(1)}^{2} o, v^{-1} o)_{o} < C_{0}.
\end{equation}
Observe the following:
\begin{enumerate}
\item $[\bar{y}_{i(\lfloor m/K_{1} \rfloor), 1}^{-}, \bar{\w}_{n}\bar{y}_{i(1), 0}^{-}]$ is $D_{0}$-witnessed by ($[\bar{y}_{i(\lfloor m/K_{1} \rfloor), 1}^{-}, \bar{y}_{i(\lfloor m/K_{1} \rfloor), 0}^{-}], [\bar{\w}_{n}\bar{y}_{i(1), 2}^{-}, \bar{\w}_{n}\bar{y}_{i(1), 0}^{-}])$: this is due to Inequality \ref{eqn:largeIntermediate}, Lemma \ref{lem:farSegment}, and $d(o, \bar{a}_{i(\lfloor m/K_{1} \rfloor)}^{-1} o), d(o, \bar{a}_{i(1)}^{2} o) <  \mathscr{M}$.
\item $[\bar{y}_{i(1), 0}^{-}, \bar{y}_{i(\lfloor m/K_{1} \rfloor), 1}^{-} ]$ is fully $D_{0}$-marked with some sequences of Schottky segments, $(\gamma_{i})_{i=1}^{N-1}$ and $(\eta_{i})_{i=2}^{N}$, where $\gamma_{1} = [\bar{w}_{i(1), 0}^{-}o, \bar{w}_{i(1), 0}^{-} c o]$ and $\eta_{N} = [\bar{y}_{i(\lfloor m/K_{1} \rfloor), 2}^{-}, \bar{y}_{i(\lfloor m/K_{1} \rfloor), 1}^{-}]$. More explicitly, such sequences are provided by Proposition \ref{prop:concatUltAlign} since $i(1), i(\lfloor m/K_{1} \rfloor) \in Q_{n}(\bar{\w}) \subseteq P_{n}(\bar{\w})$.
\item $[\bar{y}_{i(\lfloor m/K_{1} \rfloor), 1}^{-}, \bar{y}_{i(\lfloor m/K_{1} \rfloor), 2}^{-}]$ and $[\bar{y}_{i(\lfloor m/K_{1} \rfloor), 1}^{-}, \bar{y}_{i(\lfloor m/K_{1} \rfloor), 0}^{-}]$ are $C_{0}$-glued.
\item $[\bar{w}_{i(1), 0}^{-} o, \bar{w}_{i(1), 0}^{-} co]$ and $[\bar{y}_{i(1), 0}^{-}, \bar{y}_{i(1), 2}^{-}]$ are $C_{0}$-glued. 
\end{enumerate}

Applying Corollary \ref{cor:induction}, these imply that \[
\ldots, \,\,\bar{\w}_{n}^{-1}\bar{y}_{i(1), 0}^{-}, \,\,\bar{\w}_{n}^{-1} \bar{y}_{i(\lfloor m/K_{1} \rfloor), 1}^{-}, \,\,\bar{y}_{i(1), 0}^{-}, \,\,\bar{y}_{i(\lfloor m/K_{1} \rfloor), 1}^{-}, \,\,\bar{\w}_{n}\bar{y}_{i(1), 0}^{-},\,\, \bar{\w}_{n} \bar{y}_{i(\lfloor m/K_{1} \rfloor), 1}^{-},\,\, \ldots 
\]
have Gromov products at most $F_{0}$ among points (in the right order). Hence, \[\begin{aligned}
\tau(\bar{\w}_{n}) &= \lim_{k} \frac{1}{k} d(\bar{y}_{i(1), 0}^{-}, \bar{\w}_{n}^{k} \bar{y}_{i(1), 0}^{-}) \\
&= \lim_{k} \frac{1}{k} \left[ d(\bar{y}_{i(1), 0}^{-}, \bar{\w}_{n} \bar{y}_{i(1), 0}^{-}) +  \sum_{j=2}^{k} \left[d(\bar{\w}_{n}^{j-1}\bar{y}_{i(1), 0}^{-}, \bar{\w}_{n}^{j} \bar{y}_{i(1), 0}^{-}) - 2(\bar{y}_{i(1), 0}^{-},  \bar{\w}_{n}^{j} \bar{y}_{i(1), 0}^{-})_{ \bar{\w}_{n}^{j-1} \bar{y}_{i(1), 0}^{-}}\right] \right] \\
&\ge d(\bar{y}_{i(1), 0}^{-}, \bar{\w}_{n} \bar{y}_{i(1), 0}^{-})- 2F_{0} \ge d(o, \bar{\w}_{n} o) - 2 d(o, \bar{y}_{i(1), 0}^{-}) - 2F_{0}
\end{aligned}
\]
and $d(o, \bar{\w}_{n}, o) - \tau(\bar{\w}_{n}) \le  4\lambda m + 4 \lfloor m/K_{1} \rfloor(\mathscr{M} + F_{0} ) +2 \mathscr{M} + 2F_{0}\le 5\lambda m$ for sufficiently large $n$: $\bar{\w} \notin F_{n}$ in this case.

In summary, at least $303^{2}$ choices of $(\bar{a}_{i(1)}, \bar{a}_{i(\lfloor m/K_{1} \rfloor)})$ (that satisfy the conditions in Lemma \ref{lem:pivotEquiv} and Inequality \ref{eqn:alignThmC}) are for $\bar{\w} \in \mathcal{E}^{m/K_{1}}(\w) \setminus F_{n}$. Now suppose that $(\bar{a}_{i(1)}, \bar{a}_{i(\lfloor m/K_{1} \rfloor)})$ are chosen from the remaining choices, of number $305^{2} -303^{2}$ in maximum. In each case, we similarly deduce that at least $303^{2}$ choices of $(\bar{a}_{i(2)}, \bar{a}_{i(\lfloor m/K_{1} \rfloor-1)})$ are for $\bar{\w} \in \mathcal{E}^{m/K_{1}} (\w) \setminus F_{n}$ and at most $305^{2} - 303^{2}$ choices remain. Continuing this, we deduce that \[
\Prob(F_{n} | \mathcal{E}^{m/K_{1}}) \le \left( \frac{305^{2} - 303^{2}}{304^{2}} \right)^{\lfloor m/K_{1} /2 \rfloor} \le 50 \cdot (0.02)^{m/K_{1} \log n} \le 50 n^{-2}
\]
for sufficiently large $n$. Summing them up for various $\mathcal{E}^{m/K_{1}}(\w)$, we deduce that $\Prob(F_{n}) \le 50 n^{-2}$. By Borel-Cantelli, a.e. $\w$ eventually avoids $F_{n}$.

Suppose now that $\w$ avoids $F_{k}$ eventually but $d(o, \w_{n} o) -\tau(\w_{n}) \ge 5 M \log n$ for infinitely many $n$. It means that either $|Q_{n}(\w)| < n/K_{1}$ infinitely often or $|d(o, \w_{n} o) - \lambda n| \ge 0.5 \lambda n$ infinitely often. The first one happens in probability zero since $\Prob\{|Q_{n}(\w)| < n/K_{1}\}$ is summable, and the second one happens in probability zero by the subadditive ergodic theorem.
\end{proof}

A crucial ingredient of the previous proof is that the distances between $o$ and eventual pivots increase linearly, which is a consequence of the subadditive ergodic theorem. 

In fact, regardless of the choices until step $n$, the next eventual pivotal time after step $n$ appears soon (with exponentially decaying error probability). One might hope that this  serves to prove Theorem \ref{thm:tracking}. However, despite punctual appearance of pivotal \emph{times}, we cannot assure that the \emph{distance} between the $n$-th position and the forthcoming pivot \emph{locus} is proportional to the time. Since the reference step $n$ changes, we cannot apply the subadditive ergodic theorem here. Hence, we pursue a different approach.

\begin{lem}\label{lem:GromPivot}
There exists $K_{2} > 0$ such that for any $g_{k+1}\in G$ and $x \in X_{\ge \epsilon}$, \[
\Prob\left.\left[\sup_{n\ge k} (x, \w_{n} o)_{o} \ge d(o, \w_{k} o) \right| g_{k+1} \right] \le K_{2}e^{-k/K_{2}}.\]

Moreover, for any $n \ge k$, $g_{k+1}, \ldots, g_{n} \in G$ and $x \in X_{\ge \epsilon}$, we also have \[
\Prob\left.\left[(x, \w_{n} o)_{o} \ge d(o, \w_{k} o) \right| g_{k+1}, \ldots, g_{n} \right] \le K_{2}e^{-k/K_{2}}.\]

Finally, given any $\check{g}_{1}, \ldots, \check{g}_{k+1}\in G$ in addition, we have \[\Prob\left.\left[\limsup_{n} (\check{\w}_{n} o, \w_{n} o)_{o} \ge d(o, \w_{k} o) \right| g_{k+1}, \check{g}_{1}, \ldots, \check{g}_{k+1}\right] \le K_{2} e^{-k/K_{2}}.\]
\end{lem}

\begin{proof}
We recall the model in Subsection \ref{subsection:1stModel}. This time, we temporarily fix the choices of $g_{6N\lfloor k/6N \rfloor+1}, \ldots, g_{6N(\lfloor k/6N \rfloor+1)}$ and exclude them from the potential pivotal time. This modification, for example, reduces each $\sumRho(n)$ by at most 1 so the overall estimate does not change. In particular, with the same $K_{1} > 0$ as in Proposition \ref{prop:expDecay}, we observe that \[
\Prob\left(|Q_{k}|\le k/K_{1} - 1\, \big| \, g_{6N\lfloor k/6N \rfloor+1}, \ldots, g_{6N(\lfloor k/6N \rfloor+1)}\right) \le K_{1} e^{- k/K_{1}}.
\]

Now for each $\w$ with $|Q_{k}(\w)| \ge k/K_{1} - 1$, we consider the equivalence class $\mathcal{E}^{ k/K_{1}-1}(\w)$ of trajectories pivoted from $\w$ at the first $\lceil k/K_{1} - 1\rceil$ eventual pivotal times $i(1)< \ldots< i(\lceil  k/K_{1} - 1 \rceil) \in Q_{k}(\w)$. For $\tilde{\w} \in \mathcal{E}^{k/K_{1}- 1}(\w)$, $i(1), \ldots, i(\lceil k/K_{1} - 1 \rceil)$ belong to $Q_{k}(\w) = Q_{k}(\tilde{\w}) = \cup_{n \ge k} P_{n} (\tilde{\w})$. Hence, we can apply Proposition \ref{prop:concatUlt}: for each $t = 1, \ldots, \lceil k/K_{1} - 1 \rceil$ and $n \ge k$,  $d(o, \tilde{\w}_{i(t), 0}^{-} o) \le d(o, \tilde{\w}_{n} o)$ and $[o, \tilde{\w}_{n} o]$ is $F_{0}$-witnessed by $[\tilde{w}_{i(t), 2}^{-}, \tilde{w}_{i(t), 0}^{-} o]$.

Now suppose that $(x, \tilde{\w}_{n} o)_{o} \ge d(o, \tilde{\w}_{k} o)$ for some $n \ge k$. Then we have \[
(x, \tilde{\w}_{n} o)_{o} \ge d(o, \tilde{\w}_{n} o) \ge d(o, \tilde{\w}_{i(1), 0}^{-} o).
\]
If this happens for two choices $a, a'$ of $\tilde{a}_{i(1)}$, then $[o, x]$ is $G_{0}$-witnessed by $[\tilde{w}_{i(1), 2}^{-} o, \tilde{w}_{i(1), 2}^{-} a^{2} o]$ and  $[\tilde{w}_{i(1), 2}^{-} o, \tilde{w}_{i(1), 2}^{-} a'^{2} o]$ by Lemma \ref{lem:passerBy}. Then Lemma \ref{lem:concatUlt} implies that $[x, x]$ is $F_{1}$-witnessed by $[\tilde{w}_{i(1), 2}^{-} o, \tilde{w}_{i(1), 2}^{-} a^{2} o]$, a contradiction.

Hence, there exists at most 1 choice of $\tilde{a}_{i(1)}$ that makes $(x, \tilde{\w}_{n} o)_{o} \ge d(o, \tilde{\w}_{k} o)$ for some $n \ge k$. By similar reasons, there exists at most 1 choice of $\tilde{a}_{i(1)}, \ldots, \tilde{a}_{i(\lceil k/K_{1}- 1\rceil)}$ for the desired case, out of at least $304^{\kappa_{1} k - 1}$ choices for $\mathcal{E}^{k/K_{1} - 1}(\w)$. Therefore, the conditional probability on each $\mathcal{E}^{k/K_{1} - 1}(\w)$ is at most $0.005^{\lceil k/K_{1} - 1\rceil}$ and we may take $K_{2} = \frac{1}{\log 200} (K_{1} + 200)$.

The second claim follows the same line of thought. This time, we fix the choices of $g_{6N\lfloor k/6N \rfloor+1}, g_{6N\lfloor k/6N \rfloor+2}, \ldots, g_{n}$  and construct pivotal times before $k$. That means, we modify the model in Subsection \ref{subsection:1stModel} by taking\[
w'_{\sumRho(\lfloor k  / 6N\rfloor)} = \nu_{\stopping(\sumRho(\lfloor k  / 6N\rfloor))+1}^{\ast} \cdots \nu_{\lfloor k  / 6N\rfloor}^{\ast} g_{6N\lfloor k  / 6N\rfloor+1} \cdots g_{n}.
\]
Still, the estimate for $\sumRho(\lfloor k/6N \rfloor)$ and $|P_{k}|$ do not change. In other words, we have $\Prob(|P_{k}| \le k/K_{1} | g_{k+1}, \ldots, g_{(i+1)N-1}) \le K_{1} e^{-k/K_{1}}$. Now we proceed as before on the event $|P_{k}| \ge k/K_{1}$ to deduce the conclusion.

For the final claim, we keep working on each $\mathcal{E}^{k/K_{1} - 1}(\w)$. Note that $d(o, \tilde{\w}_{k} o)$ for $\tilde{\w} \in \mathcal{E}^{k/K_{1}- 1}(\w)$ is bounded above, say by $M$.

Recall that a.e. $\check{\w}_{n}$ escapes to infinity and has infinitely many eventual pivots, even when $\check{g}_{1}, \ldots, \check{g}_{k+1}$ are fixed. Hence, we condition on the paths $\check{\w}$ such that $d(o, \check{w}_{\check{i}(m), 0}^{-} o) \ge M + 2(F_{0} + F_{1}) + 1$ for some eventual pivotal time $\check{i}(m)$; we loose zero probability by doing so. Let us now take $N$ such that $\check{P}_{n}(\check{\w})$ contains $\check{i}(m)$ for all $n \ge N$. For such $n$, Proposition \ref{prop:concatUlt} tells us that  $[o, \check{\w}_{n} o]$ is $F_{0}$-witnessed by $[\check{w}_{\check{i}(m), 2}^{-} o, \check{w}_{\check{i}(m), 0}^{-} o]$. Hence, for $n, n' \ge N$, we deduce from Fact \ref{fact:GromProdFact2} that \begin{equation}\label{eqn:contradictGromProd}
(\check{\w}_{n'} o, \w_{n} o)_{o} \ge d(o, \check{\w}_{\check{i}(m), 2}^{-} o) - 2F_{0} \ge d(o, \tilde{\w}_{k} o) + 2F_{1} + 1.
\end{equation}

We now claim that at most one choice of $\tilde{a}_{i(1)}$ is possible for $(\check{w}_{n} o, \tilde{\w}_{n} o)_{o} \ge d(o, \tilde{\w}_{k} o)$ to hold for some sufficiently large $n$. Suppose to the contrary that $\tilde{a}_{i(1)} = a, a'$ work for $n, n'$, respectively. Then $[o, \check{\w}_{n} o]$ is $G_{0}$-witnessed by $[\tilde{w}_{i(1), 2}^{-} o, \tilde{w}_{i(1), 2}^{-} a^{2} o]$ and $[o, \check{\w}_{n'} o]$ is $G_{0}$-witnessed by $[\tilde{w}_{i(1), 2}^{-} o, \tilde{w}_{i(1), 2}^{-} a'^{2}o]$. Lemma \ref{lem:concatUlt} implies that $(\check{\w}_{n} o, \check{\w}_{n'} o)_{o} \le d(o, \tilde{w}_{i(1), 2}^{-} o) + F_{1}$. Also, since $i(1) \in P_{k}(\tilde{\w})$, we have $d(o, \tilde{\w}_{i(1), 2}^{-} o) + F_{1} \le d(o, \tilde{\w}_{k} o) + F_{1}$. This contradicts Inequality \ref{eqn:contradictGromProd}.

In a similar way, we deduce that there is at most 1 combination of $\tilde{a}_{i(1)}, \ldots, \tilde{a}_{i(\lceil k/K_{1} - 1 \rceil)}$ and the same conclusion follows.
\end{proof}

When $X$ is a geodesic space, we have proven the following stronger result. Outside an event with probability less than $K_{2} e^{-k/K_{2}}$, we have \[\begin{aligned}
(x, \w_{n} o)_{o} &\le d(o, w_{i(\lceil k/K_{1} - 1\rceil), 0}^{-} o)\le d(o, w_{i(\lceil k/K_{1} - 1 \rceil), 0}^{+} o),\\
(o, x)_{\w_{n} o} &\ge d(o, \w_{n} o) - d(o, w_{i(\lceil k/K_{1} - 1 \rceil), 0}^{+} o) \ge d(\w_{n} o, w_{i(\lceil k/K_{1} - 1 \rceil), 0}^{+} o) - 2F_{0}.
\end{aligned}
\]
Here again, we are using Proposition \ref{prop:concatUlt} for the inequalities among distances. This implies that $[x, \w_{n} o]$ is $G_{0}$-witnessed by $[w_{i(\lceil \kappa_{1} k - 1\rceil), 0}^{+} o,  w_{i(\lceil \kappa_{1} k - 1\rceil), 2}^{+} o]$ and
\[
d(o, [x, \w_{n} o]) \le d(o, w_{i(\lceil \kappa_{1}k - 1\rceil), 0}^{+} o) + G_{0} + 6\delta \le d(o, \w_{k} o)
\] for $n \ge k$. Similarly, $d(o, [\check{\w}_{n} o, \w_{n} o]) \le d(o, \w_{k} o)$ holds for large enough $n$.

We now prove a deviation inequality between independent random paths with doubled exponent. This was observed for $p = 2$ in \cite{mathieu2020deviation}, when $G$ is a hyperbolic group acting on its Cayley graph.

\begin{prop}\label{prop:dominantGrom2}
Suppose that $\mu$ has finite $p$-moment for some $p > 0$. Then there exists $K > 0$ such that \[
\E\left[\limsup_{n} (\check{\w}_{n}o, \w_{n}o)_{o}^{2p}\right] < K.
\]
If $X$ is geodesic, we also have\[
 \E\left[\limsup_{n} d(o, [\check{\w}_{n} o, \w_{n} o])^{2p}\right] < K.\]
\end{prop}

\begin{proof}
For later purpose, let us fix $0 \le q \le p$.
Let $D(\check{\w}_{n}, \w_{n}) := (\check{\w}_{n} o, \w_{n} o)_{o}^{p+q}$ or $d(o, [\check{\w}_{n} o, \w_{n} o])^{p+q}$, and $D_{p, q}(\check{\w}_{n}, \w_{n}) := d(o,\check{\w}_{n} o)^{p} d(o, \w_{n} o)^{q}$. Recall that both $(\check{\w}_{n} o, \w_{n} o)_{o}$ or $d(o, [\check{\w}_{n} o, \w_{n} o])$ are smaller than both $d(o, \w_{n} o)$ and $d(o, \check{\w}_{n} o)$. Hence, we have \[
D(\check{\w}_{n}, \w_{n}) \le \min( d(o, \w_{n}o), d(o, \check{\w}_{n} o))^{p+q} \le d(o,\check{\w}_{n} o)^{p} d(o, \w_{n} o)^{q} = D_{p, q}(\check{\w}_{n}, \w_{n}).
\]

We next claim that\[
f(\check{\w}, \w):=\sum_{k=0}^{\infty}  \left|D_{p, q}(\check{\w}_{k+1}, \w_{k+1}) - D_{p, q}(\check{\w}_{k}, \w_{k})\right| 1_{\limsup_{n} D(\check{\w}_{n}, \w_{n})  \ge D_{p, q}(\check{\w}_{k+1}, \w_{k+1})}
\]
dominates $\limsup_{n} D(\check{\w}_{n}, \w_{n})$. 

If $D_{p, q}(\check{\w}_{i}, \w_{i}) \le \limsup_{n} D(\check{\w}_{n}, \w_{n}) \le D_{p, q}(\check{\w}_{i+1}, \w_{i+1})$ happens for the first time at some $i$, then \[\begin{aligned} 
\limsup_{n} D(\check{\w}_{n}, \w_{n})&\le D_{p, q}(\check{\w}_{i+1}, \check{\w}_{i+1}) 
= \sum_{k=0}^{i}  \left[D_{p, q}(\check{\w}_{i+1}, \w_{i+1}) - D_{p, q}(\check{\w}_{i}, \w_{i})\right]
\end{aligned}\] 
is bounded by $f(\check{\w}, \w)$. If not, $D_{p, q}(\check{\w}_{k}, \w_{k})  \le \limsup_{n} D(\check{\w}_{n}, \w_{n})$ holds for all $k$ and
\[\begin{aligned}
D(\check{\w}_{i}, \w_{i})& \le d(o, \check{\w}_{i}o)^{p} d(o, \w_{i}o)^{q}=\sum_{k=0}^{i-1}  \left[D_{p, q}(\check{\w}_{k+1}, \w_{k+1}) - D_{p, q}(\check{\w}_{k}, \w_{k})\right] \le f(\check{\w}, \w)
\end{aligned}\]for each $i$. Since $\limsup_{n} D(\check{\w}_{n}, \w_{n}) \le \sup_{i} D(\check{\w}_{i}, \w_{i})$, the claim follows.

The following is useful: for any $r > 0$ and $k \in \Z_{>0}$, we have \begin{equation}\label{eqn:cuteExp}\begin{aligned}
\E[d(o, \w_{k} o)^{p}] & \le \E\left[ \left(\sum_{i=1}^{k} d(o, g_{i} o) \right)^{p}\right]\\
&\le \E\left[\left(k \cdot \max_{1 \le i \le k} d(o, g_{i} o)\right)^{r}\right] \\
&\le \E \left[ k^{r} \sum_{i=1}^{k} d(o, g_{i} o)^{r} \right] \le k^{r+1} \E_{\mu} [d(o, go)^{r}].
\end{aligned}
\end{equation}

Recall now that for $t, s \ge 0$, we have \begin{equation}\label{eqn:diffIneqSimple}
\begin{aligned}
|t^{p} - s^{p}| \le \left\{\begin{array}{cc} |t-s|^{p} & p\le 1, \\ 2^{p}\left( |t-s|^{p} + s^{p-1} |t-s| \right) & p > 1. \end{array}\right.
\end{aligned}
\end{equation}
This implies that for $t_{i}, s_{i} \ge 0$ we have \begin{equation}\label{eqn:diffIneq}\begin{aligned}
|t_{1}^{p} t_{2}^{q} - s_{1}^{p} s_{2}^{q}| &=| t_{1}^{p}(t_{2}^{q} - s_{2}^{q}) + (t_{1}^{p} - s_{1}^{p})s_{2}^{q}| \\
&\le 2^{p+q} \left( |t_{1} - s_{1}|^{p} + s_{1}^{p-n_{p}} |t_{1} - s_{1}|^{n_{p}} +s_{1}^{p} \right) \left( |t_{2} - s_{2}|^{q} + s_{2}^{q-n_{q}} |t_{2} - s_{2}|^{n_{q}} \right) \\
& + 2^{p} \left( |t_{1} - s_{1}|^{p} + s_{1}^{p-n_{p}} |t_{1} - s_{1}|^{n_{p}}\right) s_{2}^{q}. \\
&\left(n_{p} = \left\{\begin{array}{cc}p & 0\le p \le 1 \\ 1 & p>1 \end{array}\right., \quad n_{q} = \left\{\begin{array}{cc}q &0\le q \le 1 \\ 1 & q>1\end{array}\right.\right)
\end{aligned}
\end{equation}

Thanks to Inequality \ref{eqn:diffIneq}, it now suffices to control the expectations of \[\begin{aligned}
&f_{k; n_{1}, n_{2}}(\check{\w}, \w) \\
&:= d(o, \check{g}_{k+1}o)^{n_{1}}d(o, g_{k+1}o)^{n_{2}} d(o, \check{\w}_{k} o)^{p-n_{1}} d(o, \w_{k}o)^{q-n_{2}}1_{\limsup_{n} D(\check{\w}_{n}, \w_{n})  \ge D_{p, q}(\check{\w}_{k+1}, \w_{k+1})}
\end{aligned}
\]
for 8 combinations of $(n_{1}, n_{2})$ such that $0\le n_{1} \le p$, $0\le n_{2} \le q$ and $n_{1} + n_{2} \ge \min(q, 1)$. Let us take $c = e^{1/2pK_{2}}$ and fix $g_{k+1}$, $\check{g}_{k+1}$ at the moment. We claim $Y\le Y_{1} + Y_{2} + Y_{3} + Y_{4}$ where \[
\begin{aligned}
Y &:=  d(o, \check{\w}_{k} o)^{p-n_{1}} d(o, \w_{k}o)^{q-n_{2}}1_{\limsup_{n} D(\check{\w}_{n}, \w_{n})  \ge D_{p, q}(\check{\w}_{k+1}, \w_{k+1})},\\
Y_{1} &:= d(o, \check{\w}_{k} o)^{p-n_{1}} c^{k(q-n_{2})} 1_{\limsup_{n} D(\check{\w}_{n}, \w_{n}) \ge d(o, \w_{k} o)^{p+q}}, \\
Y_{2} &:= c^{k(p-n_{1})} d(o, \w_{k}o)^{q-n_{2}} 1_{\limsup_{n} D(\check{\w}_{n}, \w_{n}) \ge d(o, \check{\w}_{k} o)^{p+q}},\\
Y_{3} &:= d(o, \check{\w}_{k} o)^{p-n_{1}} d(o, \w_{k} o)^{q-n_{2}} 1_{d(o, \check{\w}_{k} o), d(o, \w_{k} o) \ge c^{k}}.
\end{aligned}
\]
First observe that $Y \le Y_{1}$ when $d(o, \check{\w}_{k} o) , c^{k} \ge d(o, \w_{k} o)$. Also, $Y \le Y_{2}$ when $d(o, \w_{k} o), c^{k} \ge d(o, \check{\w}_{k} o)$. In the remaining cases, $d(o, \check{\w}_{k} o), d(o, \w_{k} o) \ge c^{k}$ and $Y \le Y_{3}$.

Let us estimate each $\E[Y_{i} | g_{k+1}, \check{g}_{k+1}]$. We further fix $\check{g}_{1}, \ldots, \check{g}_{k}$ in addition to $g_{k+1}, \check{g}_{k+1}$, and pivot on the forward path $\w$ (before step $k$). Lemma \ref{lem:GromPivot} tells us that \[
\Prob[\limsup_{n} D(\check{\w}_{n}, \w_{n}) \ge d(o, \w_{k} o)^{p+q} | g_{k+1}, \check{g}_{1}, \ldots, \check{g}_{k+1}] \le K_{2} e^{- k/K_{2}},
\]
and we now integrate $d(o, \check{\w}_{k} o)^{p-n_{1}}$ for various $\check{g}_{1}, \ldots, \check{g}_{k}$ to deduce \[
\E[Y_{1}|g_{k+1}, \check{g}_{k+1}] \le K_{2} c^{-2kp} \cdot c^{k(q-n_{2})} \cdot \E[d(o, \check{\w}_{k}o)^{p-n_{1}}] \le K_{2} c^{-k(2p-q+n_{2})} k^{p-n_{1}+1} \E_{\mu}[d(o, go)^{p-n_{1}}].
\]
Similarly, we obtain \[
\E[Y_{2}|g_{k+1}, \check{g}_{k+1}] \le K_{2} c^{-2kp} \cdot c^{k(p-n_{1})} \cdot \E[d(o, \w_{k}o)^{q-n_{2}}] \le K_{2} c^{-k(p+n_{1})} k^{q-n_{2}+1} \E_{\mu}[d(o, go)^{q-n_{2}}].
\]
Finally, using the independence of events for $\w$ and $\check{\w}$, we compute \[\begin{aligned}
\E[Y_{3}] &= \E[d(o, \check{\w}_{k} o)^{p} \cdot d(o, \check{\w}_{k} o)^{-n_{1}} 1_{d(o, \check{\w}_{k} o) \ge c^{k}}] \cdot \E[d(o, \w_{k} o)^{p} \cdot d(o, \w_{k} o)^{q-p-n_{2}} 1_{d(o, \w_{k} o) \ge c^{k}}]\\
& \le \E[d(o, \check{\w}_{k} o)^{p}] \E[d(o, \w_{k} o)^{p}] \cdot c^{-k(p-q+n_{1}+n_{2})} \le k^{p+q+2} \left(\E_{\mu}[d(o, g o)^{p}]\right)^{2} \cdot c^{-k(p-q+n_{1}+n_{2})}.
\end{aligned}
\]

With these bounds, we now multiply $d(o, g_{k+1}o)^{n_{1}} d(o, \check{g}_{k+1}o)^{n_{2}}$ and integrate to deduce that \[
\E[f_{k; n_{1}, n_{2}}] \le (5K_{2} + 1) \left(1+ \E_{\mu}[d(o, go)^{p}]\right)^{4} c^{-k[p-q+\min(q, 1)]} k^{p+q+2}.
\]
This is summable for $q=p$ so the conclusion follows.
\end{proof}

\begin{remark}
We note a similar result, due to Benoist and Quint, regarding the deviation inequality between a fixed boundary point and a random path.

\begin{prop}[cf. {\cite[Proposition 5.1]{benoist2016central}}]\label{prop:dominantGrom}
Let $X$ be a Gromov hyperbolic space and suppose that $\mu$ has finite $p$-moment for some $p > 0$. Then there exists $K > 0$ such that for any $x \in X \cup \partial X$ and $m \in \Z_{>0}$, we have \[
\E \left[ (x, \w_{m}o)_{o}^{p}\right], \quad \E \left[ \limsup_{n} (x, \w_{n}o)_{o}^{p}\right] < K.
\]
\end{prop}

Benoist and Quint proved this proposition for cocompact acitons (\cite[Proposition 5.1]{benoist2016central}) by using the spectral gap of the Markov operator on $X$. Our Lemma \ref{lem:GromPivot} now provides an alternative approach to Proposition \ref{prop:dominantGrom} which does not require the cocompactness assumption. 

Meanwhile, the exponent $p$ in the statement of Proposition \ref{prop:dominantGrom} is optimal and cannot improved further. This draws contrast with Proposition \ref{prop:dominantGrom2} with $2p$-exponent. The exponent doubling in Proposition \ref{prop:dominantGrom2} comes from the independence of two random directions involved. The probabilistic rationale behind this is that the minimum of two independent RVs with finite $p$-moment has finite $2p$-moment.
\end{remark}

We now present a similar estimation of the $2p$-moment of $(\check{\w}_{m} o, \w_{m'}o)_{o}$ for $m, m' \in \N$. Note that the bounds are uniform in the case $m = m'$. Although the bounds are not uniform for distinct $m$ and $m'$, they will be sufficient for the proof of the LIL in Section \ref{section:lil}.

\begin{prop}\label{prop:dominantGromEach}
Suppose that $\mu$ has finite $p$-moment for some $p>0$ and let $q\le p$ be a nonnegative integer. Then there exists $K>0$ such that \[
\E \left[(\check{\w}_{m}o, \w_{m'}o)_{o}^{p+q}\right] <K+K e^{-m/K} (m' - m)^{q},
\]and when $X$ is geodesic, \[
\E \left[d(o, [\check{\w}_{m}o, \w_{m'}o])^{p+q}\right] < K + K e^{-m/K} (m' - m)^{q}
\]
for all $0 \le m \le m'$, respectively.
\end{prop}

\begin{proof}
We replace $\limsup_{n} D(\check{\w}_{n}, \w_{n})$ in the proof of Proposition \ref{prop:dominantGrom2} with $D(\check{\w}_{m}, \w_{m'}) := (\check{\w}_{m} o, \w_{m'} o)^{p+q}$ or $d(o, [\check{\w}_{m} o, \w_{m'} o])^{p+q}$. Then we define \[
\begin{aligned}& f(\check{\w}, \w) := \sum_{k=0}^{m-1}  \left|D_{p, q}(\check{\w}_{k+1}, \w_{k+1})  - D_{p, q}(\check{\w}_{k}, \w_{k}) \right| 1_{D(\check{\w}_{m}, \w_{m'})\ge D_{p, q}(\check{\w}_{k}, \w_{k})},\\
&g(\check{\w}, \w) :=   \left|D_{p, q}(\check{\w}_{m}, \w_{m'})  - D_{p, q}(\check{\w}_{m}, \w_{m}) \right| 1_{D(\check{\w}_{m}, \w_{m'})\ge D_{p, q}(\check{\w}_{m}, \w_{m})}\\
\end{aligned}
\]
and observe that $f(\check{\w}, \w) + g(\check{\w}, \w) \ge D(\check{\w}_{m}, \w_{m'})$. We can then estimate $\E f(\check{\w}, \w)$ with a uniform bound as in Proposition \ref{prop:dominantGrom2} by constructing $Y$ and $Y_{i}$'s. Here the relevant fact is that \[
\Prob[ D(\check{\w}_{m}, \w_{m'}) \ge d(o, \w_{k} o)^{p+q} | g_{k+1}, \check{g}_{1}, \ldots,\check{g}_{k+1}] \le K_{2} e^{-k/K_{2}}
\]
for each $k \le m-1$ (and its symmetric counterpart), which now follows from the first part of Lemma \ref{lem:GromPivot}.

Meanwhile, $g(\check{\w}, \w)$ is dominated by a linear combination of \[
d(\w_{m} o, \w_{m'}o)^{n_{2}} d(o, \check{\w}_{m} o)^{p} d(o, \w_{m} o)^{q-n_{2}} 1_{D(\check{\w}_{m}, \w_{m'}) \ge D_{p, q}(\check{\w}_{m}, \w_{m})}
\] for $n_{2} = 1, q$. Let us take $c = e^{1/2pK_{2}}$. From the previous calculations, we obtain\[\begin{aligned}
&\E[d(o, \check{\w}_{m} o)^{p} d(o, \w_{m} o)^{q-n_{2}} 1_{D(\check{\w}_{m}, \w_{m'}) \ge D_{p, q}(\check{\w}_{m}, \w_{m})} | g_{m+1}, \ldots, g_{m'}]\\
\le &K_{2} c^{-m(2p-q+n_{2})} m^{p+1}  \E_{\mu}[d(o, go)^{p}] + K_{2} c^{-mp}m^{q-n_{2}} \E_{\mu}[d(o, go)^{q-n_{2}}] \\
& + 2K_{2} c^{-m(p-q+n_{2})} + m^{p+q+1} c^{-m(p-q+n_{2})} \left(\E_{\mu}[d(o, go)^{p}]\right)^{2}.
\end{aligned}
\]
Here the relevant facts are \[
\Prob[ (\check{\w}_{m}o, \w_{m'} o)_{o} \ge d(o, \w_{m} o) | g_{m+1}, \ldots, g_{m'}, \check{g}_{1}, \ldots, \check{g}_{m}]\le K_{2}e^{-m/K_{2}}
\]
and  \[
\Prob[ (\check{\w}_{m}o, \w_{m'} o)_{o} \ge d(o, \check{\w}_{m} o) | g_{1}, \ldots, g_{m'}] \le K_{2}e^{-m/K_{2}},
\]
which are the second item of Lemma \ref{lem:GromPivot}. We then multiply $d(\w_{m} o , \w_{m'} o)^{n_{2}} \le [\sum_{k=m+1}^{m'} d(o, g_{k} o)]^{n_{2}}$ and integrate to deduce\[
\E g(\check{\w}, \w) \le (5K_{2}+1) \left(1 + \E_{\mu}[d(o, go)^{p}]\right)^{3}
 c^{-m[p-q+ \min(q, 1)]}m^{p+q+1} \cdot (m'-m)^{n_{2}}
\]
(here we use the fact that $n_{2}=1, q$ are integers). Since $p>0$ and $0\le q \le p$, $p-q + \min(q, 1)$ is positive and we get the desired estimate.
\end{proof}

A similar argument is available for measures with finite exponential moment. Although Proposition \ref{prop:dominantGromEach} has its analogy in this setting, we only discuss the analogy of Proposition \ref{prop:dominantGrom2} that is relevant to Theorem \ref{thm:tracking}.

\begin{prop}\label{prop:dominantGromExp}
Suppose that $\E_{\mu}[e^{c d(o, go)}] < \infty$ for some $c>0$. Then there exists $K > 0$ such that \[
\E \left[ \limsup_{n} e^{K(\check{\w}_{n} o, \w_{n}o)_{o}}\right] < K, 
\] and when $X$ is geodesic,  \[
\E \left[ \limsup_{n} e^{ Kd(o, [\check{\w}_{n} o, \w_{n} o])} \right]< K. \]
\end{prop}
\begin{proof}
We take $M = 4\max(1, \log \E_{\mu} [e^{cd(o, go)}] + 1)/c$ and $K = \min(c, 1/2MK_{2})$. Note that $KM \le 1/2K_{2}$ and $M(c-K/2) \ge Mc/2\ge 2\log \E_{\mu} [e^{cd(o, go)}] +1$.

We again set $D(\check{\w}_{n}, \w_{n}) = (\check{\w}_{n} o, \w_{n} o)_{o}$ or $d(o, [\check{\w}_{n} o, \w_{n} o])$, respectively. Then the desired variables are bounded by \begin{equation}\label{eqn:dominantGromExp}
f(\check{\w}, \w) := \sum_{k=0}^{\infty}e^{K [d(o, \w_{k+1} o) + d(o, \check{\w}_{k+1} o)]/2} 1_{\limsup_{n}  D(\check{\w}_{n}, \w_{n}) > \frac{d(o, {\w_{k} o})+d(o, \check{\w}_{k} o)}{2}}.
\end{equation}
Indeed, if $d(o, \w_{i} o) + d(o, \check{\w}_{i} o) < 2\limsup_{n} D(\check{\w}_{n}, \w_{n}) \le d(o, \w_{i+1} o) + d(o, \check{\w}_{i+1} o)$ holds for the first time at $i$, then \[
f(\check{\w}, \w) \ge e^{K [d(o, \w_{i+1} o) + d(o, \check{\w}_{i+1} o)]/2} \ge \limsup_{n} e^{K D(\check{\w}_{n}, \w_{n})}.
\]
If such $i$ does not exist, then we have \[
f(\check{\w}, \w) \ge e^{K [d(o, \w_{i} o) + d(o, \check{\w}_{i} o)]/2} \ge e^{K\min (d(o, \w_{i} o), d(o, \check{\w}_{i} o))} \ge e^{K D(\check{\w}_{i}, \w_{i})}
\]
for each $i$. Taking the limit supremum, we have $f(\check{\w}, \w) \ge e^{K \limsup_{n} D(\w_{n}, \w_{n})}$.

To estimate $f(\check{\w}, \w)$, let us use the decomposition \[
d(o, \w_{k+1} o) + d(o, \check{\w}_{k+1} o) \le d(o, g_{k+1} o) + d(o, \check{g}_{k+1} o) + d(o, \w_{k} o) + d(o, \check{\w}_{k} o).
\]
Given this, it suffices to control the expectation of
\[
 \sum_{k =0}^{\infty} e^{Kd(o, g_{k+1} o)/2}e^{Kd(o, \check{g}_{k+1} o)/2} e^{K d(o, \w_{k} o)/2}e^{K d(o, \check{\w}_{k} o)/2} 1_{\limsup_{n}  D(\check{\w}_{n},\w_{n}) > \frac{d(o, {\w_{k} o})+d(o, \check{\w}_{k} o)}{2}} .
\]

We fix $g_{k+1}$, $\check{g}_{k+1}$ and define $E_{k} := \{ d(o, \w_{k} o) \ge M k\}$. Then \[\begin{aligned}
\E[ e^{K d(o, \check{\w}_{k} o)/2} e^{K d(o, \w_{k} o)/2} 1_{E_{k}}] &\le \E[e^{c d(o, \check{\w}_{k} o)} e^{c d(o, \w_{k} o)} e^{(K/2-c) d(o, \w_{k} o)} 1_{E_{k}}]\\
& \le \E [e^{c \sum_{i=1}^{k}  d(o, \check{g}_{i} o)} e^{c \sum_{i=1}^{k} d(o, g_{i} o)}] e^{M(K/2 -c) k}\\
& \le \E_{\mu}[e^{c d(o, g o)}]^{2k} e^{M(K/2-c) k} \le e^{-k}.
\end{aligned}
\]
Similar estimate holds on $\check{E}_{k} := \{d(o, \check{\w}_{k} o) \ge Mk\}$. On $(E_{k} \cup \check{E}_{k})^{c}$ we have
\[\begin{aligned}
\E[e^{K d(o, \check{\w}_{k} o)/2} e^{K d(o, \w_{k} o)/2}1_{\limsup_{n} D(\check{\w}_{n}, \w_{n}) \ge d(o, \w_{k} o), E_{k}^{c} \cap \check{E}_{k}^{c}}]& \le e^{K M k} K_{2} e^{- k/K_{2}},\\
\E[e^{K d(o, \check{\w}_{k} o)/2} e^{K d(o, \w_{k} o)/2}1_{\limsup_{n} D(\check{\w}_{n}, \w_{n}) \ge d(o, \check{\w}_{k} o), E_{k}^{c} \cap \check{E}_{k}^{c}}]& \le e^{K M k} K_{2} e^{-k/K_{2}}.
\end{aligned}
\]

Overall, the conditional expectation of the $k$-th summand is dominated by $2K_{2}e^{-k/2K_{2}} + 2e^{-k}$, uniformly on the choice of $g_{k+1}$, $\check{g}_{k+1}$. We then multiply $e^{Kd(o, g_{k+1} o)/2} \le e^{cd(o, g_{k+1} o)}$ and $e^{Kd(o,\check{g}_{k+1}o)/2 } \le e^{cd(o, \check{g}_{k+1} o)}$, integrate for $g_{k+1}$ and $\check{g}_{k+1}$, and sum up for $k$ to obtain the desired bound.
\end{proof}

We are now ready to prove Theorem \ref{thm:tracking}.

\begin{proof}
We focus on the sample paths that have infinitely many forward/backward pivotal times $\{i(1), i(2), \ldots\}$ and $\{\check{i}(1), \check{i}(2), \ldots\}$. Recall that almost every path is so, and pivotal loci escape to infinity in such paths. We then define $x_{0} = o$ and $(x_{2l-1}, x_{2l}) = (y_{i(l), 0}^{-}, y_{i(l), 0}^{+})$ as in Proposition \ref{prop:concatUlt}, and concatenate segments $[x_{0}, x_{1}]$, $[x_{1}, x_{2}]$, $[x_{2}, x_{3}]$, $\ldots$ into a path $\Gamma(\w)$.

We claim that $\Gamma(\w)$ is an $\left( 1 + \frac{8F_{0}}{L_{0}}, 2F_{0} + 2D_{3}\right)$-quasi-geodesic. To show this, consider points $z \in [x_{i-1}, x_{i}]$ and $z' \in [x_{j}, x_{j+1}]$. Without loss of generality we may assume $i-1 \le j$. When $i-1=j$, the portion of $\Gamma(\w)$ between $z$ and $z'$ is geodesic so we are done. 

When $i = j$, we have \[\begin{aligned}
(z, z')_{x_{i}} &= \frac{1}{2} [d(z, x_{i}) + d(x_{i}, z') - d(z, z')]\\
& = \frac{1}{2} [d(x_{i-1}, x_{i}) + d(x_{i}, x_{i+1}) - (d(x_{i-1}, z) + d(z, z') + d( z', x_{i+1}))]\\
& \le \frac{1}{2} [d(x_{i-1}, x_{i}) + d(x_{i}, x_{i+1}) - d(x_{i-1}, x_{i+1})] =(x_{i-1}, x_{i+1})_{x_{i}} \le F_{0}.
\end{aligned}
\]
Thus, the length $d(z, x_{i}) + d(x_{i}, z')$ of the portion of $\Gamma(\w)$ between $z$ and $z'$ is bounded by $d(z, z') + 2F_{0}$. 

When $i< j$, we have $(z, x_{j})_{x_{i}} \le (x_{i-1}, x_{j})_{x_{i}} \le F_{0}$ and $(z, x_{j+1})_{x_{i}} \le (x_{i-1}, x_{j+1})_{x_{i}} \le F_{0}$. Together with $(x_{i}, x_{j+1})_{x_{j}} \le F_{0}$, we deduce that $[z, x_{j+1}]$ is $D_{3}$-witnessed by $([z, x_{i}], [x_{j}, x_{j+1}])$ and $(z, z')_{x_{j}} \le (z, x_{j+1})_{x_{j}} \le D_{3}$ holds. Note also $(x_{i}, x_{k+1})_{x_{k}}< F_{0}$ for $k=i+1, \ldots, j-1$. These imply \[
\left| \left[d(z, x_{i}) + \sum_{k=i}^{j-1} d(x_{k} ,x_{k+1}) + d(x_{j}, z') \right] - d(z, z') \right| \le 2F_{0}(j-i) + 2D_{3}.
\]
Since $d(x_{k}, x_{k+1}) \ge \frac{L_{0}}{2}$ for each $k$, we deduce that $d(z, z') \ge (\frac{L_{0}}{2}- 2F_{0}) (j-i)  - 2D_{3} \ge \frac{L_{0}}{4}(j-i) - L_{0}/8$. This in turn implies \[\begin{aligned}
d(z, x_{i}) + \sum_{k=i}^{j-1} d(x_{k} ,x_{k+1}) + d(x_{j}, z')& \le d(z, z') + 2F_{0} (j-i) + 2D_{3}\\
& \le \left(1 + 8F_{0}/L_{0}\right) d(z, z') + F_{0} + 2D_{3}.
\end{aligned}
\]

Now we assume that $\E_{\mu}[d(o, go)^{p}] <+\infty$ or $\E_{\mu}[e^{c d(o, go)}]< +\infty$ for some $c>0$. In the first case, we set $f(k) = k^{1/2p}$ and arbitrary $C>0$. In the second case, we set $f(k) = \log k$ and $C = 1/K$ for the $K$ obtained in Proposition \ref{prop:dominantGromExp}. We define  \[
E_{k} := \{(\check{\w}, \w): \limsup_{n} d(\w_{k} o, [\w_{k-n}o, \w_{k+n} o]) \le Cf(k)\}.\] Note that for each $k$, $\limsup_{n} d(\w_{k}o, [\w_{k-n}o, \w_{k+n} o])$ has the same law with $Y = \limsup_{n} d(o, [\check{\w}_{n}o, \w_{n} o])$. Thus, Proposition \ref{prop:dominantGromExp} and \ref{prop:dominantGrom2} imply that
\[
\sum_{k=1}^{\infty} \Prob(E_{k}^{c}) \le \sum_{k=1}^{\infty} \Prob(Y \ge Cf(k)) <\infty
\] and $\Prob[\liminf_{k}E_{k}]=1$ by Borel-Cantelli.

By the preceding argument and Lemma \ref{lem:eventualBiInf}, we may suppose that: \begin{enumerate}
\item there exists $m$ for $(\check{\w}, \w)$ as in Lemma \ref{lem:eventualBiInf}, and
\item $\w \in E_{k}$ for sufficiently large $k$.
\end{enumerate}
Then $\w \in E_{k}$, $|Q_{k}(\w)| > m$ and $Cf(k) \ge F_{0} + F_{2}$ hold for large enough $k$; fix such $k$. Since pivotal loci escape to infinity, there exists $M$ such that \[
d(o, y_{i(M), 0}^{+}) \ge 2Cf(k) + d(o, \w_{k} o) + F_{0} + F_{2}+ 1.
\]
Finally, we take $n$ such that: \begin{enumerate}
\item $|\check{Q}_{n-k}(\check{\w})| \ge m$ and $|Q_{n+k}(\w)| \ge M$, and 
\item $d([\w_{k-n} o, \w_{k+n} o], \w_{k} o) \le 2Cf(k)$.
\end{enumerate} 

First note that $[\w_{k- n} o, \w_{k+n} o]$ is $F_{0}$-close to each of $x_{2m}, x_{2m+1}, \ldots, x_{2M}$ by Lemma \ref{lem:eventualBiInf}. Since $\w_{k-n} o$, $\w_{k+n} o$, $x_{i}$ are $\epsilon$-thick,  $[\w_{k-n} o, \w_{k+n} o]$ and \[
\Gamma' :=[\w_{k-n} o, x_{2m}] \cup [x_{2m}, x_{2m+1}] \cup \ldots \cup [x_{2M-1}, x_{2M}] \cup [x_{2M}, \w_{k+n} o]
\]
are within Hausdorff distance $F_{2}$. Hence $d(\w_{k} o, \Gamma') \le 2Cf(k) + F_{2}$.
Here, $(\w_{k-n} o, \w_{k} o)_{x_{2m}} < F_{0}$ by Lemma \ref{lem:eventualBiInf} so \[
d(\w_{k} o, [\w_{k-n} o, x_{2m}]) \ge (\w_{k-n} o, x_{2m})_{\w_{k} o} \ge d(x_{2m}, \w_{k} o) - F_{0}
\] holds. Meanwhile,\[\begin{aligned}
d(\w_{k} o, [x_{2M}, \w_{k+n} o]) &\ge (x_{2M}, \w_{k+n} o)_{\w_{k} o}\\
&\ge (x_{2M}, \w_{k+n} o)_{o} - d(o, \w_{k} o) \\
&= d( o, x_{2M}) - ( o, \w_{k+n} o)_{x_{2M}}  - d(o, \w_{k} o)\\
&\ge 2Cf(k) + F_{2}+1
\end{aligned}\]
holds since $i(M) \in P_{k+n}(\w)$. This implies that $d(\w_{k} o, \Gamma')$ is not achieved between $\w_{k} o$ and $[x_{2M}, \w_{k+n} o]$. Hence we deduce \[\begin{aligned}
 d(\w_{k} o, \Gamma) &\le d\left(\w_{k} o, [\w_{k-n} o, x_{2m}] \cup [x_{2m}, x_{2m+1}] \cup \ldots \cup [x_{2M-1}, x_{2M}] \right)+F_{0}\\
 &\le 2Cf(k) + F_{0} + F_{2} \le 3Cf(k).
 \end{aligned}
\]
Hence, $d(\w_{k} o, \Gamma) \le 3Cf(k)$ eventually holds and the conclusion follows.

We now further assume that $X$ is proper. As in Corollary \ref{cor:masur}, there exists a subsequence $[o, \w_{n_{i}} o]$ of $[o, \w_{n}o]$ that converges to a half-infinite geodesic ray $\Gamma_{0}$. It remains to show that $\Gamma_{0}$ and $\Gamma$ has bounded Hausdorff distance. Each $x_{j}$ is $F_{0}$-close to $[o, \w_{n} o]$ eventually, so $\Gamma_{0}$ also has points $x'_{j}$ with $d(x_{j}, x'_{j}) < F_{0}$. For definiteness, we take $x'_{0} = o = x_{0}$. Then $[x_{j}, x_{j+1}]$ and $[x'_{j}, x'_{j+1}]$ are $F_{2}$-fellow traveling by Theorem \ref{thm:rafi1}, so $\Gamma = \cup_{j \ge 0} [x_{j}, x_{j+1}]$ and $\Gamma_{0} = \cup_{j \ge 0} [x_{j}', x_{j+1}']$ are also within Hausdorff distance $F_{2}$.

Finally, when $X$ is a geodesic $\delta$-hyperbolic space that is not necessarily proper, we can take a $(1, 20\delta)$-quasigeodesic $\Gamma'$ that fellow travels with $\Gamma$, in view of \cite[Remark 2.16]{kapovich2002boundaries}.
\end{proof}

\section{Central limit theorems}\label{section:central}

In this section, we consider two variations of the model in Subsection \ref{subsection:1stModel} to prove a CLT for $d(o, \w_{n} o)$ and the converse of CLTs for $d(o, \w_{n} o)$ and $\tau(\w_{n})$. Note that the CLT for $\tau(\w_{n})$ then follows from Theorem \ref{thm:logCorr}.

\subsection{Converse of central limit theorems}

Throughout this subsection, we assume that $\E_{\mu} [d(o, g o)^{2}]= +\infty$. Let also $K>0$.

For each $g \in G$, there exists $a \in S_{0}$ such that $(o, ago)_{ao} = (go, a^{-1} o)_{o} \le C_{0}$. For that choice, there exists $b \in S_{0}$ such that $(o, agbo)_{ago}  = (g^{-1} a^{-1} o, bo)_{o}\le C_{0}$. In this case we say that $g \in A_{a, b}$. Since $\cup_{a, b \in S_{0}} A_{a, b} = G$, we deduce \[
 \sum_{a, b \in S_{0}} \E_{\mu} [d(o, g o)^{2}1_{g \in A_{a, b}}] \ge \E_{\mu} [d(o, g o)^{2}] = +\infty.
 \]
Hence, there exist $a, b \in S_{0}$ such that $\E_{\mu} [d(o, g o)^{2}1_{g \in A_{a, b}}] = +\infty$. We then take a subset $S$ of $S_{0} \setminus \{a, b\}$ with cardinality 305 and define  $\mu_{a, b}$ as $\mu$ conditioned on $A_{a, b}$, i.e., \[
\mu_{a, b}(g) = \left\{ \begin{array}{cc} \mu(g) /\mu(A_{a, b}) & g \in A_{a, b} \\ 0 & \textrm{otherwise}. \end{array}\right.
\]
The elements $g$ of $A_{a,b}$ are chosen so that $[o, agbo]$ is fully $D_{0}$-marked with Schottky segments $[ago, agbo]$ and $[o, ao]$. We also have \[\begin{aligned}
d(o, agbo) &= d(o, ao) + d(ao, ago) + d(ago, agbo) - 2(o, ago)_{ao} - 2(o, agbo)_{ago} \\
&\ge d(ao, ago) + 2L_{0} - 2C_{0} \ge d(o, go).
\end{aligned}
\]
From this, we deduce \[\E_{\mu_{a, b}}[d(o, agbo)^{2}] \ge \mu(A_{a, b})^{-1} \E_{\mu} [d(o, g o)^{2}1_{g \in A_{a, b}}] = +\infty.\]

We now define $\mu_{S^{(2)}}$, $1_{\{a\}}$ and $1_{\{b\}}$ as in Subsection \ref{subsection:1stModel} and consider the decomposition \[
\mu^{6N+1} = \alpha (\eta := \mu_{S^{(2)}} \times 1_{\{a\}} \times \mu_{a, b} \times 1_{\{b\}} \times \mu_{S^{(2)}}) + (1-\alpha) \nu
\]for some measure $\nu$ and $0 < \alpha < 1$. As in Subsection \ref{subsection:random}, we define RVs $\rho_{i}$, $\nu_{i}$, $\eta_{i}$, $\gamma_{i}$, $\sumRho(k)$, $\stopping(i)$. We also define $\alpha_{i}$ ($\beta_{i}$, resp.) as the product of the first (last, resp.) $N$ coordinates of $\eta_{i}$, and $\xi_{i}$ be the $(3N+1)$-th, middle coordinate. Then $\{\rho_{i}, \alpha_{i}, \beta_{i}, \xi_{i}, \nu_{i}\}$ all become independent.

We work in the setting similar to Equation \ref{eqn:pivotSetting}; we again let $k = \lfloor n  / 6N\rfloor$, $\gamma' = g_{6Nk + 1} \cdots g_{n}$ and observe \begin{equation}\label{eqn:pivotSetting2}
\w_{n} =  w_{0}\, \cdot \, a_{1}^{2} \cdot (a \xi_{\stopping(1)} b) \cdot b_{1}^{2} \, \cdot \, w_{1} \, \cdot \, a_{2}^{2} \cdot (a \xi_{\stopping(2)} b) \cdot b_{2}^{2} \, \cdots\, a_{\sumRho(k)}^{2} \cdot(a\xi_{\stopping(\sumRho(k))} b) \cdot b_{\sumRho(k)}^{2} \, \cdot \, w_{\sumRho(k)}'
\end{equation}
where $w_{i} = \nu_{\stopping(i) + 1}^{\ast} \cdots \nu_{\stopping(i+1) - 1}^{\ast}$, $a_{i} = \alpha_{\stopping(i)}$, $b_{i} = \beta_{\stopping(i)}$ and $w'_{\sumRho(k)} = \nu_{t(\sumRho(k)) + 1}^{\ast} \ldots \nu_{k}^{\ast} \gamma'$. Fixing the intermediate words $(w_{0}, \ldots, w_{\sumRho(k) - 1}, w_{\sumRho(k)}')$ and $(a \xi_{\stopping(1)} b, \ldots, a \xi_{\stopping(k)} b)$, we construct the set of pivotal times $P_{n}(\w) = P_{\sumRho(k)}(s)$ for $s = (a_{1}, b_{1}, \ldots, a_{k}, b_{k}) \in S^{2k}$ with the uniform measure.

The difference between the simple model in Subsection \ref{subsection:1stModel} and the present one lies in the different decomposition of $\mu^{N}$. This does not affect the linear increase of $\sumRho(k)$ with respect to $n$ outside a set of exponentially decaying probability. Once $\sumRho(k)$ and intermediate words are fixed, we can apply Proposition \ref{prop:pivotEstimate}. Hence, as in Proposition \ref{prop:expDecay}, $|P_{n}(\w)| > n/K_{1}$ holds outside an event with probability less than $K_{1}e^{-n/K_{1}}$. 

Moreover, we can bring the alignment obtained in Proposition \ref{prop:concatUltAlign}. Indeed, among the four items before Proposition \ref{prop:concatUltAlign}, items (1)-(3) does not depend on the character of intermediate words $v_{i}$'s and remain the same. We then replace item (4) with \begin{enumerate}[label=(\arabic*')]
\setcounter{enumi}{3}
\item for each $l=1, \ldots, |P_{n}(\w)|$, sequences of Schottky segments \[
\left( [y_{j(l), 0}^{-}, w_{j(l), 0}^{-} a o], [y_{j(l), 0}^{+}, y_{j(l), 2}^{+}]\right), \quad \left( [y_{j(l), 2}^{-}, y_{j(l), 0}^{-} ], [w_{j(l), 0}^{+}b^{-1} o, y_{j(l), 0}^{+}]\right)
\]
are $D_{0}$-aligned.
\end{enumerate}
This follows from the full $D_{0}$-marking of $[o, a\xi_{i} b o]$ with $[o, ao]$, $[a\xi_{i} o, a\xi_{i} bo])$. Hence the results of Proposition \ref{prop:concatUltAlign} and \ref{prop:concatUlt} follow: we set $x_{0} = o$, $x_{2|P_{n}(s)|+ 1} = \w_{n} o$, $(x_{2l-1}, x_{2l}) = (y_{i(l), 0}^{-}, y_{i(l), 0}^{+})$ for $l= 1, \ldots, |P_{n}(s)|$ and obtain $(x_{i}, x_{k})_{x_{j}} \le F_{0}$ for each triple $i \le j \le k$.

We are now ready to prove Theorem \ref{thm:centralConv}.

\begin{proof}[Proof of Theorem \ref{thm:centralConv}]
 
Let us first fix a sequence $(n_{m})_{m >0}$ such that $n_{m}/K_{1} >2^{m}$ and $\lim_{m} n_{m} / K_{1}2^{m} =1$. We will work on the product space $\Omega \times \dot{\Omega}$ of $\Omega$ and its copy $\dot{\Omega}$. In other words, together with the RVs $(\rho_{i}, \eta_{i}, \alpha_{i}, \beta_{i}, \ldots)$ of $(\w, \dot{\w}) \in \Omega \times \dot{\Omega}$ that depend only on $\w$, we consider an identical copy of RVs $(\dot{\rho}_{i}, \dot{\eta}_{i}, \dot{\alpha}_{i}, \dot{\beta}_{i}, \ldots)$ that depend only on $\dot{\w}$. We will investigate the RV  $\frac{1}{\sqrt{n}}[ d(o, \w_{n} o)-d(o, \dot{\w}_{n} o)]$. 

Fix $m$ and suppose that $\w_{0} \in \Omega_{m} := \{\w : |P_{n_{m}}(\w)| \ge 2^{m}\}$. Let us denote its first $2^{m}$ pivotal times by $i_{1}, \ldots, i_{2^{m}}$. We then declare the equivalence class of $\w_{0}$ by \[
\mathcal{E}(\w_{0}) = \left\{\w : \begin{array}{c} (\rho_{i}, \nu_{i}, \alpha_{i}, \beta_{i}, \gamma')(\w) = (\rho_{i}, \nu_{i}, \alpha_{i}, \beta_{i}, \gamma')(\w_{0}),\\ \xi_{i}(\w) = \xi_{i}(\w_{0})\,\,\textrm{unless}\,\, i = \stopping(i_{1}), \ldots, \stopping(i_{2^{m}})\end{array}\right\}.
\]
(Note that $\stopping(k)$'s depend only on $\{\rho_{i}(\w)\}$). This condition is an equivalence relation because of Lemma \ref{lem:pivotEquivInterm}: all $\w \in \mathcal{E}(\w_{0})$ have their first $2^{m}$ pivotal times $i_{1}, \ldots, i_{2^{m}}$. The following quantities are also uniform across $\mathcal{E}(\w_{0})$:
 \[
d(x_{0}(\w), x_{1}(\w)), \,\, \ldots, d(x_{2^{m+1} - 2}(\w), x_{2^{m+1}-1}(\w)), \,\,d(x_{2^{m+1}}(\w), \w_{n_{m}} o).
\]
For each $l = 0, \ldots, m$, let us now consider the \emph{dyadic Gromov products} \[
(x_{2^{l} \cdot 0}, x_{2^{l} \cdot 2})_{x_{2^{l} \cdot 1}}, (x_{2^{l} \cdot 2}, x_{2^{l} \cdot 4})_{x_{2^{l} \cdot 3}}, \ldots, (x_{2^{l} \cdot (2^{m+1-l} - 2)}, x_{2^{l} \cdot 2^{m+1-l}})_{x_{2^{l} \cdot (2^{m+1-l} - 1)}}.
\]
Conditioned on $\mathcal{E}(\w_{0})$, these are $2^{m-l}$ independent variables bounded by $F_{0}$, depending on disjoint groups of $\xi_{\stopping(i_{1})}, \ldots, \xi_{\stopping(i_{2^{m}})}$. Thus, their sum $Y_{l}$ satisfies $Var(Y_{l}|\mathcal{E}(\w_{0})) \le 2^{m-l} \cdot F_{0}^{2}$, and we have \[
\Prob\left\{\big|Y_{l} - \E[Y_{l}|\mathcal{E}(\w_{0})]\big| \ge 100F_{0}\cdot  2^{(m-0.5 l)/2}\,\Big| \,\mathcal{E}(\w_{0}) \right\} \le \frac{1}{10^{4} \cdot 2^{l/2}}
\]
by Chebyshev. We sum them up to deduce that \begin{equation}\label{eqn:converseIng1}
\sum_{l=0}^{m} |Y_{l} - \E[Y_{l}|\mathcal{E}(\w_{0})]| \le 800F_{0} \cdot 2^{m/2}
\end{equation} outside an event with probability at most $1/1000$. ($\ast)$

We perform the same construction with respect to $\dot{\Omega}$. Let $\mathscr{E}$ be the set of pairs $\mathcal{E} \times \dot{\mathcal{E}}$ of equivalence classes on $\Omega_{m}, \dot{\Omega}_{m}$ such that \begin{equation}\begin{aligned}\label{eqn:symmetryImport}
\sum_{i=0}^{2^{m}-1} d(x_{2i}, x_{2i+1}) + d(x_{2^{m+1}}, \w_{n_m} o) -2 \sum_{l=0}^{m} \E[Y_{l}|\mathcal{E}]\\
\ge \sum_{i=0}^{2^{m}-1} d(\dot{x}_{2i}, \dot{x}_{2i+1}) + d(\dot{x}_{2^{m+1}}, \dot{\w}_{n_m} o) -2\sum_{l=0}^{m} \E[\dot{Y}_{l}|\dot{\mathcal{E}}].
\end{aligned}
\end{equation}
Note that Inequality \ref{eqn:symmetryImport} is reversed by the measure-preserving symmetry $\w \leftrightarrow \dot{\w}$. This implies that $\Prob(\bigcup \mathscr{E}) \ge \frac{1}{2}\Prob(\Omega_{m} \times \dot{\Omega}_{m})$.

We now make use of the fact $\E[d(o, a\xi bo)^{2}] = +\infty$. By a truncation method recorded in e.g. Exercise 3.4.3 of \cite{durrett2019probability}, we have the following.

\begin{lem}
Let $Z$ be an RV with $\E[Z^{2}] = +\infty$ and $\{Z_{i}, Z_{i}'\}_{i=1}^{\infty}$ be i.i.d. copies of $Z$. Then for any $M>0$, there exists $N>0$ (that depends on $M$ and the distribution of $Z$) such that  \[
\Prob\left( \sum_{i=1}^{n} Z_{i} \ge \sum_{i=1}^{n} Z_{i}' + M \sqrt{n}\right) \ge \frac{1}{5}
\]
for $n>N$.
\end{lem}

This lemma implies the following: for any pair of equivalence classes $\mathcal{E}$, $\dot{\mathcal{E}}$ that have the first $2^{m}$ pivotal times $\{i_{1}, \ldots, i_{2^{m}}\}$ and $\{\ddot{\imath}_{1}, \ldots, \ddot{\imath}_{2^{m}}\}$, respectively, the conditional probability of\begin{equation}\label{eqn:converseIng2}
 \sum_{j=1}^{2^{m}} d(o, a \xi_{\stopping(i_{j})} b o) \ge \sum_{j=1}^{2^{m}} d(o, a \dot{\xi}_{\dot{\stopping}(\ddot{\imath}_{j})} b o)  + K 2^{m/2}
\end{equation}
on $\mathcal{E}\times  \dot{\mathcal{E}}$ is at least 1/5, given that $m$ is large enough. ($\ast\ast$)

Let us now combine the ingredients. Conditioned on each $\mathcal{E}\times  \dot{\mathcal{E}} \in \mathscr{E}$, we have Inequalities \ref{eqn:converseIng1} (for $\Omega_{m}$ and $\dot{\Omega}_{m}$), \ref{eqn:symmetryImport} and \ref{eqn:converseIng2} with probability at least $1/5 - 1/500$ by $(\ast)$, $(\ast\ast)$. In this case, we replace $2 \sum_{l=0}^{m} \E[Y_{l}|\mathcal{E}]$ with $2\sum_{l=0}^{m} Y_{l}$ and $2 \sum_{l=0}^{m} \E[\dot{Y}_{l}|\dot{\mathcal{E}}]$ with $2 \sum_{l=0}^{m} \dot{Y}_{l}$ in Inequality \ref{eqn:symmetryImport} to obtain \begin{equation}\label{eqn:converseIng3}\begin{aligned}
&\sum_{i=0}^{2^{m}-1} d(x_{2i}, x_{2i+1}) + d(x_{2^{m+1}}, \w_{n_m} o) -2 \sum_{l=0}^{m} Y_{l}\\
&\ge \sum_{i=0}^{2^{m}-1} d(\dot{x}_{2i}, \dot{x}_{2i+1}) + d(\dot{x}_{2^{m+1}}, \dot{\w}_{n_m} o) -2\sum_{l=0}^{m} \dot{Y}_{l} - 1600F_{0} 2^{m/2}.
\end{aligned}
\end{equation}
We now add up Inequalities \ref{eqn:converseIng3} and \ref{eqn:converseIng2}. Using identities \[
\sum_{i=1}^{2^{m}} [d(x_{2i-2}, x_{2i-1}) + d(x_{2i-1}, x_{2i})] -2 \sum_{l=0}^{m} Y_{l} = d(x_{0}, x_{2^{m+1}}),
\]
$d(o, a\xi_{\stopping(i_{j})} b o) = d(x_{2j-1}, x_{2j})$ and $(o, \w_{n_{m}}o)_{x_{2^{m+1}}} \le F_{0}$, we deduce
\begin{equation}\label{eqn:converseSemi}
d(o, \w_{n_{m}} o) - d(o, \dot{\w}_{n_{m}} o) \ge (K-1600F_{0})2^{m/2} - 4F_{0}.
\end{equation}
In conclusion, Inequality \ref{eqn:converseSemi} holds with conditional probability at least $0.198$ on $\mathcal{E}\times  \dot{\mathcal{E}} \in \mathscr{E}$. Summing them up, we have probability at least \[
0.198 \cdot \Prob \left(\bigcup \mathscr{E}\right) \ge 0.198 \cdot \frac{1}{2}\left(1- \Prob[\omega \notin \Omega_{m}] - \Prob[\dot{\omega} \notin \dot{\Omega}_{m}]\right) \ge 0.099 \cdot (1 - 2K_{1}e^{- n_{m}/K_{1}}).
\]
Since $K$ is arbitrary, we conclude that \[
\Prob\left[ \frac{1}{\sqrt{n_{m}}} [d(o, \w_{n_{m}} o) - d(o, \dot{\w}_{n_{m}}o)] \ge K \right] \ge 0.09
\]
eventually holds for any $K>0$. This cannot happen if $\frac{1}{\sqrt{n}} [d(o, \w_{n} o)-c_{n}]$, and hence $\frac{1}{\sqrt{n}}[ d(o, \w_{n} o)-d(o, \dot{\w}_{n} o)]$, converges in law.

To deduce the same conclusion for translation lengths, it suffices to prove\[
\Prob \left[d(o, \w_{n} o) - \tau(\w_{n}) \ge \sqrt{n} \right] \le 0.021
\]
for all sufficiently large $n$. We first take $n_{1}$ such that $\Prob(|Q_{n_{1}}| \le 2) \le 10^{-5}$, and define an RV $Y(\w) := d(o, y_{i(2), 0}^{-})$ where $i(2)$ is the second eventual pivotal time for $\w$. $Y(\w)$ is a.e. finite so there exists $n_{2}$ such that $\Prob(Y(\w) \ge 0.1\sqrt{n_{2}})  \le 10^{-5}$. Then for $n > \max(n_{1}, n_{2})$, outside an event of probability at most $2 \cdot 10^{-5}$, we have $d(o, y_{i(1), 0}^{-}), d(o, y_{i(2), 0}^{-}) \le 0.1\sqrt{n}$; let $E$ be the collection of such path $\w$. 

For $\w \in E$, $d(o, \w_{n} o) - \tau(\w_{n}) \le \sqrt{n}$ automatically holds when $d(o, \w_{n} o) \le \sqrt{n}$. If not, we condition on $\mathcal{E}^{2}(\w)$, the collection of paths $\bar{\w}$ pivoted from $\w$ at the first two eventual pivotal times. If $n$ is sufficiently large, \[
d(o, \bar{\w}_{n} o) - d(o, \bar{y}_{i(2), 0}^{-}) - d(o, \bar{y}_{i(1), 2}^{-}) \ge 0.8 \sqrt{n} - 2\mathscr{M} - 12F_{0}\ge 2\mathscr{M}+3D_{0}
\] holds and the argument in the proof of Theorem \ref{thm:logCorr} implies that \[
d(o, \bar{\w}_{n} o) - \tau(\bar{\w}_{n}) \le 2d(o, y_{i(1), 0}^{-}) + 2F_{0}< \sqrt{n}
\] with probability at least $1- [305^{2} - 303^{2}]/305^{2} \ge 1-0.02$. Hence, summing up the conditional probability across $E$, $\frac{1}{\sqrt{n}}[d(o, \w_{n} o) - \tau(\w_{n})] \ge1$ happens in $E$ with probability at most 0.02. Outside $E$ we have $2 \cdot 10^{-5}$ more chance, hence the conclusion.
\end{proof}

\subsection{Central limit theorems}

The purpose of this subsection is to prove a CLT for $d(o, \w_{n} o)$. After obtaining a uniform control of $(\w_{n}^{-1} o, \w_{n} o)_{o}$, the convergence to a Gaussian law is due to \cite{mathieu2020deviation}. However, we should first establish a lower bound on the variance to guarantee the convergence to a non-degenerate Gaussian law, which we present below.

\begin{proof}
Since $\mu$ is nonarithmetic, there exist $a_{1}, \ldots, a_{l}$, $b_{1}, \ldots, b_{l}\in \supp \mu$ such that $g = a_{1}\cdots a_{l}$, $g' = b_{1}\cdots b_{l}$ satisfy $d(o, go) - d(o, g'o) \ge 104F_{0}+4\mathscr{M}$. Since $S_{0}$ contains more than 4 elements, there exist $a, b \in S_{0}$ such that $g, g' \in A_{a, b}$. We then take a subset $S$ of $S_{0} \setminus \{a, b\}$ with cardinality 305, and let $\mu_{g, g'}$ be the uniform measure on $\{(a_{1}, \ldots, a_{l}), (b_{1}, \ldots, b_{l})\}$. Then \[
\mu^{6N+l}=\alpha(\mu_{S}^{2} \times 1_{\{a\}} \times \mu_{g, g'} \times 1_{\{b\}} \times \mu_{S}^{2}) + (1-\alpha) \nu
\] holds for some $\nu$ and $0 < \alpha < 1$. This enables us to construct RVs and pivotal times/loci as in the previous subsection. This time, $\xi_{i}$ are defined to be product of $(3N+1)$-th, $\ldots$, $(3N+l)$-th coordinates of $\eta_{i}$.

\begin{claim}\label{claim:suffVar}
We have $Var\left[d(o, \w_{n}o) \, \Big| \,|P_{n}(\w)| \ge 2^{m}\right] \ge 900 F_{0}^{2} 2^{m} $.
\end{claim}

\begin{proof}[Proof of Claim \ref{claim:suffVar}]
On $\Omega_{m} := \{\w : |P_{n}(\w)| \ge 2^{m}\}$, we declare the equivalence relation as in the previous subsection. In other words, we declare the equivalence class of $\w_{0} \in \Omega_{m}$ by \[\mathcal{E}(\w_{0}) = \left\{\w : \begin{array}{c} (\rho_{i}, \nu_{i}, \alpha_{i}, \beta_{i}, \gamma')(\w) = (\rho_{i}, \nu_{i}, \alpha_{i}, \beta_{i}, \gamma')(\w_{0}),\\ \xi_{i}(\w) = \xi_{i}(\w_{0})\,\,\textrm{unless}\,\, i = \stopping(i_{1}), \ldots, \stopping(i_{2^{m}})\end{array}\right\}.
\]

Let us fix an equivalence class $\mathcal{E}$ with the first $2^{m}$ pivotal times $i_{1}, \ldots, i_{2^{m}}$. 
Recall that we have labelled the sample loci at pivotal times by $x_{i}$'s. More precisely, we have set $x_{0} = o$, $x_{2^{m+1}+1} = \w_{n} o$ and $(x_{2l-1}, x_{2l}) = (y_{i(l), 0}^{-}, y_{i(l),0}^{+})$ for $l = 1, \ldots, 2^{m}$. We now define $x_{l}' = x_{l}$ for $0 \le l \le 2^{m+1} - 1$ and $x_{2^{m+1}}' = \w_{n} o$.
We will inductively prove that \[
Var\left[d(x_{2^{k} (l-1)}', x'_{2^{k} l}) \, \Big| \, \mathcal{E}\right] \ge F_{0}^{2} \left[900 \cdot 2^{k} + 240\cdot 2^{k/2}\right].
\]
for $k = 1, \ldots, m+1$ and $l = 1, \ldots, 2^{m-k+1}$. In particular, for $k = m+1$ and $l = 1$, this reads \[
Var \left[ d(x_{0}', x_{2^{m+1}}') \, \Big| \, \mathcal{E}\right] = Var \left[ d(o, \w_{n} o) \, \Big| \, \mathcal{E}\right] \ge 900 F_{0}^{2} 2^{m+1}.
\]
By summing up these conditional variances for various equivalence classes, we conclude the claim.

Let us consider the case $k=1$. For each $1 \le l < 2^{m}$, $w^{(l)} := (w_{i_{l}, 0}^{-})^{-1} w_{i_{l-1}, 0}^{+}$ is constant across $\mathcal{E}$, $(x_{2l-2}, x_{2l})_{x_{2l-1}} = (wo, a\xi_{\stopping(i_{l})} b o)_{o} \le F_{0}$ and $\xi_{\stopping(i_{l})} = g$ or $g'$ with equal probabilities. This implies that
\[\begin{aligned}
Var[d(x_{2l-2}', x'_{2l}) | \mathcal{E}] &= \left[ \frac{1}{2} |d(w^{(l)}o, agbo) -d(w^{(l)}o, ag'bo)| \right]^{2} \\
& =\frac{1}{4} \left| \begin{array}{c} \left[d(w^{(l)} o, o) + d(o, agbo) -2(w^{(l)} o, agbo)_{o}\right] \\ - \left[d(w^{(l)} o, o) + d(o, ag'bo) -2(w^{(l)}o, ag'bo)_{o}\right] \end{array}\right|^{2}\\
&\ge \frac{1}{4} \left(\left| d(o, agbo) -d(o, ag'bo)\right| - 2F_{0}\right)^{2} \\
&\ge \frac{1}{4}\left(|d(o, go) - d(o, g'o)| - 4\mathscr{M} - 2F_{0}\right)^{2} \\
&\ge2500 F_{0}^{2} \ge F_{0}^{2} \cdot \left[1800+ 240\sqrt{2}\right].
\end{aligned}
\]
For $l= 2^{m}$, $w^{(2^{m})} := (w_{i_{2^{m}}, 0}^{-})^{-1} w_{i_{2^{m} - 1}, 0}^{+}$ and $w' := (w_{i_{2^{m}}, 0}^{+})^{-1} \w_{n}$ are constant across $\mathcal{E}$ and $(x_{2^{m+1} - 2}, x_{2^{m+1}})_{x_{2^{m+1} - 1}}, (x_{2^{m+1} - 2}, x_{2^{m+}1} + 1)_{x_{2^{m+1}}} \le F_{0}$. Also, $\xi_{\stopping(i_{2^{m}})} = g$ or $g'$ with equal probability. Using these, we similarly deduce \[\begin{aligned}
Var[d(x_{2l-2}', x'_{2l}) | \mathcal{E}] &= \left[\frac{1}{2} |d(w^{(2^{m})} o, agb \cdot w' o) - d(w^{(2^{m})} o, ag'b \cdot w' o)| \right]^{2} \\
&= \frac{1}{4} \left| \begin{array}{c} [d(w^{(2^{m})} o, o) + d(o, agb o) + d(agb o, agb w' o)] \\
 -[ 2(w^{(2^{m})} o, agb o)_{o} + 2(w^{(2^{m})} o, agb w' o)_{agb o} ] \\
-  [d(w^{(2^{m})} o, o) + d(o, ag'b o) + d(ag'b o, ag'b w' o) ] \\
+ [
 2(w^{(2^{m})} o, ag'b o)_{o} +2(w^{(2^{m})} o, ag'b w' o)_{agb o} ] \end{array}\right|^{2} \\
&\ge \frac{1}{4} \left(|d(o, agbo) - d(o, ag'b o) - 4F_{0}\right)^{2} \\
&\ge \frac{1}{4}\left(|d(o, go) - d(o, g'o)| - 4\mathscr{M} - 4F_{0}\right)^{2}  \ge 2500 F_{0}^{2}.
\end{aligned}
\]

Suppose now that $Y_{1} = d(x_{2^{k}(2l-2)}', x'_{2^{k} (2l-1)})$ and $Y_{2} = d(x_{2^{k} (2l-1)}', x'_{2^{k} \cdot 2l})$ satisfy the estimation for some $1\le k\le m$ and $1 \le l \le 2^{m-k}$. We now estimate the variance of $Y =d(x_{2^{k+1}(l-1)}', x'_{2^{k+1} l})= Y_{1} + Y_{2} - b$, where $b = 2(x_{2^{k}(l-2)}', x_{2^{k} l}')_{x_{2^{k}(l-1)}'}$. Since $Y_{1}, Y_{2}$ are independent and $0 \le b \le 2F_{0}$, \[\begin{aligned}
Var(Y) &\ge Var(Y_{1}) + Var(Y_{2}) - 2F_{0} \cdot \sqrt{Var(Y_{1})} - 2F_{0} \cdot \sqrt{Var(Y_{2})} \\
&= Var(Y_{1}) \left[ 1 - \frac{2F_{0}}{\sqrt{Var(Y_{1})}}\right] + Var(Y_{2}) \left[ 1 - \frac{2F_{0}}{\sqrt{Var(Y_{2})}}\right] \\
& \ge2 \cdot F_{0}^{2} \left[ 900 \cdot 2^{k}+ 240\cdot 2^{k/2}\right] \left[ 1 - \frac{2F_{0}}{F_{0}\cdot 30\cdot 2^{k/2}}\right]\\
& \ge 2 \cdot F_{0}^{2} \left[ 900\cdot2^{k} + 180 \cdot 2^{k/2} - 16\right] \\
& \ge F_{0}^{2} \left[ 900 \cdot 2^{k+1}+ 240 \cdot 2^{(k+1)/2}  + (360 - 240\sqrt{2}) 2^{k/2} - 16\right]
\end{aligned}
\]
holds. Since $360 - 240\sqrt{2}\ge 16$, we have the desired conclusion for $k+1$.
\end{proof}

In particular, Claim \ref{claim:suffVar} and Proposition \ref{prop:expDecay} together imply that\begin{equation}\label{eqn:varianceAtLeast}
Var[d(o, \w_{n} o)] \ge 100F_{0}^{2} n/K_{1}
\end{equation} for sufficiently large $n$.

In the remaining part of the proof, we employ the theory of \cite[Section 4]{mathieu2020deviation}. We first fix $M > 0$ and consider the random variables \[
Y_{k, i} = d(\w_{2^{k}M (i-1)}o, \w_{2^{k}M i} o), \quad b_{k, i} = (\w_{2^{k}M (i-1)} o, \w_{2^{k}M (i+1)} o)_{\w_{2^{k}M i }o}
\] (see Figure \ref{fig:dyadic}) and their balanced versions \[
\bar{Y}_{k, i} = Y_{k, i} - \E[Y_{k, i}], \quad \bar{b}_{k, i} =b_{k, i} - \E[ b_{k,i}].
\]
Observe the following: \begin{enumerate}
\item each of $\{Y_{k, i}\}_{i \in \Z}$, $\{b_{k, i}\}_{i \in 2\Z + 1}$, $\{b_{k, i}\}_{i \in 2\Z}$ is a family of i.i.d;
\item there exists $K>0$ such that $\E[b_{k, i}^{4}]< K^{2}$ (Proposition \ref{prop:dominantGrom2});
\item $\E[\bar{b}_{k, i}^{2}] \le \E(|b_{k, i}| + \E|b_{k, i}|)^{2} \le 4 \E[b_{k, i}^{2}] \le 4K$;
\item $Y_{k+1, i} = Y_{k, 2i-1} + Y_{k, 2i} - 2b_{k, 2i-1}$ for each $k$, $i$.
\end{enumerate}

We first show that $\frac{1}{\sqrt{n}}[\E[d(o, \w_{n} o)] - n \lambda] \rightarrow 0$ as $n \rightarrow \infty$. Observe that \[
\frac{1}{2^{k}M} \E[Y_{k, 1}] = \frac{1}{2^{k}M} \sum_{i=1}^{2^{k}} \E[Y_{0, i}] - \frac{2}{2^{k}M}  \sum_{t = 0}^{k-1} \left[ \sum_{i = 1}^{2^{k-t-1}} \E[b_{t, 2i - 1}]\right].
\]
The LHS converges to the escape rate $\lambda$ as $k \rightarrow \infty$, and the first term of the RHS is always $\frac{1}{M} \E[d(o, \w_{M} o)]$. Finally, since $\E[b_{t, 2i-1}] < \sqrt{K}$ for any $t$ and $i$, the second term of the RHS is bounded by $2\sqrt{K}/M$. Hence we deduce $|\sqrt{n} \lambda - \frac{1}{\sqrt{n}} \E[d(o, \w_{n} o)]| \le 2\sqrt{K}/\sqrt{n}$ as desired.

From now on we take $M =2^{m}$ for positive integers $m$. Observe that \begin{equation}\label{eqn:dyadicSum}
\frac{1}{\sqrt{2^{k+m}}}Y_{k, 1} = \frac{1}{\sqrt{2^{k+m}} }\sum_{i=1}^{2^{k}} Y_{0, i} -\frac{2}{\sqrt{2^{k+m}}}  \sum_{t = 0}^{k-1} \left[ \sum_{i = 1}^{2^{k-t-1}} b_{t, 2i - 1}\right].
\end{equation}
By subtracting the expectations, we also have
\begin{equation}\label{eqn:dyadicSumBal}
\frac{1}{\sqrt{2^{k+m}}}\bar{Y}_{k, 1} = \frac{1}{\sqrt{2^{k+m}} }\sum_{i=1}^{2^{k}} \bar{Y}_{0, i} -\frac{2}{\sqrt{2^{k+m}}}  \sum_{t = 0}^{k-1} \left[ \sum_{i = 1}^{2^{k-t-1}} \bar{b}_{t, 2i - 1}\right].
\end{equation}

Let us investigate the error term $\sum_{t} \sum_{i} \bar{b}_{t, 2i-1}$. For each $t$, $\sum_{i}\bar{b}_{t, 2i-1}/\sqrt{2^{k+m}}$ is the sum of $2^{k-t-1}$ independent variables, each having variance bounded by $K/2^{k+m}$. Thus, this sum has variance less than $K/2^{m+t+1}$ and \[
\Prob\left( E_{t} := \left\{\left|\frac{1}{\sqrt{2^{k+m}}}\sum_{i=1}^{2^{k-t-1}} \bar{b}_{t, 2i-1} \right| \ge 2^{-m/3}2^{-t/4} \right\} \right) \le \frac{K}{2^{m/3+t/2 + 1}}
\]
holds by Chebyshev. Thus, $\frac{1}{\sqrt{2^{k+m}}}\sum_{t} \sum_{i} \bar{b}_{t, 2i-1}$ is bounded by $7 \cdot 2^{-m/3}$ outside $\cup_{t} E_{t}$, where $\Prob(\cup_{t} E_{t}) \le 8K \cdot 2^{-m/3}$.

Meanwhile, by the classical CLT, $\frac{1}{\sqrt{2^{k+m}}} \sum_{i=1}^{2^{k}} \bar{Y}_{0, i}$ converges to a Gaussian law $\mathscr{N} (0, \sigma_{m})$ as $k$ increases. Here, Inequality \ref{eqn:varianceAtLeast} guarantees that $\sigma_{m}:= \frac{1}{\sqrt{2^{m}}} \sqrt{Var(d(o, \w_{2^{m}} o))} \ge 10F_{0} /\sqrt{K_{1}}$ when $m$ is large enough. 

In summary, the random variables $\frac{1}{\sqrt{2^{k}}} [d(o, \w_{2^{k}} o)- \E[d(o, \w_{2^{k}} o)]]$ are eventually $(16K+15) \cdot 2^{-m/3}$-close to $\mathscr{N}(0, \sigma_{m})$ in the L{\'e}vy metric. This implies that $\mathscr{N}(0, \sigma_{m})$ are Cauchy, and since $\sigma_{m}$ is bounded below, they converge to a nondegenerate Gaussian law $\mathscr{N}(0, \sigma)$ (and $\lim_{m} \sigma_{m}= \sigma)$.

To deal with distributions at general steps, we consider auxiliary variables \[\begin{aligned}
Y_{k; n} &= d(\w_{2^{k+m} \lfloor n/2^{k+m} \rfloor}o, \w_{n} o), \\ b_{k;n} &=\left\{\begin{array}{cc} (\w_{2^{k+m+1} \lfloor n/2^{m+k+1} \rfloor} o, \w_{n} o)_{\w_{2^{k+m}(2 \lfloor n/2^{m+k+1} \rfloor + 1)}o} &\textrm{if}\,\, 2^{k+m}(2 \lfloor n/2^{m+k+1} \rfloor + 1) < n \\ 0 & \textrm{otherwise}. \end{array}\right.
\end{aligned}
\]

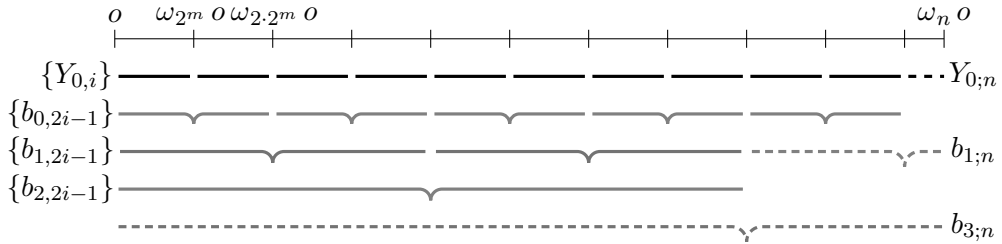
\begin{figure}[H]
\begin{tikzpicture}
\def\c{0.5}
\def\x{1.05}

\begin{scope}[shift={(0, 0.5)}]
\draw (0, 0) -- (10.5*\x, 0);
\foreach \i in {0, ..., 10}{
\draw (\i*\x, 0.12) -- (\i*\x, -0.12);
}
\draw (\x*10.5, 0.12) -- (\x*10.5, -0.12);

\draw (0, 0.3) node {$o$};
\draw (\x-0.04, 0.3) node {$\w_{2^{m}}o$};
\draw (2*\x+0.02, 0.3) node {$\w_{2 \cdot 2^{m}}o$};
\draw (10.5*\x, 0.3) node {$\w_{n}o$};

\end{scope}

\draw(-0.51, 0) node {$\{Y_{0, i}\}$};
\draw(-0.72, -\c) node {$\{b_{0, 2i-1}\}$};
\draw (-0.72, -2*\c) node {$\{b_{1, 2i-1}\}$};
\draw (-0.72, -3*\c) node {$\{b_{2, 2i-1}\}$};

\draw (10.87*\x, 0) node {$Y_{0; n}$};
\draw (10.87*\x, -2*\c) node {$b_{1; n}$};
\draw (10.87*\x, -4*\c) node {$b_{3; n}$};

\foreach \i in {0, ..., 9}{
\draw[very thick] (\x*\i+0.05, 0) -- (\x*\i + \x - 0.05, 0);
}
\draw[very thick, dashed] (10*\x + 0.05, 0) -- (10.5*\x, 0);

\foreach \i in {0, ..., 4}{
\draw[very thick, black!50] (2*\x*\i + 0.05, -\c) -- (2*\x*\i+\x - 0.13, -\c) arc (90:0:0.13) arc (180:90:0.13) -- (2*\x*\i+ 2*\x-0.05, -\c);
}
\foreach \i in {0,1}{
\draw[very thick, black!55] (4*\x*\i+0.07, -2*\c) -- (4*\x*\i+2*\x - 0.15, -2*\c) arc (90:0:0.15) arc (180:90:0.15) -- (4*\x*\i+4*\x-0.07, -2*\c);
}
\draw[very thick, black!50, densely dashed] (8*\x + 0.07, -2*\c) -- (10*\x - 0.2, -2*\c) arc (90:0:0.2) arc (180:90:0.2) -- (10.5*\x, -2*\c);

\draw[very thick, black!50](0.05, -3*\c) -- (4*\x - 0.15, -3*\c) arc (90:0:0.15) arc (180:90:0.15) -- (8*\x-0.05, -3*\c);
\draw[very thick, black!55, densely dashed](0.05, -4*\c) -- (8*\x-0.2, -4*\c) arc (90:0:0.2) arc (180:90:0.2) -- (10.5*\x, -4*\c);

\end{tikzpicture}
\caption{$\{Y_{k, i}\}$, $\{Y_{k;n}\}$, $\{b_{k, i}\}$ and $\{b_{k;n}\}$ for $10 \cdot 2^{m} \le n \le 11 \cdot 2^{m}$. Here $b_{0;n} = b_{2;n} = 0$ since $2^{m} (2 \lfloor n/2^{m+1} \rfloor + 1) = 11 \cdot  2^{m} \ge n$ and $2^{m+2} (2 \lfloor n/2^{m+3} \rfloor + 1) = 12 \cdot 2^{m} \ge n$.
}
\label{fig:dyadic}
\end{figure}

Here, $\E[b_{k; n}^{2}] \le 4K$ still holds for any $k$ and $n$ (Proposition \ref{prop:dominantGromEach} for $q=0$ and $p = 2$). We now realize that \[\begin{aligned}
&\frac{1}{\sqrt{n}} [d(o, \w_{n} o)- \E[d(o, \w_{n} o)]]\\
&= \frac{1}{\sqrt{n}}\sum_{i=1}^{\lfloor n/2^{m} \rfloor} \bar{Y}_{0, i} + \frac{1}{\sqrt{n}}\bar{Y}_{0; n}-\frac{2}{\sqrt{n}}  \sum_{2^{m+t} \le n} \left[ \bar{b}_{t; n} +\sum_{i = 1}^{\lfloor n/2^{m+t+1}\rfloor} \bar{b}_{t, 2i - 1} \right].
\end{aligned}
\]
As $n \rightarrow \infty$, the first term converges to $\mathscr{N}(0, \sigma_{m})$ in law. The second term converges to 0 in probability, because \[\begin{aligned}
\sum_{n=1}^{\infty} \Prob( \bar{Y}_{0; n} / \sqrt{n} \ge \epsilon) &= \sum_{n=1}^{\infty} \Prob( d(o, \w_{n - 2^{k+m} \lfloor n/2^{k+m} \rfloor} o) \ge \epsilon \sqrt{n}) \\
&\le \sum_{n=1}^{\infty} \Prob(Z \ge \epsilon \sqrt{n}) \le \E[(Z/\epsilon)^{2}] <+\infty,
\end{aligned}
\]
where $Z$ has the distribution of $\sum_{i=1}^{n} d(o, g_{i} o)$.

 Moreover, for $2^{m+t} \le n$ we have \[
Var\left(\frac{1}{\sqrt{n}} \left[b_{t; n} +\sum_{i = 1}^{\lfloor n/2^{m+t+1}\rfloor} b_{t, 2i - 1}\right]\right) \le \frac{4K}{n} \cdot\left[ \left\lfloor \frac{n}{2^{m+t+1}} \right\rfloor + 1 \right] \le \frac{4K}{2^{m+t}}.
\] This implies that the final term is bounded by $7 \cdot 2^{-m/3}$ outside an event with probability at most $16K \cdot 2^{-m/3}$. In conclusion, $\frac{1}{\sqrt{n}} [d(o, \w_{n} o)- \E[d(o, \w_{n} o)]]$ is eventually $(32K+15)2^{-m/3}$-close to $\mathscr{N}(0, \sigma_{m})$ for each $m$. Since $\mathscr{N}(0, \sigma_{m}) \rightarrow \mathscr{N}(0, \sigma)$, we conclude $\frac{1}{\sqrt{n}} [d(o, \w_{n} o) -\E[d(o, \w_{n} o)]]\rightarrow  \mathscr{N}(0, \sigma)$.
\end{proof}

\section{Law of the Iterated Logarithm}
\label{section:lil}
Throughout this section we set \[
LLn := \left\{ \begin{array}{cc} \log \log n & n \ge 3 \\ 1 & n <2,\end{array}\right. \quad \alpha(n) := (2n LLn)^{1/2}, \quad \beta(n) := (n/LLn)^{1/2}.
\]

In this section, we adapt de Acosta's argument for the classical LIL in \cite{deAcosta1983lil} to prove our LIL. Let us briefly summarize de Acosta's strategy before entering the proof. Let $\{X_{i}\}$ be a sequence of balanced i.i.d. with $Var(X_{i}) < K$. In order to investigate the deviation of $\sum_{i=1}^{n} X_{i}$ in the order of $\alpha(n)$, de Acosta first truncated $X_{n}$ to obtain $Y_{n} := X_{n} 1_{\{|X_{n}| \le \beta(n)\}}$, $Z_{n} := X_{n} 1_{\{|X_{n}| >\beta(n)\}}$(assume $\E[Y_{n}] = 0$ at the moment for convenience).

The truncation threshold $\beta(n)$ is so designed that the a.e. convergence of $\sum_{i=1}^{n} |Z_{i}|/\alpha(i)$ follows from finite variances of $X_{i}$. Kronecker's lemma then implies that the term $(\sum_{i=1}^{n} Z_{i})/\alpha(n)$ does not contribute significantly. For $Y_{n}$, we make use of the independence of $Y_{n}$, truncation bounds of $Y_{n}$ and Chebyshev's inequality to deduce \[
\Prob\left\{ \sum_{i=1}^{n} Y_{i}/\alpha(n) > t\right\} \le \textrm{exp} \left[ -\lambda t + \frac{\lambda^{2} K}{4 LLn} \textrm{exp} \left(\frac{\lambda}{\sqrt{2} LLn}\right) \right]
\]
for any $t, \lambda >0$. The final trick is to couple the sequence of events $E_{n} := \{\sum_{i=1}^{n} X_{i}/\alpha(n) > t \}$ with a geometric subsequence $E_{\lfloor p^{k} \rfloor}$, in the sense that \begin{equation}\label{eqn:deAcosta}
\Prob\left(\cup_{n \ge p^{k_{0}}} E_{n}\right) \le C\sum_{k \ge k_{0}} \Prob\left(E_{\lfloor p^{k} \rfloor}\right).
\end{equation}
Choosing suitable $t$ and $\lambda$, one can make this series convergent and Borel-Cantelli leads to the a.e. upper bound of $\limsup (\sum_{i=1}^{n} X_{i})/\alpha(n)$. Let us now make this discussion precise.

\begin{proof}[Proof of the LIL]

Given an integer $m \ge 16$, we set the following RVs as in Section \ref{section:central}:
\[\begin{aligned}
Y_{k, i} &= d(\w_{2^{k} \cdot 2^{m}(i-1)} o, \w_{2^{k} \cdot 2^{m} i} o), \\
b_{k, i} &= (\w_{2^{k} \cdot 2^{m} (i-1)} o, \w_{2^{k} \cdot 2^{m} (i+1)} o)_{\w_{2^{k} \cdot 2^{m} i}o},\\
Y_{k; n} &= d(\w_{2^{k+m} \lfloor n/2^{k+m} \rfloor}o, \w_{n} o), \\ b_{k;n} &=\left\{\begin{array}{cc} (\w_{2^{k+m+1} \lfloor n/2^{m+k+1} \rfloor} o, \w_{n} o)_{\w_{2^{k+m}(2 \lfloor n/2^{m+k+1} \rfloor + 1)}o} &\textrm{if}\,\, 2^{k+m}(2 \lfloor n/2^{m+k+1} \rfloor + 1) < n \\ 0 & \textrm{otherwise}. \end{array}\right.
\end{aligned}
\]
Note that
 \begin{equation}\label{eqn:lil}\begin{aligned}
&\frac{1}{\alpha(n)} [d(o, \w_{n} o)- \E[d(o, \w_{n} o)]]\\
&= \frac{1}{\alpha(n)}\sum_{i=1}^{\lfloor n/2^{m} \rfloor} \bar{Y}_{0, i} + \frac{1}{\alpha(n)}\bar{Y}_{0; n}-\frac{2}{\alpha(n)}  \sum_{2^{m+t} \le n} \left[ \bar{b}_{t; n} +\sum_{i = 1}^{\lfloor n/2^{m+t+1}\rfloor} \bar{b}_{t, 2i - 1} \right].
\end{aligned}
\end{equation}
The second term in the RHS of Equation \ref{eqn:lil} converges to 0, as we have observed that $\sum_{n} \Prob( \bar{Y}_{0; n} \ge \epsilon \sqrt{n})$ is summable for each $\epsilon > 0$. The first term is the sum of i.i.d.s divided by $\alpha(n)$ and the usual LIL applies. It is the final term in Equation \ref{eqn:lil} that requires de Acosta's argument. The additional obstacle here is that we deal with the infinite sequence $\{ \sum_{i} \bar{b}_{t, 2i-1}\}_{t}$ of sums of i.i.d.; we should not only establish a bound on RHS of Inequality \ref{eqn:deAcosta} for each family $\{\bar{b}_{t, 2i-1}\}_{i}$, but also that the bound is summable for $t$.

\begin{claim}\label{claim:lil1}
For any $K'>0$, there exists $m>16$ such that \[
\Prob \left\{\limsup_{n} \frac{1}{\alpha(n)}\left|\sum_{2^{m+t} \le n} \sum_{i=1}^{\lfloor n/2^{m+t+1}\rfloor} \bar{b}_{t, 2i-1} \right| > K' \right\} \le K'.
\]
\end{claim}

\begin{proof}
Let us consider \[\begin{aligned}
E_{t, i} &:= \left\{\w : |\bar{b}_{t, 2i-1}| > \beta(2^{t+m+1}i)/2^{(t+m)/4} \right\}, \\
B_{t, 2i-1} &:= \bar{b}_{t, 2i-1} 1_{E_{t, i}}, \,\,\,B'_{t, 2i-1} := \bar{b}_{t, 2i-1} 1_{E_{t, i}^{c}}, \,\,\, \bar{B}'_{t, 2i-1} := B'_{t, 2i-1} - \E B'_{t, 2i-1}.
\end{aligned}
\]Note that \[\begin{aligned}
\left|\E B'_{t, 2i-1}\right| = \left| \E B_{t, 2i-1}\right| &\le \E|B_{t, 2i-1}| \le \E|\bar{b}_{t, 2i-1}| \\
&= \E\left| b_{t, 2i-1} - (\E b_{t, 2i-1})\right| \le 2\E|b_{t, 2i-1}| \le 2\sqrt{K},\\
|\bar{B}'_{t, 2i-1}| &\le |B'_{t, 2i-1}| + |\E B'_{t, 2i-1}| \le 2 \cdot \beta(2^{t+m+1}i)/2^{(t+m)/4},\\
\E (\bar{B}'_{t, 2i-1})^{2} &\le \E (|B_{t, 2i-1}'| + \E|B_{t, 2i-1}'|)^{2} \\
&\le 4\E|B_{t, 2i-1}'|^{2} \le 4\E\bar{b}_{t, 2i-1}^{2} \le 16K.
\end{aligned}
\]
Using the first equality and inequality, we have \begin{equation}\label{eqn:Kronecker}
\sum_{t=1}^{\infty} \sum_{i=1}^{\infty} |\E B_{t, 2i-1}'|/\alpha(2^{t+m+1}i)  = \sum_{t=1}^{\infty} \sum_{i=1}^{\infty} |\E B_{t, 2i-1}|/\alpha(2^{t+m+1}i) \le \sum_{t=1}^{\infty} \sum_{i=1}^{\infty} \E|B_{t, 2i-1}|/\alpha(2^{t+m+1}i) .
\end{equation}
Our first aim is to show that these summations are finite. We observe that \begin{equation}\label{eqn:lilinterm1}\begin{aligned}
&\sum_{i=1}^{\infty} \E|B_{t, 2i-1}|/\alpha(2^{t+m+1}i)\\
& \le \sum_{i=1}^{\infty} \sum_{k=0}^{\infty} \frac{1}{\alpha(2^{t+m+1}i)}\frac{\beta(2^{t+m+1}(i+k+1))}{2^{(t+m)/4}} \Prob\left[\frac{\beta(2^{t+m+1}(i+k))}{2^{(t+m)/4}} <|\bar{b}_{t, 2i-1}| \le \frac{\beta(2^{t+m+1}(i+k+1))}{2^{(t+m)/4}} \right] \\
&= \sum_{j=1}^{\infty}  \frac{\beta(2^{t+m+1} (j+1))}{2^{(t+m)/4}} \Prob\left[\frac{ \beta(2^{t+m+1} j)}{2^{(t+m)/4}} < |\bar{b}_{t, 1}| \le \frac{\beta(2^{t+m+1} (j+1))}{2^{(t+m)/4}}\right] \cdot \sum_{i=1}^{j} \frac{1}{\alpha(2^{t+m+1}i)}
\end{aligned}
\end{equation}
when $m\ge 8$. Here are used the facts that $\beta(x)$ is increasing for $x \ge 8$ and that $\{\bar{b}_{t, 2i-1}\}_{i}$ are i.i.d. Moreover, we have \[
\sum_{i=1}^{j} \frac{1}{\alpha(2^{t+m+1} i)} \le \frac{10}{2^{t+m+1}} \beta(2^{t+m+1}j),\quad \beta(2^{t+m+1} (j+1)) \le 1.1\beta(2^{t+m+1}j)\]
 for each $j$. Hence the last quantity in Inequality \ref{eqn:lilinterm1} is bounded by \[\begin{aligned}
&11\sum_{j=1}^{\infty}2^{-5(t+m)/4-1} \beta^{2}(2^{t+m+1}j)  \Prob\left[\frac{ \beta(2^{t+m+1} j)}{2^{(t+m)/4}} < |\bar{b}_{t, 1}| \le \frac{\beta(2^{t+m+1} (j+1))}{2^{(t+m)/4}}\right]  \\
&\le11 \cdot  2^{-3(t+m)/4} Var(\bar{b}_{t, 1}) \le 44K \cdot 2^{-3(t+m)/4},
\end{aligned}
\]
which is summable. Hence, the summations in Display \ref{eqn:Kronecker} are finite. In particular, we have \[\begin{aligned}
\sum_{n=1}^{\infty} \E\left[\frac{1}{\alpha(n)} \sum_{t \ge 1; 2^{t+m+1} | n} |B_{t, n/2^{t+m} - 1}| \right]&= \sum_{t=1}^{\infty} \sum_{i=1}^{\infty} \E |B_{t, 2i-1}|/\alpha(2^{t+m+1}i)< \infty.
\end{aligned}
\]
This means that $ \sum_{n=1}^{\infty} \frac{1}{\alpha(n)} \sum_{t \ge 1; 2^{t+m+1} | n} |B_{t, n/2^{t+m} - 1}|$ is finite almost surely. We now recall a classical result due to Kronecker: \begin{fact}\label{fact:Kronecker}
Let $(a_{n})_{n}$ be an increasing sequence of positive numbers and $(b_{n})_{n}$ be a real sequence. If $\sum_{n} b_{n}$ converges, then \[
\lim_{n} \frac{1}{a_{n}} \sum_{j=1}^{n} a_{j} b_{j} = 0.
\]
\end{fact}
Applying this fact, we deduce that \[
\frac{1}{\alpha(n)} \sum_{j=1}^{n} \sum_{t \ge 1; 2^{t + m+ 1}| j} |B_{t, j/2^{t+m} - 1}| = \frac{1}{\alpha(n)} \sum_{2^{m+t} \le n} \sum_{i=1}^{\lfloor n/2^{m+t+1} \rfloor} |B_{t, 2i-1}| 
\]
converges to zero almost surely. For a similar reason we also have \[
\lim_{n} \frac{1}{\alpha(n)} \left| \sum_{t =1}^{\infty} \sum_{i=1}^{\lfloor n /2^{m+t+1}\rfloor} \E B_{t, 2i-1}' \right| = 0.
\]

We now handle $\{\bar{B}'_{t, 2i-1}\}_{i}$. Since these are balanced i.i.d. with \[
\E(\bar{B}'_{t, 2i-1})^{2} \le 16K \quad \textrm{and}\quad|\bar{B}'_{t, 2i-1}| \le 2^{1-(t+m)/4} \cdot \beta(2^{t+m+1} i),
\]
we can apply Lemma 2.2 of \cite{deAcosta1983lil}. It begins with the observation \[
1 + x \le e^{x} \le 1 + x + \frac{x^{2}}{2} e^{|x|} \quad \forall x \in \R.
\]
Meanwhile, since $\beta(x)$ is an increasing function of $x$ for $x \ge 7$, we have \[
|\bar{B}'_{t, 2i-1}/ \beta(n)| \le 2^{1-(t+m)/4} \beta(2^{t+m+1} i) / \beta(n) \le 2^{1-(t+m)/4}
\]
when $m \ge 2$ and $i \in \{1, \ldots, \lfloor n/2^{m+t+1}\rfloor\}$. We then have \[
\begin{aligned}
\operatorname{exp} \left( \frac{\sqrt{2 /K}}{\beta(n)} \bar{B}_{t, 2i-1}'\right) &\le 1 +  \frac{\sqrt{2/K}}{\beta(n)} \bar{B}_{t, 2i-1}' + \left( \frac{\sqrt{2 /K}}{\beta(n)} \bar{B}_{t, 2i-1}'\right)^{2} \operatorname{exp} \left(\left|  \frac{\sqrt{2/K}}{\beta(n)} \bar{B}_{t, 2i-1}' \right| \right) \\&\le 1 +  \frac{\sqrt{2 K}}{\beta(n)} \bar{B}_{t, 2i-1}' +  \frac{2 }{K\beta(n)^{2}} |\bar{B}_{t, 2i-1}'|^{2} \operatorname{exp} \left( 2\sqrt{2/K}2^{-(t+m)/4}  \right).
\end{aligned}
\]
By taking expectations, we can remove the second term of RHS: \[\begin{aligned}
\E\left[ 
\operatorname{exp}\left(\frac{\sqrt{2/K}}{\beta(n)} \bar{B}'_{t, 2i-1}\right)\right ] &\le 1 + \frac{2}{K\beta(n)^{2}} \E |\bar{B}_{t, 2i-1}'|^{2} \operatorname{exp} \left(2 \sqrt{2/K} 2^{-(t+m)/4} \right) \\
&\le \operatorname{exp} \left( \frac{32 LLn }{n  }  \operatorname{exp} \left(2 \sqrt{2/K} 2^{-(t+m)/4}\right)\right).
\end{aligned}
\]
Since $\{\bar{B}_{t, 2i-1}'\}_{i}$'s are independent, we now have \[\begin{aligned}
\E\left[ \operatorname{exp} \left( \sqrt{\frac{2LLn}{Kn}} \sum_{i=1}^{\lfloor n/2^{m+t+1} \rfloor} \bar{B}_{t, 2i-1}' \right)\right] = 
\prod_{i=1}^{\lfloor n/2^{m+t+1} \rfloor} \E\left[ \operatorname{exp} \left( \frac{\sqrt{2/K}}{\beta(n)} \bar{B}_{t, 2i-1}' \right)\right] \\
\le \operatorname{exp} \left[ \frac{n}{2^{m+t+1}} \cdot \frac{32LLn }{n}  \operatorname{exp} \left(2 \sqrt{2/K} 2^{-(t+m )/4}\right)\right].
\end{aligned}
\]
Here, Markov's inequality tells us that the above expectation bounds\[\begin{aligned}
& \, \Prob \left( \sqrt{\frac{2LLn}{Kn}} \sum_{i=1}^{\lfloor n/2^{m+t+1}\rfloor} \bar{B}_{t, 2i-1}' \ge 2\cdot 2^{-(t+m)/8} LLn \right) \cdot \operatorname{exp}\left( 2\cdot 2^{-(t+m)/8} LLn\right) \\
&= \Prob \left( \frac{1}{\alpha(n)}  \sum_{i=1}^{\lfloor n/2^{m+t+1} \rfloor} \bar{B}_{t, 2i-1}'  \ge 2^{-(t+m)/8}\sqrt{K}  \right) \operatorname{exp}\left( 2\cdot 2^{-(t+m)/8}  LLn\right). \end{aligned}
\]
Hence we have \begin{equation}\label{eqn:lilMainEst}\begin{aligned}
 &\,\Prob \left( \frac{1}{\alpha(n)}  \sum_{i=1}^{\lfloor n/2^{m+t+1} \rfloor} \bar{B}_{t, 2i-1}' 
 \ge 2^{-(t+m)/8}\sqrt{K}  \right) \\& \le \operatorname{exp} \left( \left(-2 \cdot 2^{-(t+m)/8} + \frac{32 }{2^{m+t+1}} \operatorname{exp} \left( 2 \sqrt{2/K} 2^{-(t+m)/4} \right)\right) LLn \right) \\
 &\le \operatorname{exp}  \left( - 2^{-(t+m)/8} LLn \right),
 \end{aligned}
\end{equation}
where the last inequality is for large enough $m$ such that $2^{7m/8} \ge 16e^{2 \sqrt{2/K}}$. Meanwhile, Chebyshev's inequality implies that when $m > 16$, \begin{equation}\label{eqn:lilOttPrem}
\Prob \left[ \left| \sum_{i=n}^{2^{k}} \bar{B}_{t, 2i-1}' \right| \ge 2^{-(t+m)/8} \sqrt{K}\alpha(2^{m+t+1} \cdot 2^{k})\right] \le \frac{16 \cdot 2^{k} K}{2^{-(t+m)/4} K\alpha^{2}(2^{m+t+1} \cdot 2^{k})} \le 1/2
\end{equation}
for any $t, k\ge 1$ and $n \le 2^{k}$.

We now estimate the probability that $|\sum_{i=1}^{\lfloor n / 2^{m+t+1} \rfloor} \bar{B}'_{t, 2i-1}| > 4 \cdot 2^{-(t+m)/8} \sqrt{K} \alpha(n)$ occurs for at least one $n$. This is bounded by \[\begin{aligned}
&\sum_{k=0}^{\infty} \Prob\left[\max_{2^{k} \le n < 2^{k+1}} \left|\sum_{i=1}^{n} \bar{B}'_{t, 2i-1} \right|> 4 \cdot 2^{-(t+m)/8} \sqrt{K} \alpha(2^{m+t+1}\cdot 2^{k})\right]\end{aligned}.
\] By Inequality \ref{eqn:lilOttPrem} and Ottaviani's inequality, this is bounded by\[\begin{aligned}
2 \sum_{k=0}^{\infty} \Prob\left[\sum_{i=1}^{2^{k+1}} \left|\bar{B}'_{t, 2i-1}\right| > 3 \cdot 2^{-(t+m)/8} \sqrt{K} \alpha(2^{m+t+1}\cdot 2^{k})\right].
\end{aligned}
\]
Since $3\alpha(2^{m+t+1} \cdot 2^{k}) \ge \alpha(2^{m+t+1} \cdot 2^{k+1})$ for sufficiently large $m$ and all $k$, we can rely on Inequality \ref{eqn:lilMainEst} to bound this with \[
2 \sum_{k=0}^{\infty} ([k+m+t+2] \log 2)^{-2^{(t+m)/8}} \le 2 \sum_{k=m}^{\infty} (k \log 2)^{-4} \le \frac{1}{m^{3} (\log 2)^{4}}.
\]
Taking $m$ large enough, we have $m^{-3}(\log 2)^{-4} < K'$. Outside this event, we have \begin{equation}\label{eqn:lilIngredient}
\frac{1}{\alpha(n)} \sum_{t =1}^{\infty} \sum_{i=1}^{\lfloor n/2^{m+t+1} \rfloor} \bar{B}'_{t, 2i-1} \le 4 \sqrt{K} \sum_{t = 1}^{\infty} 2^{-(t+m)/8} \le 30 \sqrt{K} \cdot 2^{-m/8} <K'
\end{equation}
for all $n$, once again by taking $m$ large enough. Combining this with the fact that  $\frac{1}{\alpha(n)} \sum_{t=1}^{\infty} \sum_{i=1}^{\lfloor n/2^{m+t+1} \rfloor} B_{t, 2i-1}$ and $\frac{1}{\alpha(n)} \sum_{t=1}^{\infty} \sum_{i=1}^{ \lfloor n / 2^{m+ t +1} \rfloor} \E \bar{B}'_{t, 2i-1}$ converge to zero almost surely, we deduce the conclusion.
\end{proof}

We should also cope with the remaining terms $\bar{b}_{t; n}$'s: note that for each $t$, only one copy of $\bar{b}_{t;n}$ arises at step $n$. This leads us to handle each deviation event $\{\bar{b}_{t; n} > K' \alpha(n)\}$ separately (for example, it is hard to rely on Ottaviani's inequality to reduce to subsequential events). Since we are observing a phenomenon of order $\sqrt{n\log\log n}$, second moments are not informative. Fourth moments, in contrast, are relatively ill-controlled as we only have bounds $\E[\bar{b}_{t;n}^{4}] \lesssim (2^{t+m})^{2}$. We thus choose third moments as compromises.

\begin{claim}\label{claim:lilLeft}
\[
\limsup_{n} \frac{1}{\alpha(n)} \left|\sum_{2^{m+t} \le n} \bar{b}_{t; n}\right| =0 \quad \textrm{a.s.}
\]
\end{claim}

\begin{proof}
Let $K'>0$. Given $t\ge 0$ and $1 \le k \le 2^{t+m}$, $\{\bar{b}_{t;2^{t+m} (2i-1) + k}\}_{i}$ is a family of i.i.d. In this case, Proposition \ref{prop:dominantGromEach} gives a uniform constant $K_{3}'$ such that \[
\E[b_{t;2^{t+m} +k}^{3}] \le K_{3}' + K_{3}'e^{-k/K_{3}'} \cdot 2^{t+m}.
\]
By taking $K_{3} = 8K_{3}'$, we also have \[
\E|\bar{b}_{t;2^{t+m} + k}|^{3}\le \E (|b_{t;2^{t+m} + k}| + |\E b_{t;2^{t+m} + k}|)^{3} \le  K_{3}+  K_{3} e^{-k/K_{3}} \cdot2^{t+m}.
\]
Let us now define
\[
E_{t, k, i} := \left\{\w: |\bar{b}_{t;2^{t+m}(2i-1) + k}| > \frac{K' \sqrt{2^{t+m} (2i-1)}}{2^{t/8}} \right\}.
\]
Then for $Y_{t, k} = |\bar{b}_{t;2^{t+m}+k}| / (2^{3t/8 + m/2} K')$, we have
\[\begin{aligned}
\sum_{i=1}^{\infty} \Prob[E_{k, t, i}] &\le \sum_{i=1}^{\infty} i \cdot \Prob \left\{ \frac{K'\sqrt{2^{t+m} i}}{2^{t/8}} < |\bar{b}_{t;2^{t+m} + k}|\le \frac{K'\sqrt{2^{t+m} (i+1)}}{2^{t/8}} \right\}\\
&\le \int Y_{t, k}^{2} 1_{Y_{t, k} \ge 1} \, d\Prob \le \int Y_{t, k}^{3}\, d\Prob \le \frac{1}{2^{9t/8 + 3m/2} K'^{3}} \E|\bar{b}_{t;2^{t+m}+k}|^{3} \\
&\le\frac{K_{3}}{K'^{3}} 2^{-9t/8 - 3m/2} +  \frac{K_{3}}{K'^{3}} e^{-k/K_{3}} \cdot 2^{-t/8 -m/2}.
\end{aligned}
\]
We sum them up to deduce \[
\sum_{t=1}^{\infty} \sum_{0 \le k \le 2^{t}} \sum_{i=1}^{\infty} \Prob[E_{k, t, i}] <\infty.
\]
Then by Borel-Cantelli, we conclude that for almost every $\w$, 
\[
|\bar{b}_{t; n}(\w)| \le \frac{K' \alpha(n)}{2^{t/8}}
\]
for all $t$ for all but finitely many $n$. For those $\w$ we have \[
\frac{1}{\alpha(n)} \left| \sum_{t} \bar{b}_{t; n} \right| \le 16K'
\] 
eventually. We decrease $K'$ to 0 and conclude.
\end{proof}

We now finish the proof of the LIL. Fix $K'>0$ and let $m>0$ be as in Claim \ref{claim:lil1}.  Claim \ref{claim:lil1} and Claim \ref{claim:lilLeft} together yield\[
\limsup_{n} \frac{1}{\alpha(n)} \left|\sum_{2^{m+t} \le n} \left[\bar{b}_{t, 2\lfloor n/2^{m+t+1}\rfloor + 1; n} +\sum_{i = 1}^{\lfloor n/2^{m+t+1}\rfloor} \bar{b}_{t, 2i - 1} \right]\right| \le K'
\]
outside a set with probability at most $K'$. Moreover, the classical LIL implies that  \[
\limsup_{n} \left| \frac{1}{\alpha(n)} \sum_{i=1}^{\lfloor n/2^{m} \rfloor} \bar{Y}_{0, i}\right| =  \sigma_{m} \quad \textrm{a.s.}
\]
Together with the fact $\frac{1}{\alpha(n)} \bar{Y}_{0, \lfloor n/2^{m} \rfloor + 1; n} \rightarrow 0$ a.s., we conclude that \[
\limsup_{n} \left| \frac{1}{\alpha(n)} [d(o, \w_{n}o) - \E[d(o, \w_{n} o)]] \right| \in \left[ \sigma_{m} - K', \sigma_{m} + K' \right]
\]
outside a set of probability $K'>0$. Since we have  $\sigma_{m} \rightarrow \sigma$ and $K' \rightarrow 0$ as $m \rightarrow \infty$, the desired conclusion follows.
\end{proof}

\section{Discussion \& Further questions}

So far, we have adapted Gou{\"e}zel's pivotal time construction to the setting of Teichm{\"u}ller space, in addition to Gromov hyperbolic spaces, and utilized it to deduce limit laws for random walks. The crucial geometric ingredient was the construction of Schottky sets (Subsection \ref{subsection:Schottky}) using the non-positively curved feature of Teichm{\"u}ller space.

A similar phenomenon is expected on $CAT(0)$ spaces (e.g., Teichm{\"u}ller space equipped with the Weil-Petersson metric), Outer space for free groups and relatively hyperbolic groups. This line of generalization is now presented in the author's more recent preprint, \cite{choi2022random1}.

Another approach to generalize this result is to relate a group action on one space with the action on another space. For example, mapping class groups can act on both the curve complex and the Teichm{\"u}ller space. Therefore, the dynamics in one space can have implications on the dynamics in another space. This philosophy has been employed in \cite{horbez2018clt}, \cite{dahmani2018spectral} and \cite{mathieu2020deviation} and resulted in fruitful observations. The author hopes this strategy leads to the analogous limit laws on the Cayley graph of the mapping class group. 

Indeed, mapping class groups enjoy the trickiest version of hyperbolicity, as opposed to Teichm{\"u}ller spaces or Gromov hyperbolic spaces that they act on. Mathieu and Sisto overcame this difficulty by using the acylindrical action on the curve complex and established various limit laws including CLT. The author hopes that their strategy leads to other limit laws including the geodesic tracking and the converse of CLT.

%%%%%%%%%%%%%%%%%%%%%%%%%%%%%%%%%%%%%%%
%
%							References
%
%%%%%%%%%%%%%%%%%%%%%%%%%%%%%%%%%%%%%%%

\medskip
\bibliographystyle{alpha}
\bibliography{clt}

\end{document}